\documentclass[a4paper]{article}

\usepackage[utf8]{inputenc}
\usepackage[width=15cm,height=22cm]{geometry}
\usepackage{amsmath,amsfonts,amssymb,amsthm}
\usepackage[hidelinks,pdfusetitle]{hyperref}
\hypersetup{bookmarksdepth=2}
\usepackage[capitalise]{cleveref}
\usepackage[dvipsnames]{pstricks}
\usepackage{bbm}
\usepackage{stmaryrd}
\usepackage{makecell}
\usepackage{enumerate}

\crefformat{footnote}{#2\footnotemark[#1]#3}

\newtheorem{proposition}{Proposition}[section]
\newtheorem{theorem}[proposition]{Theorem}
\newtheorem{definition}[proposition]{Definition}
\newtheorem{corollary}[proposition]{Corollary}
\newtheorem{lemma}[proposition]{Lemma}
\newtheorem{remark}[proposition]{Remark}
\newtheorem{example}[proposition]{Example}
\numberwithin{equation}{section}

\newcommand*\samethanks[1][\value{footnote}]{\footnotemark[#1]}

\title{A unified approach of obstructions to small-time \\ local
controllability for scalar-input systems}

\author{Karine Beauchard\texorpdfstring{\thanks{Univ Rennes, CNRS, IRMAR - UMR 6625, F-35000 Rennes, France}}{},
Fr\'ed\'eric Marbach\texorpdfstring{\samethanks}{}}

\newcommand{\N}{\mathbb{N}}
\newcommand{\Z}{\mathbb{Z}}
\newcommand{\R}{\mathbb{R}}
\newcommand{\CC}{\mathcal{C}}

\newcommand{\dd}{\,\mathrm{d}}

\newcommand{\Bs}{\mathcal{B}^\star}
\newcommand{\bb}{\mathfrak{b}}
\newcommand{\lone}{L^1((0,t);\R)}
\newcommand{\lloc}{L^1_{\mathrm{loc}}(\R_+)}
\newcommand{\opnorm}[1]{\left|\mkern-1.5mu\left|\mkern-1.5mu\left| #1 \right|\mkern-1.5mu\right|\mkern-1.5mu\right|}
\newcommand{\eval}{\textup{\textsc{e}}}
\DeclareMathOperator{\Br}{Br}
\DeclareMathOperator{\ad}{ad}
\DeclareMathOperator{\vect}{span}
\newcommand{\supp}{\operatorname{supp}}
\newcommand{\codim}{\operatorname{codim}}
\newcommand{\intset}[1]{\llbracket #1 \rrbracket}

\newcommand{\cZ}[1]{\mathcal{Z}_{\intset{1,#1}}}
\newcommand{\ZInf}{\mathcal{Z}_{\llbracket 1, \infty\llbracket}}

\newcommand{\vocab}[1]{``\underline{\emph{#1}}''}

\newcommand{\bref}[1]{(\hyperref[#1]{#1})}

\usepackage{etoolbox}
\newcounter{proofcounter}
\newcounter{stepcounter}[proofcounter]
\newcommand{\step}[1]{%
  \ifnum\value{stepcounter}>0 \medskip\fi
  \refstepcounter{stepcounter}\noindent\emph{Step \thestepcounter: #1.}}
\newcommand{\eqstep}[1]{%
  \ifnum\value{stepcounter}>0 \medskip\fi
  \refstepcounter{stepcounter}\noindent\emph{Step \thestepcounter: #1}}
\AtBeginEnvironment{proof}{\stepcounter{proofcounter} \setcounter{stepcounter}{0}}

\begin{document}

\maketitle

\begin{abstract}
    We present a unified approach for determining and proving obstructions to small-time local controllability of scalar-input control systems.
    Our approach views obstructions to controllability as resulting from interpolation inequalities between the functionals associated with the formal Lie brackets of the system.

    Using this approach, we give compact unified proofs of all known necessary conditions, we prove a conjecture of 1986 due to Kawski, and we derive entirely new obstructions.
    Our work doubles the number of previously-known necessary conditions, all established in the 1980s.
    In particular, for the third quadratic bracket, we derive a new necessary condition which is complementary to the Agrachev--Gamkrelidze sufficient ones.
    
    We rely on a recent Magnus-type representation formula for the state, a new Hall basis of the free Lie algebra over two generators, an appropriate use of Sussmann's infinite product to compute the Magnus expansion, and Gagliardo--Nirenberg interpolation inequalities.
\end{abstract}

\setcounter{tocdepth}{1}
\tableofcontents

\section{Introduction}

\subsection{Scalar-input control-affine systems}

In this article, we consider an affine control system
\begin{equation} \label{syst}
\dot{x}(t)=f_0(x(t))+u(t) f_1(x(t)) 
\end{equation}
where the state $x(t)$ lives in $\R^d$ ($d \geq 1$), the control is a scalar input $u(t) \in \R$, $f_0$ and $f_1$ are vector fields on $\R^d$, analytic on a neighborhood of $0$, such that $f_0(0)=0$. 
These assumptions are valid for the whole article and will not be recalled in the statements.
Nevertheless, the analyticity assumption can be removed, and all our results hold assuming only finite regularity on the vector fields, as we prove in \cref{s:C^k}.

For each $t > 0$ and $u \in \lone$, there exists a unique maximal mild solution to \eqref{syst} with initial data $0$, which we will denote by $x(\cdot;u)$.
We will consider small enough controls and small enough times so that this solution is defined up to time $t$.

\subsection{Definitions of small-time local controllability}
\label{s:intro-def}

In this article, we study the small-time local controllability of system \eqref{syst} in the sense of \cref{Def:WmSTLC} below, which requires the following notions.

For $t>0$ and $m \in \N$, we consider the usual Sobolev space $W^{m,\infty}(0,t)$ equipped with the usual norm
$\|u\|_{W^{m,\infty}}:=\|u\|_{L^\infty} + \dotsb + \|u^{(m)}\|_{L^\infty}$. 
For $j \in \N$, we define by induction the iterated primitives of $u$, denoted $u_j : (0,t) \to \R$ and defined by: $u_0 := u$ and $u_{j+1}(t) = \int_0^t u_j$.
For $p \in [1,\infty]$, we let
\begin{equation} \label{def:norm_W-1}
    \|u\|_{W^{-1,p}} := \|u_1\|_{L^p}.
\end{equation}
For scalar-input systems such as \eqref{syst}, the $W^{-1,\infty}$ norm of the control is important because it is an accurate measure of the size of the state (see \cref{p:small-state} and \cite[Lemma~20]{JDE}). 

For $m\in\mathbb{Z}$, we use the notation $\llbracket m , \infty \llbracket$ for $[m,\infty) \cap \mathbb{Z}$.

\begin{definition}[$W^{m,\infty}$-STLC] \label{Def:WmSTLC}
    Let $m \in \llbracket -1 , \infty \llbracket$.
    We say that system \eqref{syst} is \emph{$W^{m,\infty}$-STLC} when, for every $t,\rho>0$, there exists $\delta=\delta(t,\rho)>0$ such that, for every $x^\star \in B(0,\delta)$, there exists $u\in W^{m,\infty}((0,t);\R) \cap \lone$ with $\|u\|_{W^{m,\infty}} \leq \rho$, such that $x(t;u)=x^\star$.
\end{definition}

Any positive answer to the STLC problem may be thought of as a nonlinear local open mapping theorem, which underlines the deepness and intricacy of this problem, when the inverse mapping theorem (or linear test, see \cite[Section 3.1]{zbMATH05150528}) cannot be used.

STLC in the literature usually corresponds to what we refer to as $L^\infty$-STLC (i.e.\ $m=0$ in \cref{Def:WmSTLC} above), where controls have to be arbitrarily small in $L^\infty$ norm (see e.g.~\cite[Definition 3.2]{zbMATH05150528} or STLC$_\varepsilon$ in \cite{zbMATH04154295}). 
Sometimes (see \cite{MR935375,MR710995}) authors investigate the $\rho$-bounded-STLC: $\rho>0$ is fixed and system \eqref{syst} is $\rho$-bounded-STLC if, for every $t>0$, there exists $\delta>0$ such that, for every $x^\star \in B(0,\delta)$, there exists $u\in L^\infty(0,t)$ with $\|u\|_{L^\infty} \leq \rho$ such that $x(t;u)=x^\star$.

For any $m\in\N^*$, $\rho > 0$ and $t \in (0,1)$, $\|u\|_{W^{-1,\infty}} \leq t \|u\|_{L^\infty} \leq \|u\|_{W^{m,\infty}}$ thus
\begin{equation} \label{Wm-STLC_implications}
    (W^{m,\infty}\text{-STLC}) \Rightarrow 
    (L^\infty\text{-STLC}) \Rightarrow 
    (\rho\text{-bounded-STLC}) \Rightarrow (W^{-1,\infty}\text{-STLC}),
\end{equation}
where any reciprocal implication is false (see \cref{s:chain}).
See also \cite{boscain2021local} for a recent comparison of various controllability definitions.
The interest of the $W^{-1,\infty}$-STLC is that it is equivalent to the small-state small-time local controllability for scalar-input systems (see \cite[Section 8.2]{JDE}). 

\bigskip

In the excellent survey \cite{zbMATH04154295}, Kawski recalls the known necessary conditions (see Theorems 3.1, 3.4 and 3.5 therein) and sufficient conditions (see Theorems 3.6 ad 3.7 therein) for $L^\infty$-STLC. 
Then he explains, on clever examples, the obstacles that a more complete theory has to overcome. 
Kawski's survey is at the root of the present article: our main results are generalizations to any systems, of its observations on particular examples which will be recalled and discussed later in the present article (see \cref{Subsec:Ex_Wk,Subsec:Ex_W2,Subsec:Ex_W3,s:D-examples}).

\subsection{Algebraic notations and Lie brackets}
\label{s:notations-algebra}

The STLC is closely related to the evaluations at $0$ of the iterated Lie brackets of the vector fields $f_0$ and $f_1$. 
We, therefore, introduce the following definitions and notations.

\bigskip

Let $X := \{X_0,X_1\}$ be a set of two \emph{non-commutative indeterminates}.

\begin{definition}[Free algebra]
    \label{def:free-algebra}
    We consider $\mathcal{A}(X)$ the \emph{free algebra} generated by $X$ over the field $\R$, i.e.\ the unital associative algebra of polynomials of the indeterminates $X_0$ and $X_1$.
\end{definition}

\begin{definition}[Free Lie algebra]
    Within $\mathcal{A}(X)$ one can define the Lie bracket of two elements as $[a,b] := ab - ba$. 
    This operation is anti-symmetric and satisfies the Jacobi identity.
    Let $\mathcal{L}(X)$ be the \emph{free Lie algebra} generated by $X$ over the field $\R$, i.e.\ the smallest linear subspace of $\mathcal{A}(X)$ containing $X$ and stable by the Lie bracket $[\cdot,\cdot]$.
\end{definition}

\begin{definition}[Iterated brackets] \label{def:BR-Eval}
    Let $\Br(X)$ be the \emph{free magma over $X$}, or, more visually, the set of \emph{iterated brackets} of elements of $X$, defined by induction: $X_0, X_1 \in \Br(X)$ and if $a, b \in \Br(X)$, then the ordered pair $(a,b)$ belongs to $\Br(X)$. 

    There is a natural \emph{evaluation} mapping $\eval$ from $\Br(X)$ to $\mathcal{L}(X)$ defined by induction by $\eval(X_i) := X_i$ for $i=0,1$ and $\eval((a, b)) := [\eval(a), \eval(b)]$.  Through this mapping, $\Br(X)$ spans $\mathcal{L}(X)$.
\end{definition}
    
\begin{definition}[Homogeneous layers within $\mathcal{L}(X)$] 
    For $b \in \Br(X)$, $n_0(b)$ (respectively $n_1(b)$) denotes the number of occurrences of the indeterminate $X_0$ (resp.\ $X_1$) in $b$.
    For $A_1, A_0 \subset \N$, $S_{A_1}(X)$ and $S_{A_1,A_0}(X)$ are the vector subspaces of $\mathcal{L}(X)$ defined by
    \begin{align}
        \label{eq:S_A1}
        S_{A_1}(X) & := \vect\{ \eval(b) ; b \in \Br(X), n_1(b) \in A_1 \}, \\
        \label{eq:S_A1A0}
        S_{A_1,A_0}(X) & := \vect\{\eval(b);b\in\Br(X), n_1(b)\in A_1, n_0(b) \in A_0\}.
    \end{align}
    For $i,j \in \N$, we write\footnote{Some authors use the notation $S_i(X)$ for what is referred to here as $S_{\intset{1,i}}(X)$.} $S_i(X)$ and $S_{i,j}(X)$ instead of $S_{\{i\}}(X)$ and $S_{\{i\},\{j\}}(X)$.
\end{definition}

\begin{definition}[Bracket integration $b0^\nu$] \label{def:0nu}
    For $b \in \Br(X)$ and $\nu \in \N$, we use the unconventional short-hand $b 0^\nu$ to denote the right-iterated bracket $(\dotsb(b, X_0), \dotsc, X_0)$, where $X_0$ appears $\nu$ times.
\end{definition}

\begin{definition}[Lie bracket of vector fields]
    For smooth vector fields $f$ and $g$, we define
    \begin{equation}
        [f,g] := (Dg) f - (Df) g.
    \end{equation}
\end{definition}

\begin{definition}[Evaluated Lie bracket]
    \label{Def:evaluated_Lie_bracket}
    Let $f_0, f_1$ be $\mathcal{C}^\infty$ vector fields on an open subset $\Omega$ of $\R^d$ and $f=\{f_0,f_1\}$.
    For $B \in \mathcal{L}(X)$, we define $f_B:=\Lambda(B)$, where $\Lambda:\mathcal{L}(X) \to \CC^\infty(\Omega;\R^d)$ is the unique homomorphism of Lie algebras such that $\Lambda(X_0)=f_0$ and $\Lambda(X_1) = f_1$.

    To simplify the notation, we will write $f_b$ instead of $f_{\eval(b)}$ when $b \in \Br(X)$. 
    The vector field $f_b$ is obtained by replacing the indeterminates $X_i$ with the corresponding vector field $f_i$ in the formal bracket $b$. 
    For instance if $b=(X_1,(X_0,X_1))$ then $f_{b}=[f_1,[f_0,f_1]]$ and if $B=\alpha_1 \eval(b_1) + \dots + \alpha_n \eval(b_n) \in \mathcal{L}(X)$ where $b_1,\dots,b_n \in \Br(X)$ and $\alpha_1,\dots,\alpha_n \in \R$ then $f_B=\alpha_1 f_{b_1}+\dots+\alpha_n f_{b_n}$.

    Eventually, for a subset $\mathcal{N}$ of $\Br(X)$ we use the notation
    \begin{equation}
        \mathcal{N}(f)(0) := \vect \{ f_b(0) ; b \in \mathcal{N} \} \subset \R^d.
    \end{equation}
\end{definition}

All the known necessary conditions for STLC are stated in the following way. 
One focuses on a ``bad'' bracket $\bb \in \Br(X)$ and one identifies a subset $\mathcal{N}$ of $\Br(X)$ containing all the brackets susceptible to neutralize $\bb$.
Then the necessary condition for STLC is $f_{\bb}(0) \in \mathcal{N}(f)(0)$.

This is linked with Krener's fundamental result \cite[Theorem 1]{zbMATH03385496}, which states that, if two control systems of the form \eqref{syst} have linearly isomorphic brackets evaluated at $0$, then they are diffeomorphic.
Thus the entire information about STLC is contained in the subset of $\R^d$ made of the evaluations at $0$ of the Lie brackets of the vector fields $f_0$ and $f_1$.

\subsection{A new basis of the free Lie algebra}

In this article, we construct a new basis of the free Lie algebra $\mathcal{L}(X)$, which is of the form $\eval(\Bs)$, where $\Bs$ is a Hall set of $\Br(X)$ (see \cref{def:Hall}). 
All our results are expressed using this basis.

The main interest of~$\Bs$ is the particular form of the associated coordinates of the second kind, which appear to be very well suited for control results and functional analysis (see \cref{s:xi-B*-123445}).
In particular, a key point is that $\Bs$ seems to allow to immediately guess from the structure of a Lie bracket and/or from its associated coordinate of the second kind if this bracket will lead to an obstruction or not.
As noted by Kawski in \cite[Section 4]{zbMATH04154295}, such a feature (``splitting'' bad and good Lie brackets) is not satisfied by the usual length-compatible or Chen--Fox--Lyndon bases of the free Lie algebra. 
We plan to investigate further this property of $\Bs$ in a forthcoming paper.

The first elements of $\Bs$ are given explicitly in the following statement. 

\begin{proposition} \label{p:B*-12345}
    The first $X_1$-homogeneous layers 
    $\Bs_k:=\{b\in\Bs;n_1(b)=k\}$
    of our basis $\Bs$ are 
    \begin{align}
        \label{Bstar_S1}
        \Bs_1 & = \{ M_\nu \}, \\
        \label{Bstar_S2}
        \Bs_2 & = \{ W_{j,\nu} \}, \\
        \label{Bstar_S3}
        \Bs_3 & = \{ P_{j,k,\nu} ;  j\leq k \}, \\
        \label{Bstar_S4}
        \Bs_4 & = \{ Q_{j,k,l,\nu} ; j \leq k \leq l \} \cup \{ Q^\sharp_{j,\mu,k,\nu}; j < k \} \cup \{Q^\flat_{j,\mu,\nu} \}, \\
        \label{Bstar_S5}
        \Bs_5 & = \{ R_{j,k,l,m,\nu} ; j\leq k \leq l \leq m \} \cup \{R^\sharp_{j,k,l,\mu,\nu}; j \leq k \},
    \end{align}
    where, implicitly, $j, k, l, m \in \N^*$, $\mu,\nu\in\N$ and we define (using the notation $0^\nu$ of \cref{def:0nu}),
    \begin{align}
        \label{def:Mj}
        & M_\nu := X_1 0^\nu, \\
        \label{def:Wjnu}
        & W_{j,\nu} := (M_{j-1},M_j) 0^\nu, \\
        \label{def:Pljnu}
        & P_{j,k,\nu} := (M_{k-1},W_{j,0}) 0^\nu, \\
        \label{def:Q}
        & Q_{j,k,l,\nu} := (M_{l-1},P_{j,k,0}) 0^\nu,
        \quad
        Q^\sharp_{j,\mu,k,\nu} := (W_{j,\mu},W_{k}) 0^\nu,
        \quad
        Q^\flat_{j,\mu,\nu} := (W_{j,\mu},W_{j,\mu+1}) 0^\nu, \\
        \label{def:R}
        & R_{j,k,l,m,\nu} :=  (M_{m-1},Q_{j,k,l,0}) 0^\nu,
        \quad
        R^\sharp_{j,k,l,\mu,\nu} := (W_{l,\mu},P_{j,k,0}) 0^\nu.
    \end{align}
    To lighten the notations, $W_j$, $P_{j,k}$ and $Q_{j,k,l}$ will denote $W_{j,0}$, $P_{j,k,0}$ and $Q_{j,k,l,0}$.

    Moreover, to avoid cluttering the formulas, all these symbols will indifferently denote either the elements of $\Br(X)$ themselves or their evaluation by $\eval$ in $\mathcal{L}(X)$ (recall \cref{def:BR-Eval}).
\end{proposition}

We only write explicitly the elements $\Bs_{\intset{1, 5}}$, because notations for the elements of $\Bs_{\llbracket 6, \infty \llbracket}$ will not be required in the sequel. 
Of course, the list could go further, albeit with increasing complexity.

One could also probably extend our construction of $\Bs$ to the case of control systems with multiple inputs. 
For such systems, one needs a basis of the free Lie algebra over $\{ X_0, X_1, \dotsc, X_q \}$.
Many such extended constructions could be proposed and the ``correct'' one might depend on the intended applications.
We discuss some key structural features of $\Bs$ (which could be preserved with multiple inputs) in \cref{rk:Bs-vs-CFL}.

\subsection{Main results: old and new necessary conditions}
\label{s:main-results}

First, we recover (slightly improved versions of) the necessary conditions for STLC, due to Sussmann \cite[Proposition 6.3]{MR710995} (for $k = 1$) and Stefani \cite[Theorem 1]{MR935375} (for $k > 1$), concerning the strongest obstruction at each even order of the control, which were historically derived for the stronger $\rho$-bounded-STLC notion (recall the implications \eqref{Wm-STLC_implications}).

\begin{theorem} \label{Thm_Stefani}
    If system \eqref{syst} is $W^{-1,\infty}$-STLC (or, equivalently, small-state-STLC), then
    \begin{equation} \label{Stefani}
    \forall k \in \N^*,\quad \ad_{f_1}^{2k}(f_0)(0) \in S_{\llbracket 1 , 2k-1\rrbracket}(f)(0).
    \end{equation}
\end{theorem}

Then we prove the following necessary conditions for controllability on the Lie brackets $W_{k}$ for $k \in \N^*$ (see \eqref{def:Wjnu}), which we call \vocab{quadratic Lie brackets}, as they involve $X_1$ twice.

\begin{theorem} \label{Thm:Kawski_Wm}
    Let $m \in \llbracket -1 , \infty \llbracket$.
    If system \eqref{syst} is $W^{m,\infty}$-STLC, then
    \begin{equation} \label{Kawski_Wm_conj}
        \forall k \in \N^*, \quad f_{W_k}(0)\in S_{\llbracket 1 , \pi(k,m) \rrbracket \setminus\{2\}}(f)(0)
    \end{equation}
    where
    \begin{equation} \label{eq:pikm}
        \pi(k,m) := 1 + \left\lceil \frac{2k-2}{m+1} \right\rceil
    \end{equation}
    with the convention $\pi(k,-1)=+\infty$ and $\pi(1,-1) = 1$.
\end{theorem}

As particular cases, this result contains necessary conditions on $W_k$ for
\begin{itemize}
    \item $W^{-1,\infty}$-STLC (small-state STLC), which is new and requires particular care (see \cref{s:m=-1}),
    \item $L^\infty$-STLC: $f_{W_k}(0)\in S_{\llbracket 1, 2k-1 \rrbracket \setminus\{2\}}(f)(0)$, which was conjectured in 1986 in \cite[p.\ 63]{MR2635388},
    \item $W^{m,\infty}$-STLC with $1 \leq m \leq 2k-4$, which is a new result,
    \item $W^{2k-3,\infty}$-STLC: $f_{W_k}(0)\in S_{1}(f)(0)$, which we had already proved in \cite[Theorem~3]{JDE}.
\end{itemize}

An interest of condition \eqref{Kawski_Wm_conj} is that it illustrates that some kind of compensation on $f_{W_k}(0)$ is necessary for controllability.
We say that this condition is \vocab{loose} because we only focused on obtaining the optimal threshold $\pi(k,m)$ (see \cref{Subsec_Optim_Wk}), but, within $S_{\intset{3,\pi(k,m)}}(X)$, we did not try to obtain the minimal list of brackets.
Depending on one's needs, our general approach can be used to shrink this list.
As an illustration, one has the following result.

\begin{theorem} \label{thm:S2-S3-intro}
    Let $m \in \llbracket -1 , \infty \llbracket$.
    If system \eqref{syst} is $W^{m,\infty}$-STLC, then, for all $k \in \N^*$ such that $\pi(k,m) \geq 3$ (defined in \eqref{eq:pikm}),
    \begin{equation} \label{eq:S2-S3-comp}
        \quad f_{W_k}(0) \in
        S_1(f)(0) + \mathcal{P}_k(f)(0) + 
        S_{\intset{4,\pi(k,m)}}(f)(0)
    \end{equation}
    where
    \begin{equation} \label{eq:P_k}
        \mathcal{P}_k := \{ P_{j,l,\nu} \in \Bs_3 ; j < k \}
        \subsetneq \Bs_3.
    \end{equation}
\end{theorem}

For $k \in \{2,3\}$, a careful analysis allows to refine even more the necessary conditions of \cref{Thm:Kawski_Wm,thm:S2-S3-intro}. 
In particular, in the case $m=0$, we prove the following results.

\begin{theorem} \label{Thm:CN_W123}
    If system \eqref{syst} is $L^\infty$-STLC, then $f_{W_k}(0) \in \mathcal{N}_k(f)(0)$ for $k=1,2,3$, where
    \begin{align}
        \mathcal{N}_1 & := \Bs_1, \\
        \label{def:mathcalE2}
        \mathcal{N}_2 & := \mathcal{N}_1 \cup \{  P_{1,1,\nu} ; \nu \in \N \}, \\
        \label{def:mathcalE3}
        \mathcal{N}_3 & := \mathcal{N}_2 \cup \{ 
        P_{1,l,\nu}, 
        Q_{1,1,1}, 
        Q_{1,1,2,\nu},  
        Q^\flat_{1,0}, Q^\flat_{1,1}, Q^\flat_{1,2},
        R_{1,1,1,1,\nu}, 
        R^\sharp_{1,1,1,\mu,\nu} ; l \in \N^*, \mu,\nu \in \N \}.
    \end{align}
\end{theorem}

The statement concerning $W_2$ is proved by Kawski in \cite[Theorem~1]{zbMATH04031488}, using the Chen--Fliess expansion and technical results from Stefani \cite{MR935375}. 
We propose a different strategy, that allows us to obtain similarly the condition concerning $W_3$, which is new.
Moreover, the lists $\mathcal{N}_k$ are \vocab{minimal} in the sense that, for any bracket in $\mathcal{N}_k$, we exhibit a system where it restores STLC when in competition with $W_k$.
In the hardest case $k=3$, we prove these controllability results using the Agrachev--Gamkrelidze sufficient condition of \cite[Theorem~4]{AG93}, illustrating that, the necessary condition $f_{W_3}(0) \in \mathcal{N}_3(f)(0)$ is somehow complementary to their sufficient conditions theory.

\bigskip

To go beyond necessary conditions involving quadratic Lie brackets, we prove the following new necessary condition linked with a bracket of the sixth order with respect to the control.

\begin{theorem} \label{thm:sextic}
    If system \eqref{syst} is $L^\infty$-STLC, then
    \begin{equation}
        f_{\ad^2_{P_{1,1}}(X_0)}(0) \in 
        \vect \left\{ f_b(0) ;  b \in \Bs_{\intset{1,7}} , b \neq \ad^2_{P_{1,1}}(X_0) \right\}.
    \end{equation}
\end{theorem}

\bigskip

Throughout the paper, we will discuss the optimality of all these necessary conditions by comparing them with the known sufficient conditions, including the ones due to Agrachev and Gamkrelidze \cite[Theorem~4]{AG93} or Krastanov \cite[Theorem 2.7]{K09}.
In particular, let us already point out that, due to their structure, all the brackets involved in the above-mentioned obstructions, namely $\ad^{2k}_{X_1}(X_0)$, $W_k = \ad^{2}_{M_{k-1}}(X_0)$ and $\ad^2_{P_{1,1}}(X_0)$ are indeed always seen as ``bad'' and required to be compensated by such sufficient conditions (see \cref{s:bad-structure}).

\bigskip

Eventually, we explain in \cref{s:C^k} why all these results, derived for analytic vector fields, remain valid without change for $\CC^\infty$ vector fields.
More precisely, we show that assuming only finite regularity on $f_0$ and $f_1$ is sufficient to preserve the conclusions, provided that one gives the appropriate meaning to the evaluations at $0$ of the considered brackets (the brackets themselves being undefined elsewhere).

\subsection{Our unified approach of obstructions}

We provide a general overview of the approach that we use in this paper to conjecture and prove necessary conditions of STLC.
We claim that the method is fairly general: it already allowed us to recover all known or conjectured obstructions, and prove multiple new ones.

\subsubsection{An interpretation of obstructions to STLC as drifts}
\label{s:drifts}

Our results are of the form: $W^{m,\infty}\text{-STLC} \Rightarrow f_{\bb}(0) \in \mathcal{N}(f)(0)$, where $m \in \llbracket -1, \infty \llbracket$, $\bb \in \Bs$ and $\mathcal{N}$ is a subset of $\Bs$.
We prove these results by contraposition, starting from the assumption
\begin{equation} \label{Heuristic:Hyp_non_STLC}
    f_{\bb}(0) \notin \mathcal{N}(f)(0).
\end{equation}
Our strategy consists in proving that, when \eqref{Heuristic:Hyp_non_STLC} holds, the state $x(t;u)$ \vocab{drifts} along $f_{\bb}(0)$, in the sense of \cref{Def:drift} below, which requires the following notion. 

\begin{definition}[Component]
    Let $N$ be a vector subspace of $\R^d$ and $e \in \R^d \setminus N$.
    We say that a linear form $\mathbb{P}:\R^d\to \R$ is \emph{a component along $e$ parallel to~$N$} when $\mathbb{P} e = 1$ and $N \subset \ker \mathbb{P}$.
\end{definition}

\begin{definition}[Drift] \label{Def:drift}
    Let $\bb \in \Bs$, $\mathcal{N} \subset \Br(X)$ and $m \in \llbracket -1 , \infty \llbracket$.
    We say that system~\eqref{syst} has a \emph{drift along $f_{\bb}(0)$, parallel to $\mathcal{N}(f)(0)$, as $(t,\|u\|_{W^{m,\infty}}) \to 0$} when there exists $C>0$ and $\beta>1$ such that, for every $\varepsilon>0$, there exists $\rho>0$ such that, for every $t \in (0,\rho)$ and every $u \in W^{m,\infty}((0,t);\R)\cap L^1((0,t);\R)$ with $\|u\|_{W^{m,\infty}} \leq \rho$,
    \begin{equation} \label{eq:def-drift}
        \mathbb{P} x(t;u) \geq (1-\varepsilon) \xi_{\bb}(t,u) - C |x(t;u)|^\beta,
    \end{equation}
    where $\mathbb{P}$ gives a component along $f_{\bb}(0)$ parallel to $\mathcal{N}(f)(0)$ and $(\xi_{b})_{b\in\Bs}$ are the coordinates of the second kind associated with $\Bs$ (see \cref{Def:Coord2} and \cref{Prop:Coord_Bstar}).
\end{definition}

\begin{lemma} \label{lem:drift-stlc}
    Let $\bb \in \Bs$, $\mathcal{N} \subset \Br(X)$ and $m \in \llbracket -1 , \infty \llbracket$.
    Assume that $\xi_\bb(t,u) \geq 0$ for all $u \in L^1((0,t);\R)$ and that system \eqref{syst} has a drift along $f_{\bb}(0)$, parallel to $\mathcal{N}(f)(0)$, as $(t,\|u\|_{W^{m,\infty}}) \to 0$.
    Then system \eqref{syst} is not $W^{m,\infty}$-STLC.
\end{lemma}

\begin{proof}
    For small enough times and controls, when $\xi_\bb \geq 0$, estimate \eqref{eq:def-drift} prevents $x(t;u)$ from reaching targets of the form $x^\star=-a f_{\bb}(0)$ with $0<a \ll 1$, because this would entail $-a = \mathbb{P} x^\star \geq - C |x^\star|^\beta = - C |f_{\bb}(0)|^\beta a^{\beta}$, which fails for $a$ small enough, since $\beta > 1$. 
    Thus, estimate~\eqref{eq:def-drift} comes into contradiction with $W^{m,\infty}$-STLC. 
    Indeed, \cref{Def:WmSTLC} requires that, even for arbitrarily small times and controls, one may reach a whole neighborhood of $0$.
\end{proof}

\begin{remark} \label{rk:beta>1}
    To deny STLC, it is sufficient (as is done in the previous proof), to negate the possibility of reaching locally a half line $\R_-^* f_\bb(0)$.
    Nevertheless, since $\beta>1$, when $\xi_\bb \geq 0$, estimate~\eqref{eq:def-drift} actually also implies that the unreachable set contains locally a whole half-space.
    
    This property is quite satisfactory as it somehow complements the fact that most known sufficient conditions for STLC yield a locally convex reachable set.
    More precisely, these conditions rely on ``variations'' or ``tangent vector'', and it is known that the set of such tangent vectors is almost convex \cite[Lemma 2.3]{zbMATH04154295}.

    This motivates our definition of drifts, which therefore entails not only a lack of STLC but also a description of the unreachable space.
    A weaker definition, such as \cref{Def:Weak-drift}, replacing $C|x|^\beta$ by $\varepsilon|x|$ would not yield such a precise conclusion.
\end{remark}

\begin{example} \label{ex:easy}
    To illustrate the definitions, consider the system
    \begin{equation}
        \begin{cases}
            \dot{x}_1 = u \\
            \dot{x}_2 = x_1 \\
            \dot{x}_3 = x_1^2 - x_2^2 - x_1^3 - 4 x_1 x_2.
        \end{cases}
    \end{equation}
    Written in the form \eqref{syst}, this system satisfies 
    \begin{equation}
        f_{X_1}(0) = e_1, \quad f_{M_1}(0) = e_2, \quad 
        f_{W_1}(0) = 2e_3, \quad f_{W_2}(0) = - 2 e_3, \quad f_{P_{1,1}}(0) = - 6 e_3
    \end{equation}
    and $f_b(0)=0$ for any other $b \in \Bs$.
    Therefore, it does not satisfy Sussmann's necessary condition (case $k = 1$ of \eqref{Stefani}) which requires that $f_{W_1}(0) \in S_1(f)(0)$ .
    
    Using the notations $u_1$ and $u_2$ of \cref{s:intro-def} for the first and second primitive of $u$, explicit integrations lead to $x_1(t) = u_1(t)$, $x_2(t) = u_2(t)$ and
    \begin{equation}
        x_3(t) = \int_0^t u_1^2 - \int_0^t u_2^2 - \int_0^t u_1^3 - 2 x_2^2(t).
    \end{equation}
    Here, $\mathbb{P} : x \mapsto \frac{1}{2} x_3$ is a component along $f_{W_1}(0)$, parallel to $\R e_1 + \R e_2 = S_1(f)(0)$.

    Moreover, since $\xi_{W_1}(t,u) = \frac 12 \int_0^t u_1^2$ (see \eqref{xi_Wjnu}), one has, by Poincaré's inequality
    \begin{equation}
        \mathbb{P} x(t) \geq \left(1- t^2 - \|u_1\|_{L^\infty}\right) \xi_{W_1}(t,u) - x_2^2(t).
    \end{equation}
    Therefore, this estimate indeed matches \eqref{eq:def-drift} provided that both the time and the control are small enough.
    One also checks that the $|x(t;u)|^\beta$ term is required, reflecting the fact that changes of coordinates in the state-space can locally bend the unreachable set. 
\end{example}

In \cref{ex:easy}, proving the presence of the drift is easy as the system is explicitly integrable.
To prove our results in a general setting, we will rely on the following more robust approach.

\subsubsection{A recent approximate representation formula for the state}
\label{s:heuristic-magnus}

Unlike historical proofs which all rely on Chen's expansion (see \cref{rk:Chen}), the starting point of our strategy is a recent approximate representation formula, introduced in~\cite{P1} and explained more precisely in \cref{s:Magnus-ODE}, which states that, as $(t,\|u\|_{W^{-1,\infty}})\to 0$,
\begin{equation} \label{repform1}
    x(t;u)=\cZ{M}(t,f,u)(0)+O\left( \|u\|_{W^{-1,M+1}}^{M+1} + |x(t;u)|^{1+\frac{1}{M}} \right),
\end{equation}
where, for $M\in\N^*$, $\cZ{M}(t,f,u)$ is an analytic vector field belonging to $S_{\intset{1,M}}(f)$ and given by
\begin{equation} \label{repform2}
    \cZ{M}(t,f,u) = \sum_{b \in \Bs_{\intset{1,M}}} \eta_b(t,u) f_b,
\end{equation}
where the $\eta_b$ are functionals of $t$ and $u$, which we call \vocab{coordinates of the pseudo-first kind} (see \cref{rk:pseudo}) and the infinite sum converges (in the sense of analytic functions).

A remarkable feature of \eqref{repform1}-\eqref{repform2} is that, when computing the state $x(t;u)$ as almost $\cZ{M}(t,f,u)(0)$, each term $\eta_b(t,u) f_b(0)$ of the series decouples:
\begin{itemize}
    \item on one side, a scalar $\eta_b(t,u) \in \R$, which carries the time and control dependency, but in a universal (i.e.\ system-independent) way,
    \item on the other side, a vector $f_b(0) \in \R^d$, which encodes the algebraic and geometric dependency on the system, in a coordinate-independent way, as only Lie brackets of $f_0$, $f_1$ are involved. 
\end{itemize}

A technical drawback of \eqref{repform2} is that, unlike the coordinates of the second kind $\xi_b$, associated with Sussmann's infinite product expansion (see \cref{s:prod-inf}), the functionals $\eta_b$ are not given by nice explicit expressions.
Our insight to deal with this difficulty is to rely on the heuristic that, somehow, $\eta_b \approx \xi_b$ (see \cref{p:etab-xib-XI} for a statement, and \cref{s:pollutions} for limits to this belief).
Since both sets of coordinates are linked by iterated applications of the usual Campbell--Baker--Hausdorff formula, for each $b \in \Bs$, the difference $\eta_b-\xi_b$ is given by a sum of products of the form $\xi_{b_1}(t,u) \dotsb \xi_{b_q}(t,u)$ (for some $q \geq 2$ and other lower-order $b_i$'s) which we call \vocab{cross terms}.

Hence, we work with the formula
\begin{equation}
    \cZ{M}(t,f,u)(0) = \sum_{b \in \Bs_{\intset{1,M}}} \xi_b(t,u) f_b(0) + \underbrace{\sum_{b \in \Bs_{\intset{1,M}}} (\eta_b(t,u)-\xi_b(t,u)) f_b(0)}_{\text{cross terms}}.
\end{equation}

\subsubsection{An heuristic to conjecture drifts}
\label{s:heuristic-conj-drift}

Our formula can be used to conjecture necessary conditions for STLC in the following way.
\begin{enumerate}
    \item 
    One starts by considering a bracket $\bb \in \Bs$ for which $\xi_{\bb}(t,\cdot)$ is positive-definite (i.e.\ for every $t>0$ and $u \in \lone \setminus\{0\}$, $\xi_{\bb}(t,u) > 0$).
    The identification of such candidate ``bad'' brackets is particularly easy within $\Bs$ (see \cref{s:xi-B*-123445}) and part of the reasons for which we believe that this basis is well adapted to control theory.

    \item 
    One also fixes a regularity index $m \in \llbracket -1, \infty \llbracket$ (one can think $m = 0$ if one is mostly interested in the usual notion of $L^\infty$-STLC).

    \item \label{it:strat-3}
    One then determines $M \in \N^*$ large enough such that the main remainder of \eqref{repform1} will satisfy
    \begin{equation} \label{eq:coerc-interp}
        \|u\|_{W^{-1,M+1}}^{M+1} \lesssim \|u\|_{W^{m,\infty}}^{M+1-n_1(b)} \xi_\bb(t,u).
    \end{equation}
    Heuristically, this is an interpolation inequality, bounding the $W^{-1,M+1}$ norm between the stronger norm $\|u\|_{W^{m,\infty}}$ and the term $\xi_\bb(t,u)$ which plays the role of the weaker norm.
    Choosing $M$ larger makes \eqref{eq:coerc-interp} easier to prove as it requires even less coercivity from $\xi_\bb$.
    See \eqref{eq:quad-interpol} or \eqref{eq:D-interpol-u1-L8} for examples of such interpolation inequalities.

    \item \label{it:strat-4}
    One then determines $\mathcal{N}$ as the set of $b \in \Bs_{\intset{1,M}} \setminus \{ \bb \}$ such that $\xi_b \neq o(\xi_\bb)$ as $(t,\|u\|_{W^{m,\infty}}) \rightarrow 0$.
    $\mathcal{N}$ can be interpreted as the set of \vocab{neutralizing} brackets whose coordinate would be strong enough (from a functional analysis point of view) to counterbalance the coercivity of $\xi_\bb$ and could lead to restoring STLC (see e.g.\ \cref{Subsec:Ex_W3} for detailed examples in the case $\bb = W_3$).
\end{enumerate}
Eventually, the heuristic is that $f_\bb(0) \in \mathcal{N}(f)(0)$ will be a necessary condition for $W^{m,\infty}$-STLC.
Indeed, by contraposition, if one assumes \eqref{Heuristic:Hyp_non_STLC}, this allows to consider $\mathbb{P}:\R^d\to \R$, a component along $f_{\bb}(0)$ parallel to $\mathcal{N}(f)(0)$. 
Then, using \cref{s:heuristic-magnus},
\begin{equation} \label{Heuristic:rep_form}
    \begin{split}
        \mathbb{P} x(t;u) = \xi_\bb(t,u) & + O \big( 
        \sum_{\substack{b \in \Bs_{\intset{1,M}} \\ b \notin \mathcal{N} \cup \{ \bb \}}} |\xi_b(t,u) \mathbb{P}f_b(0)| \big)
        + O \big( \sum_{\substack{b \in \Bs_{\intset{1,M}} \\ b \notin \mathcal{N}}} |(\eta_b-\xi_b)(t,u) \mathbb{P}f_b(0)| \big)
        \\
        & + O \big( \|u\|_{W^{-1,M+1}}^{M+1}\big) + O \big( |x(t;u)|^{1+\frac{1}{M}}
        \big)
    \end{split}
\end{equation}
In this formula, within the big $O$ remainders, the intuition is that
\begin{itemize}
    \item the first term should be bounded by $\varepsilon \xi_\bb(t,u)$ thanks to the choice of $\mathcal{N}$ in \cref{it:strat-4} above (such sums are estimated in \cref{p:sum-xi-XI}),
    \item the second term should be negligible if the intuition $\eta_b \approx \xi_b$ is valid (such sums are estimated in \cref{p:sum-cross-XI}),
    \item the third term is bounded by $\varepsilon \xi_\bb(t,u)$ thanks to the choice of $M$ in \cref{it:strat-3} above and \eqref{eq:coerc-interp},
    \item the fourth term is part of the definition of a drift in \eqref{eq:def-drift},
\end{itemize}
therefore establishing the presence of a drift as in \cref{Def:drift}.

\subsubsection{Arguments used in the proofs}
\label{s:method-args}

From a technical point of view, the most painful task is to estimate the cross terms involved in~\eqref{Heuristic:rep_form}, i.e.\ in the differences $|\eta_b-\xi_b|$ for $b \in \Bs_{\intset{1,M}} \setminus \mathcal{N}$ (including $|\eta_\bb - \xi_\bb|$).
To obtain these estimates, the proofs of this paper share a common structure and involve (to various degrees of complexity) the following ingredients of different natures.

\paragraph{Geometric arguments.}
To bound the cross terms, we prove that assumption \eqref{Heuristic:Hyp_non_STLC} implies what we call \vocab{vectorial relations} involving other elements $f_b(0)$ for $b \in \Bs$.
Then, we prove that these vectorial relations entail what we call \vocab{closed-loop estimates}, i.e.\ that some coordinates $\xi_{b_i}(t,u)$ for some particular $b_i \in \Bs$ involved in the cross terms can be estimated from $|x(t;u)|$ and higher-order terms involving the control.
This is a key argument of the method.
The essence of closed-loop estimates can be spotted implicitly in previous literature: for example \cite[Lemma 3.2]{MR935375} is a (slightly less general) version of \cref{lem:stefani-loop}, while \cite[p.\ 148]{zbMATH04031488} uses a static-state feedback to guarantee that $u_1(t) = 0$.

\paragraph{Analysis arguments.}
Almost all estimates involve interpolation inequalities.
As this paper mostly concerns quadratic obstructions to controllability, most of the proofs rely on the usual Gagliardo--Nirenberg interpolation inequalities of \cref{s:gagliardo-nirenberg}.
Nevertheless, in the sextic case, we need the new unusual interpolation inequalities of \cref{s:sextic-interpolation}.
When working on other obstructions, it is likely that many new interpolation inequalities will be required, such as the ones we derived (for this purpose) in \cite{FM-GNS}.

\paragraph{Algebraic arguments.}
The previous arguments might be insufficient to bound some cross terms $\xi_{b_1}(t,u) \dotsb \xi_{b_q}(t,u)$.
In the Baker--Campbell--Hausdorff formula, these products are coefficients of Lie brackets such as $[[b_1, b_4], [b_2, b_3]]$ involving each $b_i$ exactly once.
Luckily, in such cases, we are able to prove that the decomposition in $\Bs$ of these Lie brackets is contained in $\mathcal{N}$, so that these cross terms are \emph{not} involved in \eqref{Heuristic:rep_form}.
Such arguments are purely algebraic computations in $\Bs$, and totally independent from any system or functional analysis.
They are for example of paramount importance in \cref{s:S2-S3} or
\cref{s:sextic-algebra}.

\subsubsection{The low-regularity case $m=-1$}
\label{s:method-1}

For $m=-1$ and some choices of $\bb$, estimate \eqref{eq:coerc-interp} may fail (even for arbitrarily large $M$) and then the remainder $\|u\|_{W^{-1,M+1}}^{M+1}$ in the representation formula \eqref{repform1} cannot be absorbed.
We developed an extension of our method which encompasses this difficult low-regularity case (see \cref{s:m=-1}), and relies on \vocab{embedded semi-nilpotent systems}.

\subsection{On the invariance by change of coordinates and feedbacks}

As can be seen by inspection of \cref{Def:WmSTLC} (see also \cite[Lemma 16]{JDE} for the case $m > 0$), two important classes of transformations preserve small-time local controllability. 

First, analytic changes of coordinates in the state space.
As recalled in \cref{s:notations-algebra}, stating conditions involving only Lie brackets evaluated at $0$ automatically guarantees that our conditions are coordinate-invariant, since the Lie brackets evaluated at $0$ in the new coordinates are linearly isomorphic to the ones in the old coordinates (see \cite[Theorem~1]{zbMATH03385496} or also \cite[Section 8.2]{P1}).

Second, static-state feedbacks (so particular classes of changes of coordinates for the couple $(x,u)$ which preserve the control-affine nature of the system), or even changes of time-scale (see~\cite[Weak feedback equivalence]{zbMATH04172849}). 
As criticized by Lewis in \cite{lewis2012fundamental}, methods based on the identification of vector fields with a free Lie algebra generally don't embed the invariance under such feedback transformations.
To preserve such an invariance, other approaches are necessary, such as Agrachev and Gamkrelidze's ``control of diffeomorphisms'' \cite{AG93,AGdiffeo} or Lewis' ``tautological systems'' \cite{lewis2014tautological,lewis2016linearisation}.

While some of the above necessary conditions are definitely not feedback-invariant (and we plan to study this difficult and very interesting matter further in a forthcoming work), they still provide computationally checkable necessary conditions, and provide a nice counterpoint to sufficient conditions such as \cite[Theorem~4]{AG93} or \cite[Theorem 2.7]{K09}.
Moreover, some of them are feedback-invariant. 
Indeed, \cite[Theorem 1]{brockett1978feedback} implies that \cref{Thm_Stefani} and \cref{Thm:CN_W123} for $k = 1,2$ are feedback-invariant. 

\subsection{Structure of the article}

This paper is organized in three parts:
\begin{itemize}
    \item First, we introduce the tools required to use our unified approach.
    \begin{itemize}
        \item In \cref{s:formal}, we recall the notion of formal differential equations set in the algebra of formal series over $X$, which allows to model \eqref{syst} independently on $f_0$ and $f_1$.
        \item In \cref{s:B*}, we introduce a new Hall set $\Bs$ over $\{ X_0, X_1 \}$ which yields a Hall basis of $\mathcal{L}(X)$ particularly well adapted to control problems.
        \item In \cref{s:tools}, we explain how the formal results of \cref{s:formal} translate to system \eqref{syst} driven by analytic vector fields.
        We formulate useful black-box estimates.
    \end{itemize}
    \item Then, we implement the method to prove all results of \cref{s:main-results}.
    \begin{itemize}
        \item In \cref{s:Suss-Stef}, we prove \cref{Thm_Stefani}.
        \item In \cref{s:Quad}, we prove \cref{Thm:Kawski_Wm,thm:S2-S3-intro} for $m \geq 0$.
        \item In \cref{s:W2refined}, we prove \cref{Thm:CN_W123} for $W_2$.
        \item In \cref{s:W3refined}, we prove \cref{Thm:CN_W123} for $W_3$.
        \item In \cref{s:sextic}, we prove \cref{thm:sextic}.
    \end{itemize}
    \item Eventually, we describe two extensions of our method.
    \begin{itemize}
        \item In \cref{s:m=-1}, we study the notion of embedded semi-nilpotent systems, which we use to prove \cref{Thm:Kawski_Wm,thm:S2-S3-intro} in the low-regularity case $m = -1$.
        \item In \cref{s:C^k}, we remove the analyticity assumption used throughout the paper.
    \end{itemize}
\end{itemize}

\section{Tools from formal power series}
\label{s:formal}

In \cref{s:formal-DE}, we introduce the formal differential equation \eqref{eq:formal-DE} whose solution $x(t)$, is a formal power series.
In \cref{s:Hall}, we recall the well-known notions of Hall sets, which yield bases of~$\mathcal{L}(X)$, within which one can express the solutions to \eqref{eq:formal-DE}.
In \cref{s:prod-inf,s:Magnus,s:Magnus-Interaction,sec:eta}, we present formulas which allow to compute these solutions within such bases.


\subsection{The formal differential equation} 
\label{s:formal-DE}

Fundamental in this project is the use of the formal differential equation
\begin{equation} \label{eq:formal-DE}
    \begin{cases}
        \dot{x}(t) = x(t) (X_0+u(t)X_1), \\
        x(0) = 1.
    \end{cases}
\end{equation}
Although this equation is linear, a classical linearization principle (see \cite[Section 4.1]{P1}) allows to recast the study of nonlinear ODEs such as \eqref{syst} driven by vector fields to this setting.
A key benefit of this abstract formulation is that it is now independent on $f_0$ and $f_1$.

The goal of this section is to define the solutions to \eqref{eq:formal-DE}. 
This requires the following notions.

\begin{definition}[Graded algebra] \label{def:free.algebra}
    The \emph{free associative algebra} $\mathcal{A}(X)$ (see \cref{def:free-algebra}) can be seen as a graded algebra:
    \begin{equation}
        \mathcal{A}(X) = \underset{n \in \N}{\oplus} \mathcal{A}_n(X),
    \end{equation}
    where $\mathcal{A}_n(X)$ is the finite-dimensional $\R$-vector space spanned by monomials of degree $n$ over $X$. 
    In particular $\mathcal{A}_0(X) = \R$ and $\mathcal{A}_1(X) = \vect_{\R}(X)$.
\end{definition}

\begin{definition}[Formal series]\label{-def:formalseries}
    We consider the (unital associative) algebra $\widehat{\mathcal{A}}(X)$ of formal series generated by $\mathcal{A}(X)$.  An element $a \in \widehat{\mathcal{A}}(X)$ is a sequence $a = (a_n)_{n\in\N}$ written $a = \sum_{n \in \N} a_n$, where $a_n \in \mathcal{A}_n(X)$ with, in particular, $a_0 \in \R$ being its constant term.
    
    We also define the Lie algebra of formal Lie series $\widehat{\mathcal{L}}(X)$ as the Lie algebra of formal power series $a \in \widehat{\mathcal{A}}(X)$ for which $a_n \in \mathcal{L}(X)$ for each $n \in \N$.
\end{definition}

Within the realm of formal series, one can define the operators $\exp$ and $\log$.
For instance, for $a \in \widehat{\mathcal{A}}(X)$ with $a_0=0$, $\exp(a) := \sum_{k=0}^\infty \frac{a^k}{k!}$ is a well-defined formal series.

\bigskip


The solutions to \eqref{eq:formal-DE} are defined in the following way.

\begin{definition}[Solution to \eqref{eq:formal-DE}]
    \label{def:formalODE}
    Let $t > 0$ and $u \in \lone$.
    The solution to the formal differential equation \eqref{eq:formal-DE} is the formal-series valued function $x: [0,t] \to \widehat{\mathcal{A}}(X)$, whose homogeneous components $x_n : \R_+ \to \mathcal{A}_n(X)$ are the unique continuous functions that satisfy, for every $s \geq 0$, $x_0(s) = 1$ and, for every $n \in \N^*$,
    \begin{equation} \label{eq:xn.xn1}
        x_{n}(s) = \int_0^s x_{n-1}(s') (X_0+u(s')X_1) \dd s'.
    \end{equation}
\end{definition}

Iterating the integral formula \eqref{eq:xn.xn1} yields the following power series expansion, known as the \emph{Chen series} (introduced in \cite{MR0073174,zbMATH03126609}), which is the most direct way to compute the solution to \eqref{eq:formal-DE}:
\begin{equation} \label{eq:Chen}
    x(t) = \sum_{\omega \in X^*} \left(\int_0^t u_\omega \right) \omega,
\end{equation}
where the sum ranges over all elements $\omega$ of $X^*$, the free monoid over $X$ (i.e.\ non-commutative monomials built as products of $X_0$ and $X_1$), and $\int_0^t u_\omega$ is a notation for a coefficient which can be computed recursively (see \cite[Section~2.2]{P1} for a gentle introduction).

\begin{remark} \label{rk:Chen}
    Although this expansion is the one which was used to prove all known necessary conditions (in \cite{zbMATH04031488,MR935375,MR710995}), we will not use it in this paper.
    Indeed, although it enjoys nice convergence properties when substituting $X_0$ and $X_1$ with analytic vector fields $f_0$ and $f_1$ (see \cite[Section 5]{P1}), it is not expressed in the Lie algebra $\widehat{\mathcal{L}}(X)$ but in the whole algebra $\widehat{\mathcal{A}}(X)$.
    This makes it, in our opinion, difficult to use to conjecture and prove more complex obstructions such as our new results stated in \cref{s:main-results}.
\end{remark}

\subsection{Hall sets and bases} \label{s:Hall}

We recall the notion of Hall sets and Hall bases.
For more details on theses bases of $\mathcal{L}(X)$, we refer to \cite{CasselmanFree}, \cite[Chapter 4]{zbMATH00417855} or \cite[Chapter 1]{MR516004}.

\begin{definition}[Length, left and right factors]
    For $b \in \Br(X)$, $|b|$ denotes the length of $b$.
    If $|b| > 1$, $b$ can be written in a unique way as $b = (b_1, b_2)$, with $b_1, b_2 \in \Br(X)$. 
    We use the notations $\lambda(b) = b_1$ and $\mu(b) = b_2$, which define maps $\lambda,\mu: \Br(X)\setminus X \to \Br(X)$.
\end{definition}

\begin{definition}[Hall set] \label{def:Hall}
 A \emph{Hall set} is a subset $\mathcal{B}$ of $\Br(X)$, totally ordered by a relation $<$ and such that
\begin{itemize}
    \item $X \subset \mathcal{B}$,
    \item for $b = ( b_1, b_2 ) \in \Br(X)$, $b \in \mathcal{B}$ iff $b_1, b_2 \in \mathcal{B}$, $b_1 < b_2$ and either $b_2 \in X$ or $\lambda(b_2) \leq b_1$, 
    \item for every $b_1, b_2 \in \mathcal{B}$ such that $(b_1,b_2) \in \mathcal{B}$, one has $b_1 < (b_1,b_2)$.
\end{itemize}
\end{definition}

The main interest of Hall sets is that their images by $\eval$ (recall \cref{def:BR-Eval}) yield algebraic bases of $\mathcal{L}(X)$, called Hall bases, as proved in \cite[Corollary 1.1, Proposition 1.1 and Theorem 1.1]{MR516004}. 

\begin{theorem}[Viennot]
    \label{thm:viennot}
    Let $\mathcal{B} \subset \Br(X)$ be a Hall set. 
    Then $\eval(\mathcal{B})$ is a basis of $\mathcal{L}(X)$.
\end{theorem}

\begin{remark}
    Historically, Hall sets were introduced by Marshall Hall in \cite{zbMATH03059664}, based on ideas of Philip Hall in \cite{zbMATH03010343}. 
    In his historical narrower definition, the third condition in \cref{def:Hall} was replaced by the stronger condition: for every $b_1, b_2 \in \mathcal{B}$, $b_1 < b_2 \Rightarrow |b_1| \leq |b_2|$.
    
    Two famous families of Hall sets are the Chen--Fox--Lyndon ones (see \cite[p.~15-16]{MR516004}) whose order stems from the lexicographic order on words and the historical length-compatible Hall sets, for which $b_1<b_2 \Rightarrow|b_1|\leq|b_2|$. 
    Other examples, such as the Spitzer--Foata basis are studied in \cite{A1} and \cite[Chapter~1]{MR516004}.

    We define in \cref{s:B*} below our new Hall set $\Bs$, which combines important features of the Chen--Fox--Lyndon and length-compatible ones (see \cref{rk:Bs-vs-CFL}) as well as new ones.
\end{remark}

\begin{definition}[Support] \label{def:support}
    Let $\mathcal{B}$ be a Hall set of $\Br(X)$ and $a \in \mathcal{L}(X)$. For $b \in \mathcal{B}$, we denote by $\langle a, b \rangle_{\mathcal{B}}$ the coefficient of $\eval(b)$ in the expansion of $a$ on the basis $\eval(\mathcal{B})$. 
    We define
    \begin{equation}
        \supp_\mathcal{B} (a) := \left\{ b \in \mathcal{B} ; \langle a, b \rangle_{\mathcal{B}} \neq 0 \right\}.
    \end{equation}
    For $a \in \Br(X)$, $\supp_{\mathcal{B}} (a) := \supp_{\mathcal{B}} (\eval(a))$.
    If $A \subset \Br(X)$, we let $\supp_\mathcal{B} (A) := \cup_{a \in A} \supp_{\mathcal{B}}(a)$.
    We drop the subscripts $\mathcal{B}$ when there is no possible confusion on which basis is used.
\end{definition}

\subsection{Sussmann's infinite product expansion} \label{s:prod-inf}

We present an expansion for the formal power series $x(t)$ solution to \eqref{eq:formal-DE} as a product of exponentials of the elements of a Hall set, multiplied by coefficients that have simple expressions.

This infinite product is an extension to all Hall sets of Sussmann's infinite product on length-compatible Hall sets \cite{MR935387}, suggested in \cite{edd33d15c4b947f39991805f8c1d726f} and proved in \cite[Section 2.5]{P1}.

\begin{definition}\label{Def:Coord2}
    Let $\mathcal{B} \subset \Br(X)$ be a Hall set.
    The \emph{coordinates of the second kind} associated with~$\mathcal{B}$ is the unique family $(\xi_b)_{b\in\mathcal{B}}$ of functionals $\R_+ \times L^1_{\mathrm{loc}}(\R_+;\R) \to \R$ defined by induction in the following way: for every $t>0$ and $u \in \lone$
    \begin{itemize}
    \item $\xi_{X_0}(t,u) := t$ and $\xi_{X_1}(t,u) := \int_0^t u = u_1(t)$,
    \item for $b \in \mathcal{B} \setminus X$, there exists  a unique couple $(b_1,b_2)$ of elements of $\mathcal{B}$ such that $b_1<b_2$ and a unique maximal integer $m \in\N^*$ such that $b=\ad_{b_1}^m (b_2)$ and then
    \begin{equation} \label{eq:def:coord2}
        \xi_{b}(t,u):=\frac{1}{m!} \int_0^t  \xi_{b_1}^m(s,u) \dot{\xi}_{b_2}(s,u) \dd s.
    \end{equation}
    \end{itemize}
\end{definition}

\begin{theorem} \label{thm:Inf_Prod}
    Let $\mathcal{B}\subset \Br(X)$ be a Hall set, $t > 0$ and $u \in \lone$. 
    The solution to the formal differential equation \eqref{eq:formal-DE} satisfies,
    \begin{equation} \label{eq:prod-inf}
        x(t)= \underset{b \in \mathcal{B}}{\overleftarrow{\prod}} e^{\xi_b(t,u) \eval(b)}. 
    \end{equation}
\end{theorem}

\begin{remark}
    In \eqref{eq:prod-inf}, the right-hand side is an infinite oriented product, indexed by elements of $\mathcal{B}$ which are increasing towards the left (see \cite[Section 2.5]{P1} for more precise definitions).

    Outside of the length-compatible case, Hall sets can have infinite segments, i.e.\ for some fixed $b_1 < b_2 \in \mathcal{B}$, there might exist an infinite number of $b \in \mathcal{B}$ such that $b_1 < b < b_2$.
    Hence, one must be careful when defining and interpreting the product \eqref{eq:prod-inf}.

    This situation occurs for our basis $\Bs$, in which $X_1 < M_\nu < W_1$ for all $\nu > 0$ (see \cref{s:B*-def}).
    Hence, in this basis
    \begin{equation} \label{eq:etx0-xtw1-ex1}
        x(t) = e^{t X_0} \dotsb e^{\xi_{W_1} W_1} \dotsb e^{\xi_{M_2} M_2} e^{\xi_{M_1} M_1} e^{\xi_{X_1} X_1}.
    \end{equation}
\end{remark}

\subsection{Magnus expansion in the usual setting}
\label{s:Magnus}

Let us mention the following classical expansion, due to Magnus \cite{MR0067873}, to contrast it with our variant of the next subsection.
Multiple proofs of the following result are known (see \cite[Section 2.3]{P1}).

\begin{theorem} \label{thm:Magnus1.0_formel}
    For $t > 0$ and $u \in \lone$, the solution $x$ to \eqref{eq:formal-DE} satisfies
    \begin{equation} \label{eq:10formal}
     x(t) = \exp \left( Z_{\llbracket 1, \infty \llbracket}(t,X,u) \right)
    \end{equation}
    where $Z_{\llbracket 1, \infty \llbracket}(t,X,u) \in \widehat{\mathcal{L}}(X)$.
    Moreover, if $\mathcal{B} \subset \Br(X)$ is a Hall set, there exists a unique family $(\zeta_b)_{b \in \mathcal{B}}$ of functionals $\R_+ \times \lloc \to \R$, \emph{called coordinates of the first kind} associated with $\mathcal{B}$, such that
    \begin{equation} \label{tilde_Z_infty_decomp_coord1.0}
        Z_{\llbracket 1, \infty \llbracket} (t,X,u) = \sum_{b \in \mathcal{B}} \zeta_b(t, u) \eval(b),
    \end{equation}
    with, in particular, $\zeta_{X_0}(t,u) = t$ and $\zeta_{X_1}(t,u) = u_1(t)$.
\end{theorem}

This expansion is naturally well-suited for truncations with respect to the total length of the involved brackets.
When substituting the formal indeterminates $X_0$ and $X_1$ with vector fields $f_0$ and $f_1$, such truncations lead to error estimates of the form $O(t^{M+1})$ (see e.g.\ \cite[Proposition 93]{P1}).

As explained in \eqref{eq:coerc-interp}, our method to prove obstructions requires error estimates scaling like powers of the control, which can be absorbed by interpolation.
Hence, we will not use expansion \eqref{eq:10formal}, but rather the one described in the following paragraph, which naturally yields such error estimates.
Another approach could be to try to sum infinite subseries of \eqref{tilde_Z_infty_decomp_coord1.0} (say all terms with $n_1(b) \leq M$), but there are convergence issues associated with such an approach, even for analytic vector fields (see \cite[Proposition 138]{P1}).
These convergence issues do not occur for our expansion presented below (see \cref{thm:Key_1}).

\subsection{Our formal expansion}
\label{s:Magnus-Interaction}

We present another expansion for the formal power series $x(t)$ solution to \eqref{eq:formal-DE}, which we recently developed for control theory in \cite[Section 2.4]{P1}, for which we proved the following result.
This expansion is the one at the root of the current work.

\begin{theorem} \label{thm:Magnus1.1_formel}
    For $t > 0$ and $u \in \lone$, the solution $x$ to \eqref{eq:formal-DE} satisfies
    \begin{equation} \label{eq:11formal}
     x(t) = \exp (t X_0) \exp \left( \ZInf (t,X,u) \right),
    \end{equation}
    where $\ZInf(t,X,u) \in \widehat{\mathcal{L}}(X)$.
    Moreover, if $\mathcal{B} \subset \Br(X)$ is a Hall set, there exists a unique family $(\eta_b)_{b \in \mathcal{B}}$ of functionals $\R_+ \times \lloc \to \R$, \emph{called coordinates of the pseudo-first kind} associated with $\mathcal{B}$, such that
    \begin{equation} \label{tilde_Z_infty_decomp_coord1.1}
        \ZInf (t,X,u) = \sum_{b \in \mathcal{B}} \eta_b(t, u) \eval(b),
    \end{equation}
    with, in particular, $\eta_{X_0}(t,u) = 0$ and $\eta_{X_1}(t,u) = u_1(t)$.
\end{theorem}

\begin{remark} \label{rk:pseudo}
    In \cite[Section 2.4]{P1}, we named this expansion ``Magnus expansion in the interaction picture''.
    Indeed, the interaction picture is a representation used in quantum mechanics when the dynamics can be written as the sum of a time-independent part, which can be solved exactly, and a time-dependent perturbation.
    It corresponds to factoring out the $e^{t X_0}$ part of the dynamics.

    As in \cite[Section 2.4.2]{P1}, we call the $\eta_b$ coordinates of the \emph{pseudo-first kind}.
    This is our own wording to denote that these coordinates are somewhere in between those of the first kind (used in \eqref{tilde_Z_infty_decomp_coord1.0}) and those of the second kind (used in \eqref{eq:prod-inf}), but much closer to the former.
    First kind and second kind are well established names \cite[III.4.3]{zbMATH03499968}.
\end{remark}

Since we will work with truncated version of this expansion, we also introduce, for $M \in \N^*$, the notation $\cZ{M}(t,X,u)$ to denote the canonical projection of $\ZInf(t,X,u)$ onto $\mathcal{A}_1(X) \oplus \dotsb \oplus \mathcal{A}_M(X)$, so that one has
\begin{equation}
    \cZ{M}(t,X,u) = \sum_{n_1(b) \leq M} \eta_b(t,u) \eval(b).
\end{equation}
These truncations still contain an infinite number of terms (since $n_0(b)$ is unbounded).
However, when substituting $X_0$ and $X_1$ with analytic vector fields $f_0$ and $f_1$, then enjoy nice convergence properties (see \cref{thm:Key_1}).

\subsection{Computing coordinates of the pseudo-first kind}
\label{sec:eta}

Unlike coordinates of the second kind, which enjoy the nice explicit recursive integral expression~\eqref{eq:def:coord2}, no such formula is known for the coordinates of the first kind or pseudo-first kind.
Our viewpoint in this work is to compute the latter from those of the second kind, by means of the Campbell--Baker--Hausdorff formula, leading to formula \eqref{eq:etab-xib-exact} and estimate \eqref{eq:etab-xib}.

\begin{proposition} \label{Prop:ZM_Bstar}
    There exists a family of elements $F_{q,h}(Y_1,\dotsc,Y_q) \in \mathcal{L}(\{Y_1,\dotsc,Y_q\})$ for $q \in \N^*$ and $h \in (\N^*)^q$, such that, $F_{q,h}(Y_1,\dotsc,Y_q)$ is of degree $h_i$ with respect to $Y_i$ for each $i \in \intset{1,q}$ and, for every Hall set $\mathcal{B} \subset \Br(X)$ with $X_0$ as maximal element, $t>0$ and $u \in \lone$,
    \begin{equation} \label{ZM_Bstar}
        \ZInf(t,X,u)=\sum_{\substack{q\in \N^*, h\in (\N^*)^q \\ b_1 > \dots > b_q \in \mathcal{B} \setminus\{ X_0\} }} \xi_{b_1}^{h_1}(t,u) \dots \xi^{h_q}_{b_q}(t,u) F_{q,h}(b_1,\dots,b_q).
    \end{equation}
    Equivalently, for every $b \in \mathcal{B}$, one has
    \begin{equation} \label{eq:etab-xib-exact}
        \eta_b(t,u) = \xi_b(t,u) + \sum_{\substack{q \geq2, h\in (\N^*)^q \\ b_1 > \dots > b_q \in \mathcal{B} \setminus\{ X_0\} }} \xi_{b_1}^{h_1}(t,u) \dots \xi_{b_q}^{h_q}(t,u) \langle F_{q,h}(b_1,\dots,b_q), b \rangle_{\mathcal{B}}.
    \end{equation}
\end{proposition}

\begin{proof}
    We deduce from \Cref{thm:Inf_Prod,thm:Magnus1.0_formel} and the maximality of $X_0$ (see also \eqref{eq:etx0-xtw1-ex1}) that
    \begin{equation}
        e^{\ZInf(t,X,u)} = \underset{b \in \mathcal{B} \setminus\{X_0\}}{\overleftarrow{\prod}} e^{\xi_b(t,u)\eval(b)}.
    \end{equation}
    Then \eqref{ZM_Bstar} and \eqref{eq:etab-xib-exact} follow from the multivariate CBHD formula \cite[Proposition 34]{P1}\footnote{Technically, this proposition is stated for finite products. 
    Nevertheless, one can use the graded structure of $\widehat{\mathcal{A}}(X)$ to reduce the proof to this finite setting.}.
\end{proof}

\begin{remark}
    The elements $F_{q,h}$ are deeply linked with the CBHD formula and can be iteratively computed from its usual two-variables coefficients.
    One has for example for $q =1$, $F_{1,(1)}(Y_1)=Y_1$, for $q = 2$, $F_{2,(1,1)}(Y_1,Y_2)=\frac{1}{2}[Y_1,Y_2]$, $F_{2,(2,1)}(Y_1,Y_2) = \frac 1 {12} [Y_1, [Y_1, Y_2]]$, and for $q = 3$, $F_{3,(1,1,1)}(Y_1,Y_2,Y_3) = \frac 1 4 [Y_1,[Y_2,Y_3]]$ (see \cite[Proposition 34]{P1} for more details). 
\end{remark}

Equality \eqref{eq:etab-xib-exact} leads to the idea that, in some sense, one has $\eta_b \approx \xi_b$, provided that one can estimate the appropriate cross terms of the right-hand side.

\begin{definition}[$\mathcal{F}$]
    Given $q \geq 2$ and $b_1, \dotsc, b_q \in \Br(X)$, we define $\mathcal{F}(b_1,\dotsc,b_q)$ as the subset of $\Br(X)$ of brackets of $b_1, \dotsc, b_q$ involving each of these elements exactly once.
    For example
    \begin{align}
        \mathcal{F}(b_1,b_2) & = \{ (b_1, b_2), (b_2, b_1) \}, \\
        \mathcal{F}(b_1,b_2,b_3) & = \{ (b_1, (b_2, b_3)),   ((b_1,b_2),b_3), \dotsc \text{and 10 others} \}.
    \end{align}
\end{definition}

\begin{proposition} \label{p:etab-xib-XI}
    Let $\mathcal{B}$ be a Hall set with $X_0$ maximal.
    Let $b \in \mathcal{B}$.
    There exists $C_b > 0$ such that the following property holds.
    Assume that there exists $\Xi : \R_+^* \times \lloc \to \R_+$ such that, for all $q \geq 2$, $b_1 \geq \dotsb \geq b_q \in \mathcal{B} \setminus \{ X_0 \}$ such that $b \in \supp \mathcal{F}(b_1, \dotsc, b_q)$, for every $u \in \lloc$ and $t > 0$,
    \begin{equation} \label{eq:prod-xibi-XI}
        |\xi_{b_1}(t,u) \dotsb \xi_{b_q}(t,u) | 
        \leq \Xi(t,u).
    \end{equation}
    Then, for every $u \in \lloc$ and $t > 0$,
    \begin{equation} \label{eq:etab-xib}
        |\eta_b(t,u) - \xi_b(t,u)| \leq C_b \Xi(t,u).
    \end{equation}
\end{proposition}

\begin{proof}
    This is a straightforward consequence of \eqref{eq:etab-xib-exact} and the fact that the sum in the right-hand side of this equality is finite.
    Indeed, $\langle F_{q,h}(b_1,\dotsc,b_q), b \rangle_{\mathcal{B}} \neq 0$ implies in particular that $h_1 |b_1|+\dotsb+h_q|b_q| = |b|$, so there is a finite number of possibilities for $q$, $h$ and the $b_i$.
\end{proof}

\section{A new Hall basis of the free Lie algebra}
\label{s:B*}

In this section, we define our new basis of the free Lie algebra over two generators $\{ X_0, X_1 \}$, designed for applications to control theory, and compute some of its elements.

In \cref{s:B*-def}, we introduce our definition of a new Hall set, which we call $\Bs$ and motivate its interest for control problems.
In \cref{s:B*-12345}, we give an exhaustive description of the elements of $\Bs$ involving $X_1$ at most 5 times.
In \cref{s:xi-B*-123445}, we compute the associated coordinates of the second kind, while in \cref{s:xi-B*-estimates}, we provide estimates of these coordinates.

\subsection{Definition of \texorpdfstring{$\Bs$}{B*} and first properties}
\label{s:B*-def}

The main result of this paragraph is \cref{Prop:B*exists!} which states the existence and uniqueness of our basis $\Bs$.
We start by introducing some notations and definitions which will make the presentation more meaningful.

First, we define by induction a subset $G$ of $\Br(X)$ by requiring that, $X_0, X_1 \in G$ and, for every $a,b \in G$ with $a \neq X_0$, $(a,b) \in G$.
Heuristically, $G$ is the subset of $b \in \Br(X)$ for which $X_0$ is never the left factor of any sub-bracket within $b$.
This leads to the following definition.

\begin{definition}[Germ] \label{def:germ}
    For any $b \in G \setminus \{ X_0 \}$, there exists a unique couple $(b^*,\nu_b) \in G \times \N$ such that $b = b^* 0^{\nu_b}$, with $b^* = X_1$ or $b^* = (b_1,b_2)$ with $b_1 \neq X_0$ and $b_2 \neq X_0$, where the notation $0^\nu$ is introduced in \cref{def:0nu}.
    We call $b^*$ the \emph{germ} of $b$ and we say that $b$ is a germ when $b = b^*$ (i.e.\ $\nu_b = 0$).
    Let $G^*$ be the subset of $G$ made of germs.
\end{definition}

\begin{example}
    Let $b := ((X_1,X_0),X_0) = X_1 0^2$. 
    Then $b \in G$, the germ of $b$ is $X_1$ and $\nu_b = 2$. 
    Hence $b \notin G^*$.
    However $c := (X_1, (X_1, X_0)) \in G$ and $c$ is a germ.
\end{example}

\begin{definition}[Order for $\Bs$]
    \label{def:order}
    We endow $G$ with the following total order.
    \begin{enumerate}[(B1)] \addtocounter{enumi}{-1}
        \item \label{B0} $X_0$ is the maximal element.
        \item \label{B1} for $a, b \in G \setminus \{X_0\}$, $a < b$ if and only if $a^* < b^*$ or $a^* = b^*$ and $\nu_a < \nu_b$.
        \item \label{B2} for $a^*,b^* \in G^*$, $a^* < b^*$ if and only if
        \begin{itemize}
            \item either $n_1(a^*) < n_1(b^*)$,
            \item or $n_1(a^*) = n_1(b^*)$ and $\lambda(a^*) < \lambda(b^*)$,
            \item or $n_1(a^*) = n_1(b^*)$ and $\lambda(a^*) = \lambda(b^*)$ and $\mu(a^*) < \mu(b^*)$.
        \end{itemize}
    \end{enumerate}
\end{definition}

In other words, $X_1$ is minimal, $X_0$ is maximal and, on $G \setminus X$, the order is the lexicographic order on the quadruple $b \mapsto (n_1(b^*), \lambda(b^*), \mu(b^*), \nu_b)$.

\begin{theorem} \label{Prop:B*exists!}
    There exists a unique Hall set $\Bs \subset G \subset \Br(X)$ associated with \cref{def:order}.
\end{theorem}

\begin{proof}
    By \cite[Lemma 1.37]{A1}, it suffices to check that $<$ is a Hall order on $G$, i.e.\ a total order such that, for every $(a,b) \in G \setminus X$, $a \in G$ and $a < (a,b)$.

    \step{We prove that, for every $a, b \in G$, if neither $a < b$ nor $b < a$ holds, then $a = b$}
    By contradiction, let $a$ and $b$ be a pair, of minimal total length $|a|+|b|$, such that $a \neq b$, and neither $a < b$ nor $b < a$.
    By \bref{B0}, $a \neq X_0$ and $b \neq X_0$.
    By \bref{B1}, $a^* \neq b^*$ (otherwise $\nu_a = \nu_b$ so $a = b$). 
    By \bref{B2}, $n_1(a^*)=n_1(b^*)$ and,
    \begin{itemize}
        \item either $\lambda(a^*) \neq \lambda(b^*)$, and these two brackets are an incomparable pair of shorter total length,
        \item or $\lambda(a^*) = \lambda(b^*)$, but then $\mu(a^*) \neq \mu(b^*)$ is an incomparable pair of shorter total length.
    \end{itemize}
    
    \step{We prove that, for every $(a,b) \in G \setminus X$, $a \in G$ and $a < (a,b)$}
    Let $(a,b) \in G \setminus X$.
    Then $a \in G$ by construction of $G$ by induction.
    If $b = X_0$ then $a < (a,b)$ by \bref{B1}.
    If $b \neq X_0$, then $n_1(a) < n_1((a,b))$ so $a < (a,b)$ by \bref{B2}.
\end{proof}

\begin{remark} \label{rk:Bs-vs-CFL}
    In $\Bs$, $X_0$ is maximal.
    This is similar to the fact that $X_0$ would be maximal in the Chen--Fox--Lyndon basis associated with the order $X_1 < X_0$ on $X$.
    So $\Bs$ shares some properties of this basis (for example, the fact that $b \in \Bs \Rightarrow \forall \nu \in \N, b0^\nu \in \Bs$).
    
    If $b \in \Bs$ is a germ, then, by \bref{B2}, $\mu(b) < b$, because $n_1(\mu(b)) < n_1(b)$.
    This is similar to the situation in length-compatible Hall sets where one always has $\mu(b) < b$ because $|\mu(b)| < |b|$.
    In the Chen--Fox--Lyndon basis however, one has $b < \mu(b)$.
    So $\Bs$ shares some properties of length-compatible Hall sets.
\end{remark}

By analogy with \eqref{eq:S_A1} and \eqref{eq:S_A1A0}, for $A_1,A_0\subset \N$, we will also adopt the notations 
\begin{equation} \label{Bstar_A1_A0}
    \Bs_{A_1}:=\{b\in\Bs;n_1(b)\in A_1\} \quad \text{and} \quad  \Bs_{A_1,A_0}:=\{b\in\Bs;n_1(b)\in A_1, n_0(b)\in A_0\}.
\end{equation}

\subsection{Elements of \texorpdfstring{$\Bs$}{B*} up to the fifth order} 
\label{s:B*-12345}

The goal of this section is to prove \cref{p:B*-12345}, i.e.\
to determine the germs of $\Bs_{\intset{1,5}}$.

If $b^*$ is such a germ, then, by \Cref{def:Hall}, for every $\nu \in \N$, $b^* 0^\nu \in \Bs$ and, by \bref{B1}, for every $\nu_1<\nu_2 \in \N$ then $b^* 0^{\nu_1} < b^* 0^{\nu_2}$.

\begin{proof}[Proof of \cref{p:B*-12345}]
    $X_1$ is the only possible germ in $\Bs_1$, which proves \eqref{Bstar_S1}. Moreover, the sequence $(M_\nu)_{\nu \in \N}$ is increasing
    \begin{equation} \label{Mnu_croit}
        \forall \nu \leq \nu', \quad M_\nu \leq M_{\nu'}.
    \end{equation}
    By \cref{def:Hall}, any germ of $\Bs_{\intset{2,5}}$ is of the form $(a,b)$ where $a, b \in \Bs_{\intset{1,4}}$, and $\lambda(b) \leq a<b$.
    By \bref{B1}, this implies that either $a^* = b^*$ and then $b = (a,X_0)$ so $(a,b) = \ad^2_a(X_0)$, or $a^* < b^*$ and then $b = b^*$ and $n_1(a) \leq n_1(b)$.
    We proceed by increasing degree in $X_1$.
    \begin{itemize}
        \item \emph{Germs of $\Bs_2$:}
        By \Cref{def:Hall}, for every $j\in\N^*$, $W_{j,0}$ belongs to $\Bs$. 
        Indeed $W_{j,0}=\ad_{M_{j-1}^2}(X_0)$ and $M_{j-1}<X_0$ by \bref{B0}.
        These are the only elements of $\Bs_2$ that one may construct by bracketing two elements of $\Bs_1$. 
        Moreover, by \bref{B2}, $W_{j,0}<W_{k,0}$ when $j<k$, thus, by \bref{B1},
        \begin{equation} \label{Wk_croit}
            \forall j<k\in\N^*, \forall \mu\in\N, \quad W_{j,\mu} < W_{k,0}.
        \end{equation}

        \item \emph{Germs of $\Bs_3$:}
        By \Cref{def:Hall}, for $j \leq k \in\N^*$,
        \begin{equation} \label{def:Plj0}
        P_{j,k,0}=(M_{k-1},W_{j,0})=(M_{k-1},(M_{j-1},M_j))
        \end{equation}
        belongs to $\Bs$. 
        Indeed, $M_{j-1} \leq M_{k-1} < W_{j,0}$ by \eqref{Mnu_croit} and \bref{B2} because $n_1(M_{k-1}) < n_1(W_{j,0})$.
        These are the only elements of $\Bs_3$ that one may construct by bracketing an element of $\Bs_1$ with an element of $\Bs_2$.

        \item \emph{Germs of $\Bs_4$ in $(\Bs_1,\Bs_3)$:}
        By \Cref{def:Hall}, for $j \leq k \leq l \in\N^*$,
        \begin{equation} \label{def:Qljk0}
        Q_{j,k,l,0}=(M_{l-1},P_{j,k,0})= (M_{l-1},(M_{k-1},W_{j,0}))
        \end{equation}
        belongs to $\Bs$ because 
       $M_{k-1} \leq M_{l-1} < P_{j,k}$ 
        by \eqref{Mnu_croit} and \bref{B2}. These are the only elements of $\Bs_4$ that one may construct by bracketing an element of $\Bs_1$ with an element of $\Bs_3$.

        \item \emph{Germs of $\Bs_4$ in $(\Bs_2,\Bs_2)$:} 
        By \Cref{def:Hall}, for $j<k \in \N^*$ and $\mu \in \N$,
        \begin{equation}
            Q_{j,\mu,k,0}^\sharp = (W_{j,\mu},W_{k,0}) = (W_{j,\nu},(M_{k-1},M_k))
        \end{equation}
        belongs to $\Bs$. 
        Indeed, $M_{k-1} < W_{j,\mu} < W_{k,0}$ by \bref{B2} and \eqref{Wk_croit}. 
        These are the only elements of $\Bs_4$ that one may construct by bracketing two elements of $\Bs_2$ having different germs.
        
        For $j\in\N^*$ and $\mu\in\N$,
        \begin{equation}
            Q^\flat_{j,\mu,0}=(W_{j,\mu},W_{j,\mu+1})= \ad_{W_{j,\mu}}^2(X_0)
        \end{equation}
        belongs to $\Bs$.
        Indeed, by \bref{B0}, $W_{j,\mu}<X_0$.
        These are the only elements of $\Bs_4$  that one may construct by bracketing two elements of $\Bs_2$ having the same germ.

        \item \emph{Germs of $\Bs_5$ in $(\Bs_1,\Bs_4)$:} By \Cref{def:Hall}, for $j\leq k \leq l \leq m  \in\N^*$,
        \begin{equation}
        R_{j,k,l,m,0}=(M_{m-1},Q_{j,k,l,0})= (M_{m-1},(M_{l-1},P_{j,k,0}))
        \end{equation} 
        belongs to $\Bs$. 
        Indeed, $M_{l-1} \leq M_{m-1}<Q_{j,k,l,0}$ by \eqref{Mnu_croit} and \bref{B2}.
        These are the only elements of $\Bs_5$ that one may construct by bracketing an element of $\Bs_1$ with an element of~$\Bs_4$.

        \item \emph{Germs of $\Bs_5$ in $(\Bs_2,\Bs_3)$:}
        By \Cref{def:Hall}, for $j, k,l\in\N^*$ such that $j\leq k$ and $\mu\in\N$,
        \begin{equation}
            R^\sharp_{j,k,l,\mu,0} = (W_{l,\mu},P_{j,k,0})=(W_{l,\mu},(M_{k-1},W_{j}))
        \end{equation}
        belongs to $\Bs$.
        Indeed, $M_{k-1}<W_{l,\mu}<P_{j,k,0}$ by \bref{B2}.
        These are the only elements of $\Bs_5$ one may construct by bracketing an element of $\Bs_2$ with an element of $\Bs_3$.
    \end{itemize}
    This concludes the proof.
\end{proof}

\subsection{Expressions of coordinates of the second kind up to the fifth order} \label{s:xi-B*-123445}

In this paragraph, we give explicit expressions of the coordinates of the second kind, as defined in \cref{Def:Coord2} associated with the elements of $\Bs$ up to the fifth order in the control introduced in \cref{s:B*-12345}.
We start with the following lemma, which helps in visualizing the coordinates of the second kind associated with the elements of $\Bs_{\intset{1,5}}$ listed in \cref{p:B*-12345}. 

\begin{lemma} \label{Lem:adnuX0}
    For $b\in\Bs$ and $\nu\in\N$, 
    \begin{equation}
        \xi_{b 0^\nu}(t,u) = \int_0^t \frac{(t-s)^\nu}{\nu!} \dot{\xi}_{b}(s,u) \dd s.
    \end{equation}
\end{lemma}

\begin{proof}
    By induction on $\nu$, this follows from \cref{Def:Coord2} and the fact that $\Bs$ satisfies \bref{B0}.
\end{proof}

\begin{proposition} \label{Prop:Coord_Bstar}
    For every $j \leq k \leq l \leq m \in\N^*$,  $\mu, \nu \in \N$, we have
    \begin{align} 
    \label{xi_S1}
    & \xi_{M_\nu}(t,u)=\int_0^t \frac{(t-s)^\nu}{\nu!} u(s) \dd s = u_{\nu+1}(t),
    \\ 
    \label{xi_Wjnu}
    & \xi_{W_{j,\nu}}(t,u)= \frac 1 2 \int_0^t \frac{(t-s)^\nu}{\nu!} u_j^2(s) \dd s,
    \\ 
    \label{xi_S3}
    & \xi_{P_{j,k,\nu}}(t,u)=
    \alpha_{j,k}
    \int_0^t \frac{(t-s)^\nu}{\nu!} u_k(s) u_j^2(s) \dd s,
    \\ 
    \label{xi_Qljknu}
    & \xi_{Q_{j,k,l,\nu}}(t,u)= \beta_{j,k,l}
    \int_0^t \frac{(t-s)^\nu}{\nu!} u_l(s) u_k(s) u_j^2(s) \dd s,
    \\ 
    \label{xi_Qbemolknu}
    & \xi_{Q^\flat_{j,\mu,\nu}}(t,u)=\frac{1}{8} \int_0^t \frac{(t-s)^\nu}{\nu!} 
    \left( \int_0^s  \frac{(s-s')^\mu}{\mu!} u_j^2(s') \dd s' \right)^2 \dd s,
    \\
    \label{xi_Qdieseknujmu}
    &\xi_{Q^\sharp_{j,\mu,k,\nu}}(t,u)= \frac 1 4
    \int_0^t \frac{(t-s)^\nu}{\nu!} 
   \left( \int_0^s \frac{(s-s')^\mu}{\mu!} u_j^2(s') \dd s' \right) u_k^2(s) \dd s,
    \\ 
    \label{xi_Rl3l2l1knu}
    & \xi_{R_{j,k,l,m,\nu}}(t,u)=\gamma_{j,k,l,m}\int_0^t \frac{(t-s)^\nu}{\nu!} u_{m}(s) u_{l}(s) u_{k}(s) u_j^2(s) \dd s,
    \\
    \label{xi_Rdiesel3l2l1knu}
    & \xi_{R^\sharp_{j,k,l,\mu,\nu}}(t,u)= \frac{\alpha_{j,k}}{2} \int_0^t \frac{(t-s)^\nu}{\nu!} \left( \int_0^s \frac{(s-s')^\mu}{\mu!} u_l^2(s') \dd s' \right) u_k(s) u_j(s)^2 \dd s,
    \end{align}
    where $j < k$ in \eqref{xi_Qdieseknujmu} (only), and the coefficients are given by
    \begin{align}
        \alpha_{j,k}&=\frac{1}{2!} \delta_{j<k} + \frac{1}{3!} \delta_{j=k}, \\
        \beta_{j,k,l}&=\alpha_{j,k} \delta_{k < l} + \frac{1}{(2!)^2} \delta_{j<k=l}+\frac{1}{4!} \delta_{j=k=l}, \\
        \gamma_{j,k,l,m} &= \beta_{j,k,l} \delta_{l < m} + \frac{1}{5!} \delta_{j=k=l=m} + \frac{1}{(2!)^2} \delta_{j<k<l=m} + \frac{1}{2! 3!} (\delta_{j<k=l=m}+\delta_{j=k<l=m}).
    \end{align}
\end{proposition}

\begin{proof}
    It follows from the application of \cref{Def:Coord2} to the elements of \cref{p:B*-12345}.
\end{proof}
 
\begin{remark} \label{Rk_OP1_Kawski}
    In \cite{kawski-unsolved}, motivated by questions of control theory, Kawski formulated the following open problem: ``\emph{construct a basis for the free Lie algebra such that the corresponding coordinates of the first kind have simple formulas}''. 
    In this paper, we follow a slightly different approach: we construct the basis $\Bs$, whose coordinates of the second kind have particularly nice explicit expressions, and we use these to obtain controllability results by bounding the differences $\eta_b-\xi_b$ (the cross terms), even though formula \eqref{eq:etab-xib} for these differences is quite messy.

    A key strength of the basis $\Bs$ is that it is easy to visualize if a given coordinate of the second kind is positive-definite, and will lead to an obstruction to STLC. 
    We observe in particular that, for every $k \in \N^*$, the quadratic form $\xi_{W_k}$ is positive-definite: this is a key point for \cref{Thm:Kawski_Wm}. 
    The positivity of the values of $\xi_{Q_{j,k,k}}$ is a key point for the quartic necessary conditions which we intend to study in a forthcoming work.
    Finally, one may expect that for any germ $\bb \in \Bs$ such that $\xi_{\bb}$ is positive-definite, a necessary condition for $W^{m,\infty}$-STLC of the form $f_\bb(0) \in \mathcal{N}(f)(0)$ holds, at least for $m$ large enough.
    As an example, the lie bracket $D := \ad_{P_{1,1}}^2(X_0) \in \Bs$ studied in \cref{s:sextic} is associated with $\xi_D(t,u) := \frac{1}{72} \int_0^t (\int_0^s u_1^3)^2 \dd s \geq 0$ and indeed leads to an obstruction.
\end{remark}

\subsection{Estimates on the coordinates of the second kind up to the fifth order}
\label{s:xi-B*-estimates}

We start with a rough estimate valid for all brackets of $\Bs \setminus X$, which will be mainly used to prove convergence of the considered series.
This statement follows from \cite[Lemma 156]{P1} and is thus valid within any Hall set such that $X_1 < X_0$.
For self-containedness, we give a direct proof in the case of $\Bs$ in \cref{s:proof-bound-xi-univ}.

\begin{proposition} \label{p:xib-u1-Lk}
    For every $k \in \N^*$, there exists $c=c(k)>0$ such that, for every $b \in \Bs \setminus\{X_1\}$ with $n_1(b)=k$, $t>0$ and $u \in \lone$,
    \begin{equation} \label{eq:xib-u1-Lk}
        | \xi_b(t,u) | \leq \frac{(ct)^{|b|}}{|b|!} t^{-(1+k)} \|u_1\|_{L^k}^k.
    \end{equation}
\end{proposition}

To prove our obstruction results, we need more accurate estimates on the coordinates of the second kind associated with $\Bs_{\llbracket 1 , 5 \rrbracket}$, in terms of Sobolev norms of primitives of the control.
This is the goal of the following statement, proved in \cref{s:proof-bound-xi-12345}

\begin{proposition} \label{p:xi-bounds}
    The following bounds hold.
    \begin{enumerate}
        \item Let $p \in [1,\infty]$ and $j_0 \in \N^*$.
        There exists $c > 0$ such that, for every $j \geq j_0$, $t > 0$ and $u \in \lone$, $\ell := |M_j| \geq j_0+1$ and
        \begin{equation}
            \label{bound-xiMj/J0}
            |\xi_{M_j}(t,u)| 
            \leq \frac{(ct)^{\ell}}{\ell!} 
            t^{-(j_0+1)} t^{1-\frac 1 p} \|u_{j_0}\|_{L^p}.
        \end{equation}
        
        \item Let $p \in [1,\infty]$ and $j_0 \in \N^*$.
        There exists $c > 0$ such that, for every $j \geq j_0$, $\nu \geq 0$, $t > 0$ and $u \in \lone$, $\ell:= |W_{j,\nu}| \geq 2j_0+1$ and
        \begin{equation}
            \label{bound-xiWjnu/j0}
            |\xi_{W_{j,\nu}}(t,u)| 
            \leq \frac{(ct)^{\ell}}{\ell!}
            t^{-(2j_0+1)} t^{1-\frac 1 p} \|u_{j_0}\|_{L^{2p}}^2.
        \end{equation}
        
        \item 
        Let $p_1,p_2 \in [1,\infty]$ such that $\frac{1}{p_1}+\frac{1}{p_2}\leq1$ and $j_0, k_0 \in \N^*$.
        There exists $c > 0$ such that, for every $j \geq j_0$, $k \geq k_0$ with $j \leq k$, $\nu \geq 0$, $t > 0$ and $u \in \lone$, $\ell := |P_{j,k,\nu}| \geq 2j_0+k_0+1$ and
        \begin{equation}
            \label{bound-xiPjknu/j0k0}
            |\xi_{P_{j,k,\nu}}(t,u)|
            \leq \frac{(ct)^{\ell}}{\ell!}
            t^{-(2j_0+k_0+1)}
            t^{1-\frac{1}{p_1}-\frac{1}{p_2}}
            \|u_{j_0}\|_{L^{2p_1}}^2 \|u_{k_0}\|_{L^{p_2}}.
        \end{equation}
        
        \item Let $p_1,p_2,p_3 \in [1,\infty]$ such that $\frac{1}{p_1}+\frac{1}{p_2}+\frac{1}{p_3}\leq1$ and $j_0,k_0,l_0 \in \N^*$.
        There exists $c > 0$ such that, for every $j \geq j_0$, $k \geq k_0$, $l \geq l_0$ with $j \leq k \leq l$, $\nu \geq 0$, $t > 0$ and $u \in \lone$, $\ell := |Q_{j,k,l,\nu}| \geq 2j_0+k_0+l_0+1$ and
        \begin{equation}
            \label{bound-xiQjklnu/j0k0l0}
            |\xi_{Q_{j,k,l,\nu}}(t,u)| 
            \leq \frac{(ct)^{\ell}}{\ell!} t^{-(2j_0+k_0+l_0+1)}
            t^{1-\frac{1}{p_1}-\frac{1}{p_2}-\frac{1}{p_3}}
            \|u_{j_0}\|_{L^{2p_1}}^2 \|u_{k_0}\|_{L^{p_2}} \|u_{l_0}\|_{L^{p_3}}.
        \end{equation}
        
        \item Let $p \in [1,\infty]$ and $j_0 \in \N^*$.
        There exists $c > 0$ such that, for every $j \geq j_0$, $\mu,\nu \in \N$, $t > 0$ and $u \in \lone$, $\ell := |Q^\flat_{j,\mu,\nu}| \geq 4j_0+3$ and
        \begin{equation}
            \label{bound-xiQflatjmunu/j0}
            |\xi_{Q^\flat_{j,\mu,\nu}}(t,u)|
            \leq \frac{(ct)^{\ell}}{\ell!} 
            t^{-(4j_0+3)}
            t^{3-\frac{2}{p}}
            \|u_{j_0}\|_{L^{2p}}^4.
        \end{equation}
        
        \item Let $p_1, p_2 \in [1,\infty]$ and $j_0,k_0 \in \N^*$.
        There exists $c > 0$ such that, for every $j \geq j_0$, $k \geq k_0$, with $j < k$, $\mu,\nu \geq 0$, $t > 0$ and $u \in \lone$, $\ell := |Q^\sharp_{j,\mu,k,\nu}| \geq 2j_0+2k_0+2$ and
        \begin{equation}
            \label{bound-xiQsharpjmuknu/j0k0}
            |\xi_{Q^\sharp_{j,\mu,k,\nu}}(t,u)|
            \leq \frac{(ct)^{\ell}}{\ell!} 
            t^{-(2j_0+2k_0+2)}
            t^{2-\frac{1}{p_1}-\frac{1}{p_2}}
            \|u_{j_0}\|_{L^{2p_1}}^2 \|u_{k_0}\|_{L^{2p_2}}^2.
        \end{equation}
        
        \item Let $p_1,p_2,p_3,p_4 \in [1,\infty]$ such that $\frac{1}{p_1}+\frac{1}{p_2}+\frac{1}{p_3}+\frac{1}{p_4} \leq 1$ and $j_0,k_0,l_0,m_0 \in \N^*$.
        There exists $c > 0$ such that, for every $j \geq j_0$, $k \geq k_0$, $l \geq l_0$, $m \geq m_0$ with $j \leq k \leq l \leq m$, $\nu \geq 0$, $t > 0$ and $u \in \lone$, $\ell := |R_{j,k,l,m,\nu}| \geq 2j_0+k_0+l_0+m_0+1$ and
        \begin{equation}
            \label{bound-xiRjklmnu/j0k0l0m0}
            \begin{split}
                |\xi_{R_{j,k,l,m,\nu}}(t,u)|
                \leq \frac{(ct)^{\ell}}{\ell!} 
                t^{-(2j_0+k_0+l_0+m_0+1)}
                & t^{1-\frac{1}{p_1}-\frac{1}{p_2}-\frac{1}{p_3}-\frac{1}{p_4}} \\
                & \times \|u_{j_0}\|_{L^{2p_1}}^2 \|u_{k_0}\|_{L^{p_2}} \|u_{l_0}\|_{L^{p_3}} \|u_{m_0}\|_{L^{p_4}}.
            \end{split}
        \end{equation}
        
        \item Let $p,p_1,p_2 \in [1,\infty]$ such that $\frac{1}{p_1}+\frac{1}{p_2} \leq 1$ and $j_0,k_0,l_0 \in \N^*$.
        There exists $c > 0$ such that, for every $j \geq j_0$, $k \geq k_0$, $l \geq l_0$, with $j \leq k$, $\mu,\nu \geq 0$, $t > 0$ and $u \in \lone$, $\ell := |R^\sharp_{j,k,l,\mu,\nu}| \geq 2j_0+k_0+2l_0+2$ and
        \begin{equation}
            \label{bound-xiRsharpjklmunu/j0k0l0}
            |\xi_{R^\sharp_{j,k,l,\mu,\nu}}(t,u)|
            \leq \frac{(ct)^{\ell}}{\ell!} 
            t^{-(2j_0+k_0+2l_0+2)}
            t^{2-\frac{1}{p}-\frac{1}{p_1}-\frac{1}{p_2}}
            \|u_{j_0}\|_{L^{2p_1}}^2 \|u_{k_0}\|_{L^{p_2}} \|u_{l_0}\|_{L^{2p}}^2.
        \end{equation}
    \end{enumerate}
\end{proposition}

\section{Toolbox for our approach to obstructions}
\label{s:tools}

In this section, we gather results of various nature as a toolbox for the sequel.

First, we recall elementary definitions and estimates for analytic vector fields in \cref{s:analytic} and introduce in \cref{s:O} a notation $O(\cdot)$ which will be used heavily throughout the paper.

Then, we state in \cref{s:Magnus-ODE} the counterpart for system \eqref{syst} of the formal expansion \eqref{eq:11formal} and give in \cref{s:black-box} a sufficient condition to replace, in some sense, the coordinates of the pseudo-first kind by those of the second kind in \eqref{tilde_Z_infty_decomp_coord1.1}.
We show nevertheless in \cref{s:pollutions} that this simplification is not always valid.

Eventually, we recall in \cref{s:gagliardo-nirenberg} the Gagliardo--Nirenberg interpolation inequalities, and state straight-forward consequences of the Jacobi identity in \cref{s:jacobi}.

\subsection{Analytic estimates for vector fields}
\label{s:analytic}

For $a \in \N^*$ and a multi-index $\alpha=(\alpha^1,\dotsc,\alpha^a) \in \N^a$, we use the notations $|\alpha|:=\alpha^1+\dotsb+\alpha^a$,  $\partial^\alpha:=\partial_{x_1}^{\alpha^1}\dotsb\partial_{x_a}^{\alpha^a}$ and $\alpha!:=\alpha^1!\dotsb\alpha^a!$. Then, the following estimate can be proved by iterating $ 2^{-(p+q)} (p+q)! \leq p! q! \leq (p+q)! $ for every $p,q\in\N$,

\begin{align}
    \label{factorielle}
    \forall a \in \N^*, \forall \alpha=(\alpha^1,\dotsc,\alpha^a) \in \N^a, & \quad
2^{-(a-1)|\alpha|} |\alpha|! \leq  \alpha! \leq |\alpha|!
\end{align}

\begin{definition}[Analytic vector fields, analytic norms]
    Let $\delta > 0$ and $B_\delta$ be the closed ball of radius $\delta$, centered at $0 \in \R^d$.
    For $r > 0$, we define $\CC^{\omega,r}(B_\delta;\R^d)$ as the subspace of analytic vector fields on an open neighborhood of $B_\delta$, for which the following norm is finite
     \begin{equation} \label{eq:def.analytic.r}
        \opnorm{f}_r :=
        \sum_{i=1}^b \sum_{\alpha \in \N^d} \frac{r^{|\alpha|}}{\alpha!} \| \partial^\alpha f_i \|_{L^\infty(B_\delta)}.
    \end{equation}
    For each a analytic vector field $f$ on a neighborhood of $0$, there exist $r,\delta > 0$ such that $f \in \CC^{\omega,r}(B_\delta;\R^d)$ (see \cite[Proposition 2.2.10]{krantz2002primer}).
\end{definition}

The following classical result is proved, for instance in \cite[Lemma  80]{P1}.

\begin{lemma}[Analytic estimate] \label{thm:bracket.analytic}
    Let $r, \delta > 0$, $r':=r/e$, $f_0, f_1 \in \CC^{\omega,r}(B_\delta;\R^d)$ and $b \in \Br(X)$.
    Then, $f_b \in \CC^{\omega,r'}(B_\delta;\R^d)$ and
    \begin{equation} \label{eq:bracket.analytic}
    \opnorm{f_b}_{r'} 
            \leq \frac{r}{9} (|b|-1)! \left( \frac{9 \opnorm{f}_r}{r} \right)^{|b|} ,
    \end{equation}
    where $\opnorm{f}_r :=\max\{\opnorm{f_0}_r;\opnorm{f_1}_r \}$.
\end{lemma}

\subsection{Implicit limit for the big \texorpdfstring{$O$}{O} notation} 
\label{s:O}

Given two functions $A(x,u)$ and $B(x,u)$ of interest, we will write that $A(x,u) = O(B(x,u))$ when there exists $C, \rho>0$ such that, for every $t \in (0,\rho)$, $u \in \lone$ with $\|u\|_{W^{-1,\infty}}\leq \rho$ (recall definition \eqref{def:norm_W-1}), then
\begin{equation}
    |A( x(t;u), u )| \leq C B( x(t;u) , u ).
\end{equation}
Hence, throughout this paper, this notation refers to the implicit limit $(t,\|u\|_{W^{-1,\infty}}) \to 0$. 

As examples, one has $t = O(1)$ and $\|u_1\| = O(1)$.
A deeper result is the following estimate which states that, for scalar-input systems of the form \eqref{syst}, the $W^{-1,\infty}$ norm of the control is an upper bound for the size of the sate.

\begin{lemma} \label{p:small-state}
    Let $f_0$, $f_1$ be analytic vector fields on a neighborhood of $0$ with $f_0(0) = 0$.
    Then
    \begin{equation}
        x(t;u) = O(\|u_1\|_{L^\infty}).
    \end{equation}
\end{lemma}

\begin{proof}
    This follows from \cite[Proposition 145]{P1}.
\end{proof}

\subsection{A new representation formula for ODEs} \label{s:Magnus-ODE}

As stated in \cref{s:heuristic-magnus}, our proofs rely on the following recent representation formula, which is the counterpart of the formal expansion \eqref{eq:11formal} for solutions to nonlinear ODEs of the form \eqref{syst} involving analytic vector fields.

\begin{theorem} \label{thm:Key_1}
    Let $M\in\N^*$, $\delta,r>0$ and $f_0, f_1 \in \CC^{\omega,r}(B_{\delta};\R^d)$ with $f_0(0)=0$. 
    Then
    \begin{equation} \label{eq:x=ZM+O}
        x(t;u) = \cZ{M}(t,f,u)(0) + O \left(\|u_1\|_{L^{M+1}}^{M+1} + |x(t;u)|^{1+\frac{1}{M}} \right),
    \end{equation}
    where
    \begin{equation} \label{eq:ZM-etab}
        \cZ{M}(t,f,u) = \sum_{b \in \Bs_{\intset{1,M}}} \eta_b(t,u) f_b,
    \end{equation}
    where this infinite sum converges absolutely in $\CC^{\omega,r'}(B_{\delta};\R^d)$ for any $r' \in [r/e,r)$.
\end{theorem}

\begin{proof}
    Equality \eqref{eq:x=ZM+O} is the third item of \cite[Proposition 161]{P1}.
    The absolute convergence in \eqref{eq:ZM-etab} is proved in \cite[Proposition 103]{P1} and relies on the fundamental observation that the structure constants of Hall bases exhibit ``asymmetric geometric growth'' (see \cite[Theorem 1.9]{A1}).
\end{proof}

\subsection{Black-box estimates for infinite sums and cross terms}
\label{s:black-box}

In order to carry out the program sketched in \cref{s:heuristic-conj-drift}, we will need to estimate infinite sums of terms of the form $\eta_b f_b$ (say for $b$ ranging over $\mathcal{E} \subset \Bs$).
We state below two important black-box estimates to deal with such sums.
\cref{p:sum-xi-XI} deals with sums of the main terms $\sum_{b \in \mathcal{E}} \xi_b f_b$ and \cref{p:sum-cross-XI} deals with the associated cross terms $\sum_{b \in \mathcal{E}} (\eta_b-\xi_b) f_b$ (in particular, it can be seen as a uniformly summed version of \cref{p:etab-xib-XI}, which involved a constant depending on~$b$).
In the sequel, we will only rely on the ``packaged'' version given by \cref{p:PZM-XI}.

The statements are quite technical since bounding such infinite sums requires a lot of uniformity in the assumptions.
We postpone the proofs to \cref{s:proof-black-box}.

\begin{proposition}[Estimate of main terms] \label{p:sum-xi-XI}
    Let $M, L \in \N^*$. 
    Let $\mathcal{E} \subset \Bs_{\intset{1,M}}$.
    Assume that there exist $c > 0$, and a functional $\Xi : \R_+^* \times \lloc \to \R_+$ such that the following holds:
    \begin{itemize}
        \item for all $b \in \mathcal{E}$, there exists an exponent $\sigma \leq \min \{L,|b|\}$, \\
        such that, for all $t > 0$ and $u \in \lone$,
        \begin{equation} \label{eq:xib-XI-L}
            |\xi_b(t,u)| \leq \frac{(ct)^{|b|}}{|b|!} t^{-\sigma} \Xi(t,u).
        \end{equation}
    \end{itemize}
    Let $\delta, r > 0$ and $f_0,f_1 \in \CC^{\omega,r}(B_\delta;\R^d)$.
    Then, for any $r' \in [r/e,r)$,
    \begin{equation}
        \sum_{b \in \mathcal{E}} \opnorm{ \xi_b(t,u) f_b }_{r'} = O(\Xi(t,u)).
    \end{equation}
\end{proposition}

\begin{proposition}[Estimate of cross terms] \label{p:sum-cross-XI}
    Let $M, L \in \N^*$. 
    Let $\mathcal{E} \subset \Bs_{\intset{1,M}}$.
    Assume that there exist $c > 0$, and $\Xi : \R_+^* \times \lloc \to \R_+$ with $\Xi(t,u) = O(1)$ such that the following holds:
    \begin{itemize}
        \item for all $q \geq 2$, $b_1 \geq \dotsb \geq b_q \in \Bs \setminus \{ X_0 \}$ such that $\supp \mathcal{F}(b_1, \dotsc, b_q) \cap \mathcal{E} \neq \emptyset$, there exist $\sigma_1,\dotsc,\sigma_q \leq L$ with $\sigma_i \leq |b_i|$ and $(\alpha_1, \dotsc, \alpha_q) \in [0,1]^q$ with $\alpha := \alpha_1 + \dotsb + \alpha_q \geq 1$, \\
        such that, for all $t>0$ and $u \in \lone$,
        \begin{equation} \label{eq:xib-othercross-XI-L}
            |\xi_{b_i}(t,u)| 
            \leq \frac{(ct)^{|b_i|}}{|b_i|!} t^{-\sigma_i} (\Xi(t,u))^{\alpha_i}.
        \end{equation}
    \end{itemize}
    Let $\delta, r > 0$ and $f_0,f_1 \in \CC^{\omega,r}(B_\delta;\R^d)$.
    Then, for any $r' \in [r/e,r)$,
    \begin{equation}
        \sum_{b \in \mathcal{E}} \opnorm{ (\eta_b-\xi_b)(t,u) f_b }_{r'} = O(\Xi(t,u)).
    \end{equation}
\end{proposition}

\begin{corollary} \label{p:PZM-XI}
    Let $M, L \in \N^*$.
    Let $\bb \in \Bs_{\intset{1,M}}$ and $\mathcal{N} \subset \Bs_{\intset{1,M}}$ with $\bb \notin \mathcal{N}$.
    Assume that there exist $c > 0$ and a functional $\Xi : \R_+^* \times \lloc \to \R_+$ with $\Xi(t,u) = O(1)$ such that 
    \begin{itemize}
        \item the assumption of \cref{p:sum-xi-XI} holds for $\mathcal{E} = \Bs_{\intset{1,M}} \setminus (\mathcal{N} \cup \{\bb\})$,
        \item the assumption of \cref{p:sum-cross-XI} holds for $\mathcal{E} = \Bs_{\intset{1,M}} \setminus \mathcal{N}$.
    \end{itemize}
    Let $f_0$, $f_1$ be analytic vector fields on a neighborhood of $0$.
    If $f_{\bb}(0) \notin \mathcal{N}(f)(0)$ and $\mathbb{P}$ is a component along $f_{\bb}(0)$ parallel to $\mathcal{N}(f)(0)$, then
    \begin{equation} \label{eq:PZM-xibb-OXi}
        \mathbb{P} \cZ{M}(t,f,u)(0) 
        = \xi_{\bb}(t,u)
        + O\left(\Xi(t,u)\right).
    \end{equation}
\end{corollary}

\begin{proof}
    This is a direct consequence of the definition \eqref{eq:ZM-etab} of $\mathcal{Z}_{\intset{1,M}}$ and the above propositions.
\end{proof}

\subsection{Cross terms are not negligible in general}
\label{s:pollutions}

The expression \eqref{eq:etab-xib-exact} of $\eta_b$ as $\xi_b$ plus a finite sum of cross terms leads to the idea that one could maybe replace the coordinates of the pseudo-first kind by those of the second kind in \eqref{eq:x=ZM+O}, by absorbing the difference in the remainder terms which already appear in the right-hand side.
One could define
\begin{equation}
    \cZ{M}^{\text{pure}}(t,X,u) := \sum_{b \in \Bs_{\intset{1,M}}} \xi_b(t,u) \eval(b)
\end{equation}
and ask whether the estimate \eqref{eq:x=ZM+O} holds when
$\cZ{M}(t,f,u)(0)$ is replaced by  $\cZ{M}^{\text{pure}}(t,f,u)(0)$. The following example gives a negative answer to this question, which motivates the introduction of appropriate techniques (sketched in \cref{s:method-args}) to deal with the cross terms.

\begin{proposition}
    The estimate
    \begin{equation} \label{ZMpropre_estim_fausse}
        x(t;u)=\cZ{M}^{\text{pure}}(t,f,u)(0) + O \left(\|u_1\|_{L^{M+1}}^{M+1} + |x(t;u)|^{1+\frac{1}{M}} \right) 
    \end{equation}
    is false with $M=4$ for the system
    \begin{equation} \label{eq:syst-no-ZM-pure}
        \begin{cases}
            \dot{x}_1 = u \\ 
            \dot{x}_2 = x_1 + \frac 12 x_1^2 \\
            \dot{x}_3 = - x_1 x_2.
        \end{cases}
    \end{equation}
\end{proposition}

\begin{proof}
    \step{Computation of $\cZ{4}^{\text{pure}}(t,f,u)(0)$}
    Let $f_0,f_1:\R^3 \to \R^3$ be defined by $f_0(x):=(x_1+ x_1^2/2 )e_2 - x_1 x_2 e_3$ and $f_1(x):=e_1$. 
    Elementary computations prove that, for every $j \in \N^*$ and $\nu \in \N$,
    \begin{equation}
        f_{M_j}(x)= 
        \begin{cases}
            e_1 & \text{ if } j=0, \\
            (1+x_1)e_2-x_2 e_3 & \text{ if } j=1, \\
            \frac 12 x_1^2 e_3 & \text{ if } j=2, \\
            0 & \text{ if } j \geq 3,
        \end{cases}
        \qquad
        f_{W_{j,\nu}}(x) = 
        \begin{cases}
            e_2 & \text{ if } (j,\nu)=(1,0), \\
            -x_1 e_3 & \text{ if } (j,\nu)=(1,1), \\
            0 & \text{ otherwise }
        \end{cases}
    \end{equation}
    $f_{P_{1,2}}(x)=e_3$ and $f_b(x)=0$ for any other $b \in \Bs_{\intset{3,4}}$. 
    Thus
    \begin{equation}
        f_{X_1}(0)=e_1, \qquad f_{M_1}(0)=f_{W_1}(0)=e_2, \qquad f_{P_{1,2}}(0)=e_3,
    \end{equation}
    and $f_b(0)=0$ for any other $b \in \Bs_{\intset{1,4}}$.
    Using \cref{Prop:Coord_Bstar}, we obtain
    \begin{equation}
        \begin{split}
            \cZ{4}^{\text{pure}}(t,f,u)(0) 
            & =
            \xi_{X_1}(t,u) e_1 + (\xi_{M_1}+\xi_{W_1})(t,u) e_2 + \xi_{P_{1,2}}(t,u) e_3
            \\ & =
            u_1(t) e_1 + \left( u_2(t) + \int_0^t \frac{u_1^2}{2} \right) e_2 + \int_0^t u_2 \frac{u_1^2}{2} e_3.
        \end{split}
    \end{equation}

    \step{Computation of $x(t;u)-\cZ{4}^{\text{pure}}(t,f,u)(0)$} \label{proof410:step2}
    By solving explicitly the system and using an integration by parts, one obtains
    \begin{equation}
        x(t;u)=u_1(t) e_1 + \left( u_2(t) + \int_0^t \frac{u_1^2}{2} \right) e_2 + \left(-\frac{1}{2}u_2(t)^2 - u_2(t) \int_0^t \frac{u_1^2}{2} + \int_0^t u_2 \frac{u_1^2}{2} \right)e_3
    \end{equation}
    thus
    \begin{equation}
        x(t;u) - \cZ{4}^{\text{pure}}(t,f,u)(0) = - \frac{1}{2} u_2(t) \left( u_2(t) + \int_0^t u_1^2  \right)e_3. 
    \end{equation}
    In particular, for any $u \in \lone$ such that $x_2(t;u)=0$, one has
    \begin{equation} \label{eq:x-ZMpure-contra_bis}
        x(t;u)-\cZ{4}^{\text{pure}}(t,f,u)(0) =  \frac{1}{8} \left( \int_0^t u_1^2 \right)^2 e_3.
    \end{equation}

    \eqstep{We prove that, for every $t>0$ and $u \in L^\infty(0,t)$ with $u_1(t)=0$,}
    \begin{equation} \label{L5^5<L2^4}
        \int_0^t |u_1|^5 \leq 3 \|u\|_{L^\infty} \left( \int_0^t u_1^2 \right)^2.
    \end{equation}
    Using an integration by parts, we obtain
    \begin{equation}
        \int_0^t |u_1|^5 
        = \int_0^t u_1^2 |u_1|^3  
        =  - 3 \int_0^t \left(\int_0^{\tau} u_1^2\right) u(\tau) u_1(\tau) |u_1(\tau)| \dd\tau 
        \leq 3 \|u\|_{L^\infty} \left( \int_0^t u_1^2 \right)^2.
    \end{equation}

    \step{Conclusion} 
    Working by contradiction we assume there exists $C,\rho>0$ such that, for every $t \in (0,\rho)$ and $u \in L^1(0,t)$ with $\|u_1\|_{L^\infty}<\rho$,
    \begin{equation}
        |x(t;u) - \cZ{4}^{\text{pure}}(t,f,u)(0)| \leq C  \left(\|u_1\|_{L^{5}}^{5} + |x(t;u)|^{\frac{6}{5}} \right).
    \end{equation}
    Then, by Step \ref{proof410:step2}, for every $t \in (0,\rho)$ and $u \in L^1(0,t)$ such that $\|u_1\|_{L^\infty}<\rho$ and $x_2(t;u)=0$, one has
    \begin{equation} \label{L2^4<C*L5^5}
    \frac{1}{8} \left( \int_0^t u_1^2 \right)^2 \leq C  \left(\|u_1\|_{L^{5}}^{5} + |x(t;u)|^{\frac{6}{5}} \right) 
    \end{equation}
    System \eqref{eq:syst-no-ZM-pure} is $L^\infty$-STLC (for example thanks to Hermes' sufficient condition of \cite[Theorem 3.2]{hermes1982control}). 
    In particular, for every $\varepsilon>0$ there exists $t \in (0,\varepsilon)$ and $u \in L^{\infty}(0,t) \setminus \{0\}$ with $\|u\|_{L^\infty}\leq \varepsilon$ such that $x(t;u) = 0$ and then \eqref{L2^4<C*L5^5} and \eqref{L5^5<L2^4} prove
    \begin{equation}
    \frac{1}{8} \left( \int_0^t u_1^2 \right)^2
    \leq C \int_0^t |u_1|^5 \leq 3 C \varepsilon \left( \int_0^t u_1^2 \right)^2
    \end{equation}
    which gives a contradiction when $\varepsilon$ is small enough, precisely $\varepsilon<\rho$, 
    $\varepsilon^2<\rho$ and $24 C \varepsilon <1$.
\end{proof}

\subsection{Interpolation inequalities}
\label{s:gagliardo-nirenberg}

We recall below the Gagliardo--Nirenberg interpolation inequalities (see \cite{gagliardo,nirenberg}) used in this article.

\begin{proposition} \label{thm:GN}
    Let $p, q, r, s \in [1,+\infty]$, $0 \leq j < l \in \N$ and $\alpha \in (0,1)$ such that
    \begin{equation} \label{eq:GNS-condition}
        \frac{j}{l} \leq \alpha
        \quad \text{and} \quad 
        \frac{1}{p}=j+\left( \frac{1}{r}-l\right)\alpha + \frac{1-\alpha}{q}.
    \end{equation}
    There exists $C > 0$ such that, for every $t > 0$ and $\phi \in \CC^\infty([0,t];\R)$,
    \begin{equation} \label{eq:GNS-estimate}
        \| D^j \phi \|_{L^p} \leq C \| D^l \phi \|_{L^r}^\alpha \|\phi\|_{L^q}^{1-\alpha}  + C t^{\frac{1}{p}-j-\frac{1}{s}} \|\phi\|_{L^s}.
    \end{equation}
\end{proposition} 

\begin{remark}
    For functions on bounded intervals, adding the lower-order term in the right-hand side of~\eqref{eq:GNS-estimate} is mandatory (see \cite[item 5, p.\ 126]{nirenberg}).
    To obtain the dependency of the constant on $t > 0$, one uses scaling arguments to work within a fixed domain, say $[0,1]$.
\end{remark}

\subsection{A consequence of the Jacobi identity}
\label{s:jacobi}

The following straightforward consequences of the Jacobi identity will be useful to compute the expansion of brackets of two elements of $\Bs$ (see \cref{def:0nu} for the notation $0^{\nu}$).

\begin{lemma}
    Using the notation $0^\nu$ of \cref{def:0nu}, the following expansions hold.
    \begin{enumerate}
    \item For any $\nu \in \N$ and any $a,b \in \mathcal{L}(X)$,
        \begin{equation} \label{eq:jacobi.rtl}
            [a, b 0^\nu] = \sum_{\nu'=0}^{\nu} \binom{\nu}{\nu'} (-1)^{\nu'} [a 0^{\nu'}, b] 0^{\nu-\nu'}.
        \end{equation}
    \item For any $\nu \in \N^*$, there exist coefficients $\alpha^\nu_j \in \Z$ for $1 \leq 2j+1 \leq \nu$, such that, for any $b \in \mathcal{L}(X)$,
        \begin{equation} \label{eq:jacobi.balance}
            [b, b0^\nu] = \sum_{1 \leq 2j+1 \leq \nu} \alpha^\nu_j [b0^j, b0^{j+1}] 0^{\nu-2j-1}.
        \end{equation}
    \end{enumerate}
\end{lemma}

\begin{proof}
    The validity of \eqref{eq:jacobi.rtl} for any $a,b$ can be proved by induction on $\nu\in\N$, the heredity relies on the Jacobi identity and the binomial relation $\binom{\nu-1}{\nu'}+\binom{\nu-1}{\nu'-1}=\binom{\nu}{\nu'}$ for $\nu'=1,\dots,\nu-1$. 
    The validity of \eqref{eq:jacobi.balance} for any $b$ can be proved by induction on $\nu\in\N^*$; the Jacobi relation leads to $\alpha_j^\nu=\alpha_j^{\nu-1} - \alpha_{j-1}^{\nu-2}$.
\end{proof}

\section{Sussmann's and Stefani's obstructions}
\label{s:Suss-Stef}

The goal of this section is to give a new proof of \cref{Thm_Stefani}, within the framework of the unified approach proposed in this paper, as a consequence of the following more precise statement.

\begin{theorem} \label{thm:stefani-precise}
    Assume that \eqref{Stefani} does not hold. 
    Let $k \in \N^*$ such that
    \begin{equation} \label{eq:not-stefani}
        \ad_{f_1}^{2k}(f_0)(0) \notin S_{\intset{1, 2k-1}}(f)(0).
    \end{equation}
    Then system \eqref{syst} has a drift along $\ad_{f_1}^{2k}(f_0)(0)$, parallel to $S_{\intset{1, 2k-1}}(f)(0)$, as $(t,\|u\|_{W^{-1,\infty}}) \to 0$.
\end{theorem}

\subsection{Dominant part of the logarithm}

\begin{lemma}
    Let $k \in \N^*$ such that \eqref{eq:not-stefani} holds.
    Let $\mathbb{P}$ be a component along $\ad_{f_1}^{2k}(f_0)(0)$, parallel to $S_{\intset{1,2k-1}}(f)(0)$.
    Then
    \begin{equation} \label{eq:stefani-z2k}
        \mathbb{P} \cZ{2k}(t,f,u)(0) = \xi_{\ad^{2k}_{X_1}(X_0)}(t,u)
        + O\left(|u_1(t)|^{2k} + t^{\frac{1}{2k-1}} \|u_1\|_{L^{2k}}^{2k}\right).
    \end{equation}
\end{lemma}

\begin{proof}
    We intend to apply \cref{p:PZM-XI} with $M \gets 2k$, $L \gets 2k+2$, $\bb \gets \ad^{2k}_{X_1}(X_0)$ and $\mathcal{N} \gets \Bs_{\intset{1,2k-1}}$, so that \eqref{eq:stefani-z2k} will follow from \eqref{eq:PZM-xibb-OXi}, for the appropriate choice of $\Xi(t,u)$.
    Let us check that the required estimates are satisfied.
    
    \step{Estimates of other coordinates of the second kind}
    Let $b \in \Bs_{\intset{1,2k}}$ such that $b \notin \mathcal{N} \cup \{ \bb \}$.

    Since $\mathcal{N} = \Bs_{\intset{1,2k-1}}$, one has $n_1(b) = 2k$ and $n_0(b) \geq 2$.
    Hence $|b| \geq 2k+2$.
    By \eqref{eq:xib-u1-Lk} of \cref{p:xib-u1-Lk}, estimate \eqref{eq:xib-XI-L} holds with $\sigma = 2k+2$ and $\Xi(t,u) = t \|u_1\|_{L^{2k}}^{2k}$.
    
    \step{Estimates of cross terms}
    Let $q \geq 2$, $b_1 \geq \dotsb \geq b_q \in \Bs \setminus \{ X_0 \}$ such that $n_1(b_1) + \dotsb + n_1(b_q) \leq 2k$ and $\supp \mathcal{F}(b_1, \dotsc, b_q) \not \subset \mathcal{N}$.

    For each $i \in \intset{1,q}$, 
    \begin{itemize}
        \item if $b_i = X_1$, then
        \begin{equation}
            |\xi_{b_i}(t,u)| = |u_1(t)|,
        \end{equation}
        so \eqref{eq:xib-othercross-XI-L} holds with $\sigma_i = 1$ and $\alpha_i = 1/(2k) = n_1(b_i) / 2k$ and $\Xi(t,u) = |u_1(t)|^{2k}$.
        
        \item otherwise, $|b_i| \geq 1 + n_1(b_i)$ and, by \eqref{eq:xib-u1-Lk} of \cref{p:xib-u1-Lk} and Hölder's inequality,
        \begin{equation} \label{eq:stefani-xi-bi-u1-Lk}
            | \xi_{b_i}(t,u) | 
            \leq \frac{(ct)^{|b_i|}}{|b_i|!}
            t^{-1-n_1(b_i)}
            \|u_1\|_{L^{n_1(b_i)}}^{n_1(b_i)}
            \leq 
            \frac{(ct)^{|b_i|}}{|b_i|!}
            t^{-\sigma_i}
            \left( t^{\frac{1}{\alpha_i}-1} \| u_1 \|_{L^{2k}}^{2k} \right)^{\alpha_i}
        \end{equation}
        with $\sigma_i = 1 + n_1(b_i)$ and $\alpha_i = n_1(b_i) / (2k)$.
        Since $q \geq 2$, $n_1(b_i) \leq 2k-1$.
        Thus $\frac{1}{\alpha_i}-1 \geq \frac{1}{2k-1}$ and, assuming $t \leq 1$,
        \begin{equation}
            t^{\frac{1}{\alpha_i}-1} \|u_1\|_{L^{2k}}^{2k} \leq t^{\frac{1}{2k-1}} \|u_1\|_{L^{2k}}^{2k},
        \end{equation}
        so \eqref{eq:xib-othercross-XI-L} holds with $\Xi(t,u) = t^{\frac1{2k-1}} \|u_1\|_{L^{2k}}^{2k}$.
    \end{itemize}
    Since $\mathcal{N} = \Bs_{\intset{1,2k-1}}$, one has $n_1(b_1) + \dotsb + n_1(b_q) = 2k$.
    Hence $\alpha = \alpha_1 + \dotsb + \alpha_q = 1$.
\end{proof}

\subsection{Vectorial relation}

\begin{lemma} \label{p:stefani-vectorial}
    Let $k\in\N^*$ such that \eqref{eq:not-stefani} holds.
    Then, $f_1(0) \neq 0$.
\end{lemma}

\begin{proof}
    By contradiction, if $f_1(0) = 0$, since $f_0(0) = 0$, all iterated Lie brackets of $f_0$ and $f_1$ vanish so $\ad^{2k}_{f_1}(f_0)(0) = 0 \in S_{\intset{1,2k-1}}(f)(0) = \{0\}$.
\end{proof}

\subsection{Closed-loop estimate}

\begin{lemma} \label{lem:stefani-loop}
    Assume that $f_1(0) \neq 0$.
    Then,
    \begin{equation} \label{eq:stefani-loop}
        |u_1(t)| = O\left(|x(t;u)| + \|u_1\|_{L^1}\right).
    \end{equation}
\end{lemma}

\begin{proof}
    This estimate is proved in \cite[Proposition 162]{P1}.
    For the sake of self-containedness, and as an illustration of the approach used in the following sections, let us give another proof.
    
    Let $\mathbb{P}$ be a component along $f_1(0)$, parallel to the null vector space $\{ 0 \}$.
    By \cref{p:PZM-XI} with $M \gets 1$, $L \gets 2$, $\bb \gets X_1$ and $\mathcal{N} \gets \emptyset$, \eqref{eq:PZM-xibb-OXi} entails that
    \begin{equation} \label{eq:pz1}
        \mathbb{P} \mathcal{Z}_1(t,f,u)(0) = u_1(t) + O(\|u_1\|_{L^1}).
    \end{equation}
    Indeed, on the one hand, for every $b \in \Bs_1 \setminus \{ X_1 \}$, by \eqref{bound-xiMj/J0} with $(p,j_0) \gets (1,1)$, one has $|b| \geq 2$ and 
    \begin{equation}
        |\xi_b(t,u)| \leq \frac{(ct)^{|b|}}{|b|!} t^{-2} \|u_1\|_{L^1},
    \end{equation}
    so \eqref{eq:xib-XI-L} holds with $\sigma = 2$ and $\Xi(t,u) = \|u_1\|_{L^1}$.
    On the other hand, we don't need to estimate any cross terms because, when $q \geq 2$ and $b_1, \dotsc, b_q \in \Bs\setminus\{X_0\}$, $n_1(b_1) + \dotsb + n_1(b_q) > 1$.
    
    By \cref{thm:Key_1} with $M \gets 1$,
    \begin{equation} \label{eq:xz1}
        x(t;u) = \mathcal{Z}_1(t,f,u)(0) + O\left(\|u_1\|_{L^2}^2+|x(t;u)|^2\right).
    \end{equation}
    Then \eqref{eq:stefani-loop} follows from \eqref{eq:pz1}, \eqref{eq:xz1} and the small-state estimate of \cref{p:small-state}.
\end{proof}

\subsection{Interpolation inequality}

\begin{lemma}
    For $t > 0$ and $u \in \lone$,
    \begin{equation} \label{eq:stefani-interpol}
        \| u_1 \|_{L^{2k+1}}^{2k+1} 
        \leq \| u_1 \|_{L^\infty} \| u_1 \|_{L^{2k}}^{2k}.
    \end{equation}
\end{lemma}

\subsection{Proof of the presence of the drift}
\label{s:stefani-proof}

\begin{proof}[Proof of \cref{thm:stefani-precise}]
    Let $\mathbb{P}$ be a component along $\ad_{f_1}^{2k}(f_0)(0)$ parallel to $S_{\intset{1,2k-1}}(f)(0)$.
    By \cref{thm:Key_1},
    \begin{equation}
        x(t;u) = \cZ{2k}(t,f,u)(0) + O\left( \|u_1\|_{L^{2k+1}}^{2k+1} + |x(t;u)|^{1+\frac{1}{2k}} \right)
    \end{equation}
    and, by \eqref{eq:stefani-z2k} and \eqref{eq:def:coord2},
    \begin{equation}
        \mathbb{P} \cZ{2k}(t,f,u)(0)
        = \frac{1}{(2k)!} \int_0^t u_1^{2k}
        + O\left(|u_1(t)|^{2k} + t^{\frac{1}{2k-1}} \|u_1\|_{L^{2k}}^{2k}\right).
    \end{equation}
    Moreover, by the closed-loop estimate \eqref{eq:stefani-loop} and Hölder's inequality,
    \begin{equation}
        |u_1(t)|^{2k} = O\left(|x(t;u)|^{2k} + t^{2k-1} \|u_1\|_{L^{2k}}^{2k}\right).
    \end{equation}
    Gathering these equalities and \eqref{eq:stefani-interpol} yields
    \begin{equation}
        \mathbb{P} x(t;u) =
        \int_0^t \frac{u_1^{2k}}{(2k)!} + O \left( \left(t^{\frac{1}{2k-1}}+\|u_1\|_{L^\infty}\right) \int_0^t u_1^{2k} +|x(t;u)|^{1+\frac{1}{2k}} \right),
    \end{equation}
    which establishes the presence of drift, in the sense of \cref{Def:drift},  along $\ad_{f_1}^{2k}(f_0)(0)$, parallel to $S_{\intset{1, 2k-1}}(f)(0)$, as $(t,\|u\|_{W^{-1,\infty}}) \to 0$.
\end{proof}

\section{New loose quadratic obstructions, conjectured by Kawski}
\label{s:Quad}

We start by proving \cref{Thm:Kawski_Wm} when $m \in \N$, as a consequence of the following more precise statement (the proof of \cref{Thm:Kawski_Wm} for $m=-1$ will be done in \cref{s:m=-1}).
The case $k = 1$ is already covered by \cref{thm:stefani-precise} so we can assume without loss of generality in this section that $k \geq 2$ (hence $\pi(k,m) \geq 2$, by \eqref{eq:pikm}).

\begin{theorem} \label{Thm:Kawski_Wm_precis-pos}
    Let $m\in\N$ and $k \geq 2$. 
    We assume $k$ is the smallest integer for which
    \begin{equation} \label{eq:Wk-Spikm}
         f_{W_k}(0)\notin S_{\llbracket 1 , \pi(k,m) \rrbracket \setminus\{2\}}(f)(0),
    \end{equation}
    where $\pi(k,m)$ is defined in \eqref{eq:pikm}.
    Then system \eqref{syst} has a drift along $f_{W_k}(0)$, parallel to 
    $S_{\intset{1,\pi(k,m)}\setminus\{2\}}(f)(0)$, 
    as $(t,t^{-\alpha} \|u\|_{W^{m,\infty}}) \rightarrow 0$
    where $\alpha=\frac{\pi(k,0)-\pi(k,m)}{\pi(k,m)-1}$.
\end{theorem}

\begin{remark} \label{rk:m>0-stlc}
    When $m>0$, the smallness assumption on the control in \cref{Thm:Kawski_Wm_precis-pos} depends on~$t$. 
    See \cref{rk:m=0} for a comment on the source of this dependency.
    For example, with $k = 2$ and $m = 1$, one gets a limit of the form $(t, t^{-1} \|u\|_{W^{1,\infty}}) \to 0$.
    One checks that \cref{lem:drift-stlc} still holds for such limits, so that \cref{Thm:Kawski_Wm_precis-pos} indeed denies $W^{m,\infty}$-STLC in the sense of \cref{Def:WmSTLC}.
    This time-dependency is not technical (see \cite[Section 2.4.4]{JDE} for a counter-example).
\end{remark}

\subsection{A previous result on a prototype example} \label{Subsec:Ex_Wk}

In \cite[System (32)]{zbMATH04154295}, Kawski considers the system
\begin{equation} \label{Ex:WkvsX1mX0}
    \begin{cases}
        \dot{x}_1=u \\
        \dot{x}_2=x_1 \\
        \dots \\
        \dot{x}_{k}=x_{k-1}\\
        \dot{x}_{k+1}=x_{k}^2-\lambda x_1^p,
    \end{cases}
\end{equation}
where $\lambda>0$. 
Written in the form \eqref{syst}, this system satisfies
\begin{equation}
    f_{M_{j-1}}(0)=e_j \text{ for } j \in \intset{1,k}, \quad f_{W_{k}}(0)=2 e_{k+1}, \quad 
    f_{\ad_{X_1}^p(X_0)}(0)=-\lambda p! e_{k+1}
\end{equation}
and $f_b(0)=0$ for any other $b \in \Bs$.
In \cite[Proposition 5.1]{zbMATH04154295}, Kawski proves that, if $p\geq 2^{k+1}$ then the system \eqref{Ex:WkvsX1mX0} is not $L^\infty$-STLC. 
This result can be recovered by applying \cref{Thm:Kawski_Wm_precis-pos} to system \eqref{Ex:WkvsX1mX0} with $m \gets 0$. 
Indeed, $p \geq 2^{k+1} > 2k-1=\pi(k,0)$.

With respect to this previous result, \cref{Thm:Kawski_Wm_precis-pos} can be viewed as an improvement in the following directions:
\begin{itemize}
    \item any perturbation in $\Bs_{\llbracket p, \infty \llbracket}$ is allowed (not only $\ad_{X_1}^p(X_0)$),
    \item as correctly conjectured in \cite[section 2.4, p.\ 63]{MR2635388}, the critical threshold for $L^\infty$-STLC is proved to be $2k-1$ (instead of $2^{k+1}-1$ obtained in \cite[Proposition 5.1]{zbMATH04154295}),
    \item other regularity scales $W^{m,\infty}$ for $m > 0$ are included.
\end{itemize}

\subsection{Dominant part of the logarithm}
\label{subsec:DPL}

\begin{lemma} \label{Lem:DPL}
    Let $m\in\N$ and $k \geq 2$. 
    Assume that $k$ is the minimal value for which \eqref{eq:Wk-Spikm} holds. 
    Let $\mathbb{P}$ be a component along $f_{W_k}(0)$, parallel to $S_{\intset{1,\pi(k,m)} \setminus \{2\}}(f)(0)$.
    Then
    \begin{equation} \label{eq:quad-zpi}
        \mathbb{P} \cZ{\pi(k,m)}(t,f,u)(0) = \xi_{W_k}(t,u)
        + O\left( |(u_1,\dotsc,u_k)(t)|^2 + t \|u_k\|_{L^2}^2 \right).
    \end{equation}
\end{lemma}

\begin{proof}
    By minimality of $k$, for every $j \in \intset{1,k-1}$,
    \begin{equation}
        f_{W_j}(0) \in S_{\intset{1,\pi(j,m)} \setminus \{2\}}(f)(0) \subset S_{\intset{1,\pi(k,m)}\setminus\{2\}}(f)(0),
    \end{equation}
    since $\pi(\cdot,m)$ is non-decreasing.
    Since $S_{\intset{1,\pi(k,m)}\setminus\{2\}}(X)$ is stable by right bracketing with $X_0$, one also has
    \begin{equation}
        f_{W_{j,\nu}}(0) \in S_{\intset{1,\pi(k,m)}\setminus\{2\}}(f)(0),
    \end{equation}
    for every $j \in \intset{1,k-1}$ and $\nu \geq 0$.
    Hence $S_{\intset{1,\pi(k,m)}\setminus\{2\}}(f)(0) = \mathcal{N}(f)(0)$ where
    \begin{equation} \label{eq:quad-N}
        \mathcal{N} := \Bs_{\intset{1,\pi(k,m)}\setminus\{2\}} \cup \left\{ W_{j,\nu} ; j \in \intset{1,k-1}, \nu \in \N \right\}. 
    \end{equation}
    We intend to apply \cref{p:PZM-XI} with $M \gets \pi(k,m)$, $L \gets 2k+2$, $\bb \gets W_k$ and $\mathcal{N}$ as in \eqref{eq:quad-N}, so that \eqref{eq:quad-zpi} will follow from \eqref{eq:PZM-xibb-OXi}, for the appropriate choice of $\Xi(t,u)$.
    Let us check that the required estimates are satisfied.
    
    \step{Estimates of other coordinates of the second kind} 
    Let $b \in \Bs_{\intset{1,\pi(k,m)}}$ such that $b \notin \mathcal{N} \cup \{ \bb \}$.
    
    By definition \eqref{eq:quad-N}, one necessarily has $n_1(b) = 2$ and $b = W_{j,\nu}$ with either $j > k$ or ($j = k$ and $\nu \geq 1$).
    By estimate \eqref{bound-xiWjnu/j0} with $(p,j_0) \gets (1,k)$, \eqref{eq:xib-XI-L} holds with $\sigma = 2k+2$ and
    \begin{equation}
        \Xi(t,u) := t \|u_k\|_{L^2}^2.
    \end{equation}

    \step{Estimates of other cross terms} \label{step:quad-dom-cross}
    Let $q \geq 2$, $b_1 \geq \dotsb \geq b_q \in \Bs$ such that $n_1(b_1) + \dotsb + n_1(b_q) \leq \pi(k,m)$ and $\supp \mathcal{F}(b_1, \dotsc, b_q) \not \subset \mathcal{N}$.
    
    \medskip
    \noindent We start with preliminary estimates.
    \begin{itemize}
        \item If $b_i = M_j$ for some $j \in \intset{0,k-1}$, by \eqref{xi_S1},
        \begin{equation}
            |\xi_{b_i}(t,u)| = |u_{j+1}(t)| = \frac{t^{|b_i|}}{|b_i|!} t^{-(j+1)} (j+1)! |u_{j+1}(t)|
        \end{equation}
        so \eqref{eq:xib-othercross-XI-L} holds with $\sigma_i = j+1$, $\alpha_i = 1/2$ and $\Xi(t,u) = |(u_1,\dotsc,u_k)(t)|^2$.
        
        \item If $b_i = M_j$ for $j \geq k$, by \eqref{bound-xiMj/J0} (with $(p,j_0) \gets (2,k)$), \eqref{eq:xib-othercross-XI-L} holds with $\sigma_i = k+1$, $\alpha_i = 1/2$ and $\Xi(t,u) = t\|u_k\|_{L^2}^2$.
    \end{itemize}
    Since $\supp \mathcal{F}(b_1, \dotsc, b_q) \not \subset \mathcal{N}$, one has $q = 2$ and $b_1, b_2 \in \Bs_1$.
    So the previous estimates apply and $\alpha_1 = \alpha_2 = 1/2$ so $\alpha_1+\alpha_2 = 1$.
\end{proof}

\subsection{Vectorial relations}
\label{subsec:VR}

\begin{lemma} \label{p:quad-vectors}
    Let $k\in\N^*$, 
    $\pi:\N^* \to \llbracket 1 , \infty \rrbracket$
    be a non-decreasing map and 
    $\vartheta:\N^* \to \llbracket 1 , \infty \rrbracket$ be defined by $\vartheta(k)=\max\{1;\lfloor\frac{\pi(k)}{2}\rfloor\}$, with the convention $\vartheta(k)=+\infty$ when $\pi(k)=+\infty$.
    Assume that $k$ is the minimal value for which $f_{W_k}(0) \notin S_{\intset{1 , \pi(k)} \setminus \{2\}}(f)(0)$.
    Then,
    \begin{enumerate}
        \item the vectors $f_{M_0}(0),\dots,f_{M_{k-1}}(0)$ are linearly independent,
        \item if $\vartheta(k) \geq 2$, then $\vect\{ f_{M_0}(0),\dots,f_{M_{k-1}}(0) \} \cap S_{\intset{2, \vartheta(k)}}(f)(0) = \{0\}$.
    \end{enumerate}
\end{lemma}

\begin{proof}
    Let $H_0:=f_0'(0)$. 
    Since $f_0(0) = 0$, for any $b \in \Br(X)$, $f_{(b,X_0)}(0) = H_0 f_{b}(0)$. 
    Thus, for each $A \subset \N$, the space $S_A(f)(0)$ is stable by left multiplication by $H_0$. 
    In particular, by minimality of~$k$, for each $l \in \intset{1,k-1}$ and $\nu \in \N$,
    \begin{equation} 
        f_{W_{l,\nu}}(0)=H_0^\nu f_{W_{l,0}}(0) \in S_{\intset{1,\pi(l)} \setminus\{2\}}(f)(0) \subset S_{\intset{1,\pi(k)} \setminus \{2\}}(f)(0),
    \end{equation}
    where the last inclusion results from the monotony of $\pi$.
    Thus,
    \begin{equation} \label{S22k-2-pik}
        S_{2,\intset{1,2k-2}}(f)(0) \subset S_{\intset{1,\pi(k)} \setminus \{2\}}(f)(0).
    \end{equation}

    \step{Proof of statement 1} 
    By contradiction, assume that there exists $(\beta_0,\dots,\beta_{k-1}) \in \R^k \setminus\{0\}$ such that $\beta_0 f_{M_0}(0)+\dots+\beta_{k-1} f_{M_{k-1}}(0)=0$, i.e.\ $f_{B_1}(0)=0$ where $B_1:=\beta_{k-1} M_{k-1} + \dots +\beta_0 M_0$. One may assume that $\beta_{k-1}\neq 0$; otherwise replace $B_1$ by $\ad_{X_0}^{k-1-K}(B_1)$ where $K=\max\{j;\beta_j\neq 0\}$. 
    By linearity, one may assume $\beta_{k-1}=1$. Then $f_{B_2}(0)=0$ where 
    \begin{equation}
        B_2 :=\ad_{B_1}^2(X_0)=[M_{k-1}+\dots+\beta_0 M_0, M_k+\dots+\beta_0 M_1] = W_{k} - B_3,
    \end{equation}
    where $B_3 \in S_{2,\intset{1,2k-2}}(X)$. 
    Finally, by \eqref{S22k-2-pik}, $f_{W_k}(0)=f_{B_3}(0)\in S_{\intset{1,\pi(k)} \setminus \{2\}}(f)(0)$, which contradicts \eqref{eq:Wk-Spikm}.
 
    \step{Proof of statement 2}
    By contradiction, assume that $\vartheta(k)\geq 2$ and that there exists $B \in S_{\llbracket 2, \vartheta(k) \rrbracket }(X)$ and $(\gamma_0,\dots,\gamma_{k-1}) \in \R^k \setminus \{0\}$ such that $f_{B_4}(0)=0$ where $B_4:=\gamma_{k-1} M_{k-1}+\dots+\gamma_0 M_0 + B$. 
    One may assume $\gamma_{k-1}=1$; otherwise, replace $B_4$ by $\ad_{X_0}^{k-1-K}(B_4)$ where $K=\max\{j;\gamma_j\neq 0\}$ and renormalize. 
    Then $f_{B_5}(0)=0$ where
    \begin{equation}
        \begin{split}
            B_5
            :=\ad_{B_4}^2(X_0)
            & =[ M_{k-1}+\dots+\gamma_0 M_0 + B,M_{k}+\dots+\gamma_0 M_1 + [B,X_0] ] \\
            & \in W_k + S_{2,\intset{1,2k-2}}(X) + S_{\intset{3,2\vartheta(k)}}(X).
        \end{split}
    \end{equation}
    This fact and \eqref{S22k-2-pik} contradict 
    \eqref{eq:Wk-Spikm} because $2\vartheta(k)\leq \pi(k)$.
\end{proof}

\subsection{Closed-loop estimate}
\label{subsec:CLE}

\begin{lemma} \label{p:quad-loop}
    Let $k\in\N^*$, $\pi:\N^* \to \N^*$ be a non-decreasing map and $\vartheta:\N^* \to \N^*$ be defined by $\vartheta(k)=\max\{1;\lfloor\frac{\pi(k)}{2}\rfloor\}$.
    Assume that $k$ is the minimal value for which $f_{W_k}(0) \notin S_{\intset{1 , \pi(k)} \setminus \{2\}}(f)(0)$.
    Then,
    \begin{equation} \label{eq:quad-loop}
        |(u_1,\dotsc,u_k)(t)| = O \left( |x(t;u)| + \|u_1\|_{L^{\vartheta(k)+1}}^{\vartheta(k)+1} + t^{\frac 12} \|u_k\|_{L^2} \right).
    \end{equation}
\end{lemma}

\begin{proof}
    By \cref{thm:Key_1} with $M \gets \vartheta(k)$,
    \begin{equation} \label{eq:quad-CL-1}
        x(t;u) = \cZ{\vartheta(k)}(t,f,u)(0) + O \left( \|u_1\|_{L^{\vartheta(k)+1}}^{\vartheta(k)+1} + |x(t;u)|^{1+\frac{1}{\vartheta(k)}} \right).
    \end{equation}
    Let $i \in \intset{0,k-1}$.
    By \cref{p:quad-vectors}, we can consider $\mathbb{P}$, a component along $f_{M_i}(0)$, parallel to $\mathcal{N}(f)(0)$ where $\mathcal{N} := (\{ M_0, \dotsc, M_{k-1} \} \setminus M_i) \cup \Bs_{\intset{2,\vartheta(k)}}$.
    We intend to apply \cref{p:PZM-XI} with $M \gets \vartheta(k)$, $L \gets k+1$, $\bb \gets M_i$ and $\mathcal{N}$ as above, so that \eqref{eq:PZM-xibb-OXi}, for the appropriate choice of $\Xi(t,u)$, will yield
    \begin{equation} \label{eq:quad-CL-2}
        \mathbb{P} \cZ{\vartheta(k)}(t,f,u)(0) = u_{i+1}(t) + O\left(t^{\frac 12}\|u_k\|_{L^2}\right).
    \end{equation}
    Then, combining \eqref{eq:quad-CL-1} and \eqref{eq:quad-CL-2} concludes the proof of \eqref{eq:quad-loop}.
    Let us check that the required estimates are satisfied.

    \step{Estimates of other coordinates of the second kind}
    Let $b \in \Bs_{\intset{1,\vartheta(k)}}$ such that $b \notin \mathcal{N} \cup \{ \bb \}$.
    
    By choice of $\mathcal{N}$, one has necessarily $n_1(b) = 1$. 
    Then $b = M_j$ for $j \geq k$.
    Thus, by \eqref{bound-xiMj/J0} (with $(p,j_0) \gets (2,k)$), $|b| \geq k+1$ and \eqref{eq:xib-XI-L} holds with $\sigma = k+1$ and $\Xi(t,u) := t^{\frac 12} \|u_k\|_{L^2}$.
    
    \step{Estimates of cross terms}
    Let $q \geq 2$, $b_1 \geq \dotsb \geq b_q \in \Bs \setminus \{ X_0 \}$ such that $n_1(b_1) + \dotsb + n_1(b_q) \leq \vartheta(k)$ and $\supp \mathcal{F}(b_1, \dotsc, b_q) \not \subset \mathcal{N}$.
    
    By construction of $\mathcal{N}$, there is no such cross term.
\end{proof}

\subsection{Interpolation inequality}
\label{subsec:II}

\begin{lemma}
    Let $m \in \N$, $k \geq 2$ and $\pi := \pi(k,m) \geq 2$ as in \eqref{eq:pikm}. 
	There exists $C > 0$ such that, for every $t > 0$ and $u \in \lone$,
	\begin{equation}
		\label{eq:quad-interpol}
		\|u_1\|_{L^{\pi+1}}^{\pi+1} \leq C \left(
		\|u_1\|_{L^\infty}^{\pi+1-p} \|u\|_{W^{m,\infty}}^{p}
		+ 
		t^{\pi+1-2k} \|u\|_{L^\infty}^{\pi-1}
		\right) \|u_k\|_{L^2}^2,
	\end{equation}
	where $p := (2m+2k)/(m+1)$ satisfies $p \leq \pi+1$.
\end{lemma}

\begin{proof}
	By \cref{thm:GN} with $\phi \gets u_k$, $(p,q,r,s) \gets (p,2,\infty,2)$, $(j,l) \gets (k-1,m+k)$, $\alpha \gets (p-2)/p$, we obtain
	\begin{equation}
	    \|u_1\|_{L^p}^p \leq C \|u^{(m)}\|_{L^\infty}^{p-2} \|u_k\|_{L^2}^2 + C t^{1-pk+\frac{p}{2}} \|u_k\|_{L^2}^p.
	\end{equation}
	By Hölder's inequality,
	\begin{equation}
	    \|u_k\|_{L^2}^{p-2} \leq t^{\left(\frac{1}{2}+k\right)\left(p-2\right)} \|u\|_{L^\infty}^{p-2}.
	\end{equation}
    Moreover, by \eqref{eq:pikm},
    \begin{equation}
        \pi(k,m)+1 \geq \frac{2k+m-1}{m+1} + 1
        = \frac{2k+2m}{m+1} = p,
    \end{equation}
    and this concludes the proof of \eqref{eq:quad-interpol}, writing 
    \begin{equation}
        \|u_1\|_{L^{\pi+1}}^{\pi+1} \leq \|u_1\|_\infty^{\pi+1-p} \|u_1\|_{L^p}^p
    \end{equation}
    and $\|u_1\|_{L^\infty} \leq t \|u\|_{L^\infty}$.
\end{proof}

\subsection{Proof of the presence of the drift for \texorpdfstring{$m\geq0$}{m >= 0}}
\label{subsec:D}

\begin{proof}[Proof of \cref{Thm:Kawski_Wm_precis-pos}]
    Let $\mathbb{P}$ be a component along $f_{W_k}(0)$ parallel to $S_{\intset{1,\pi(k,m)}\setminus\{2\}}(f)(0)$.
    Let $M := \pi(k,m)$.
    Let $\vartheta:=\max\{1;\lfloor\frac{\pi(k,m)}{2}\rfloor\}$.
    By \cref{thm:Key_1},
    \begin{equation}
        x(t;u) = \cZ{M}(t,f,u)(0) + O\left( \|u_1\|_{L^{M+1}}^{M+1} + |x(t;u)|^{1+\frac{1}{M}} \right),
    \end{equation}
    where, by \eqref{eq:quad-zpi} and \eqref{xi_Wjnu},
    \begin{equation}
        \mathbb{P} \cZ{M}(t,f,u)(0)
        = \frac{1}{2} \int_0^t u_k^2 + O\left( |(u_1,\dotsc,u_k)(t)|^2 + t \|u_k\|_{L^2}^2 \right).
    \end{equation}
    Moreover, by the closed-loop estimate \eqref{eq:quad-loop},
    \begin{equation}
    	|(u_1,\dotsc,u_k)(t)|^2 = O\left(
    		|x(t;u)|^2
 			+ \|u_1\|_{L^{\vartheta+1}}^{2\vartheta+2}
		    + t \|u_k\|_{L^2}^2	
    	 \right).
    \end{equation}
    By definition of $\vartheta$, one has $2(\vartheta+1) \geq \pi(k,m)+1$.
    Hence, in particular, 
    \begin{equation}
        \|u_1\|_{L^{\vartheta+1}}^{2\vartheta+2} = O\left( \|u_1\|_{L^{M+1}}^{M+1} \right).
    \end{equation}
    Gathering these equalities and the interpolation estimate \eqref{eq:quad-interpol} yields
    \begin{equation}
        \mathbb{P} x(t;u) = \frac 12 \int_0^t u_k^2
        + O \left( \left(t + (1+ t^{\pi(k,m)+1-2k}) \|u\|_{W^{m,\infty}}^{\pi(k,m)-1} \right) \|u_k\|_{L^2}^2 +|x(t;u)|^{1+\frac{1}{M}} \right).
    \end{equation}
    This implies, in the sense of \cref{Def:drift}, a drift along $f_{W_k}(0)$, parallel to $S_{\intset{1,\pi(k,m)}\setminus\{2\}}(f)(0)$, 
    as $(t,t^{-\alpha} \|u\|_{W^{m,\infty}}) \rightarrow 0$
    where $\alpha=\frac{2k-1-\pi(k,m)}{\pi(k,m)-1} =
    \frac{\pi(k,0)-\pi(k,m)}{\pi(k,m)-1}$.
\end{proof}

\begin{remark} \label{rk:m=0}
When $m=0$, the smallness assumption on the control does not depend on the final time $t$ (because $\alpha=0$). When $m > 0$, the dependence on time of the smallness assumption on the control  stems from the second term in the right-hand side of the Gagliardo--Nirenberg inequality of \cref{thm:GN}. For appropriate classes of functions, for instance $\phi \in W^{m,\infty}_0$, the Gagliardo--Nirenberg inequality holds without this second term. Thus, for controls $u \in W^{m,\infty}_0$, the argument above proves a drift as $(t,\|u\|_{W^{m,\infty}_0}) \rightarrow 0$.
\end{remark}

\subsection{Optimality of the functional framework} \label{Subsec_Optim_Wk}

We illustrate the optimality of the functional framework given in \cref{Thm:Kawski_Wm_precis-pos} by an example in the case $m = 0$ and $k = 2$.
In this case, the condition $f_{W_2}(0) \in S_{\{1,3\}}(f)(0)$ is necessary for $L^\infty$-STLC, but not for a different small-time local controllability notion involving large enough controls in $L^\infty$ (instead of arbitrarily small controls in $L^\infty$), called $\rho$-bounded-SLTC in \cref{s:intro-def}. 
In this sense, our result is optimal.

To prove this claim, let us consider the following system (introduced in \cite[Example 5.2]{stefani1985polynomial}): 
\begin{equation} \label{eq:ex:W2vsQ111}
    \begin{cases}
        \dot{x}_1 = u \\
        \dot{x}_2 = x_1 \\
        \dot{x}_3 = x_2^2 - x_1^4.
    \end{cases}
\end{equation}
Written in the form \eqref{syst}, this system satisfies
\begin{equation}
    f_{M_0}=e_1, \quad 
    f_{M_1}(0)=e_2, \quad 
    f_{W_2}(0)=2 e_3, \quad 
    f_{Q_{1,1,1}}(0)= - 24 e_3    
\end{equation}
and $f_b(0)=0$ for any $b \in \Bs \setminus\{ M_0, M_1, W_2, Q_{1,1,1}\}$.
Thus $f_{W_1}(0) \in \mathcal{S}_1(f)(0)$ and
$f_{W_2}(0) \notin S_{\{1,3\}}(f)(0)$. 
By \cref{Thm:Kawski_Wm}, this system is not $L^\infty$-STLC, i.e.\ locally controllable in small time with $L\infty$-small controls. 
By \cref{Thm:Kawski_Wm_precis-pos}, solutions associated to controls small in $L^\infty$ cannot reach in small time targets of the form $-\beta e_3$ with $\beta>0$.

\bigskip

In \cite[Example 5.1 and p.\ 452]{zbMATH04154295}, Kawski claims that this system is STLC with controls large enough in $L^\infty$. 
This can also be deduced from the arguments given in \cite[Example 5.2]{stefani1985polynomial} by a scaling argument.
Let us indeed construct explicit controls (large in $L^\infty$) achieving a motion along $-e_3$.
Let $\varphi \in \CC^\infty_c((0,1);\R) \setminus \{0\}$ and $A>0$ large enough such that
\begin{equation}
    C:= -\int_0^1 \varphi^2 + A^2 \int_0^1 (\varphi')^4 > 0.
\end{equation}
Let $t>0$ and $u \in \lone$ be defined by $u(s) := A \varphi''\left(s/t\right)$.
Then $u_1(s)=A t \varphi' (s/t)$ and $u_2(s)=A t^2 \varphi (s/t)$.
Thus
\begin{equation}
    x_3(t) = \int_0^t u_2^2 - u_1^4 
    = \int_0^t \left( \left(A t^2 \varphi'\left(\frac{s}{t}\right) \right)^2 - 
    \left( At \varphi'\left(\frac{s}{t}\right) \right)^4 \right) \dd s
    = -t^5 A^2 C.
\end{equation}
Therefore $x(t;u)=-t^5 A^2 C e_3$, so we have indeed achieved a motion along $-e_3$.
Standard arguments using either tangent vectors or power series expansions (see e.g.\ \cite[Appendix]{zbMATH04154295} or \cite[Section~8.1]{zbMATH05150528}) then allow to prove that there exists $\rho > 0$ large enough such that \eqref{eq:ex:W2vsQ111} is indeed $\rho$-bounded-STLC.

\subsection{A coarse screening of cubic brackets}
\label{s:S2-S3}

\cref{thm:S2-S3-intro} for $m\in\N$ is a corollary of the following more precise statement.

\begin{theorem}
	\label{Thm:Kawski_Wm_precis-pos_crible}
    Let $m\in\N$ and $k\in\N^*$ with $\pi(k,m) \geq 3$.
    We assume $k$ is the smallest integer for which
    \begin{equation} \label{eq:S2-S3-comp_NEG}
        \quad f_{W_k}(0) \notin
        \left(\Bs_1 \cup \mathcal{P}_k \cup \Bs_{\intset{4,\pi(k,m)}}\right)(f)(0),
    \end{equation}
    where $\mathcal{P}_k$ is defined in \eqref{eq:P_k}.
    Then system \eqref{syst} has a drift along $f_{W_k}(0)$, parallel to the subspace $( \Bs_1 \cup \mathcal{P}_k \cup \Bs_{\intset{4,\pi(k,m)}} )(f)(0)$,  as $(t,t^{-\alpha} \|u\|_{W^{m,\infty}}) \rightarrow 0$
    where $\alpha=\frac{\pi(k,0)-\pi(k,m)}{\pi(k,m)-1}$.
\end{theorem}

\cref{Thm:Kawski_Wm_precis-pos_crible} follows from the same strategy as \cref{Thm:Kawski_Wm_precis-pos}, presented in \cref{subsec:DPL,subsec:VR,subsec:CLE,subsec:II,subsec:D}.
We explain below how to adapt \cref{subsec:DPL,subsec:VR,subsec:CLE} in order to conclude with the same \cref{subsec:II,subsec:D}. 
We will make repeated use of the following algebraic result.

\begin{lemma}[Algebraic preliminaries] \label{p:1+2}
    Let $j,k \in \N^*$.
    \begin{enumerate}
        \item \label{p:1+2/1}
        If $j \leq k$, then $(M_{k-1},W_{j}) = P_{j,k} \in \Bs$.
        
        \item \label{p:1+2/2} 
        If $j > k$, then $\supp [M_{k-1}, W_{j}] \subset \{ P_{j',k',\nu'}; j' < j \}$.

        \item \label{p:1+2/3}
        If $\nu \in \N$, then $\supp [M_{k-1}, W_{j,\nu}] \subset \{ P_{j',k',\nu'}; j' \leq j \}$. 
    \end{enumerate}
\end{lemma}

\begin{proof}
    For \cref{p:1+2/2}, let $b \in \supp [M_{k-1}, W_{j}]$.
    Since $\Bs_3$ spans $S_3(X)$, $b = P_{j',k',\nu'}$ with $j' \leq k' \in \N^*$ and $\nu' \in \N$.
    On the one hand $n_0(b) = 2j'+k'-2+\nu' \geq 3j'-2$.
    On the other hand $[M_{k-1}, W_{j}] \in S_{3,k+2j-2}(X)$, where $k+2j-2 \leq 3j-3$.
    Thus $3j'-2 \leq 3j-3$ so $j' < j$.
    
    Thanks to \eqref{eq:jacobi.rtl}, \cref{p:1+2/3} follows from \cref{p:1+2/1,p:1+2/2}.
\end{proof}

\begin{lemma}[Dominant part of the logarithm]
    \label{lem:dominant-S2-S3}
	Under the assumptions of \cref{Thm:Kawski_Wm_precis-pos_crible}, let $\mathbb{P}$ be a component along $f_{W_k}(0)$, parallel to the subspace $(\Bs_1 \cup \mathcal{P}_k \cup \Bs_{\intset{4,\pi(k,m)}} )(f)(0)$.
	Then \eqref{eq:quad-zpi} holds.
\end{lemma}

\begin{proof}
	We follow the proof of \cref{Lem:DPL}.
	By minimality of $k$, for every $j<k$,
    \begin{equation}
        f_{W_j}(0) \in 
        \left(\Bs_1 \cup \mathcal{P}_j \cup \Bs_{\intset{4,\pi(j,m)}}\right)(f)(0)
        \subset
        \left(\Bs_1 \cup \mathcal{P}_k \cup \Bs_{\intset{4,\pi(k,m)}}\right)(f)(0),
    \end{equation}
    since $j \mapsto \mathcal{P}_{j}$ and $j \mapsto \pi(j,m)$ are non-decreasing.
    Since $\Bs_1 \cup \mathcal{P}_k \cup \Bs_{\intset{4,\pi(k,m)}} $ is stable by right bracketing with $X_0$, one also has
    \begin{equation}
        f_{W_{j,\nu}}(0) \in \left(\Bs_1 \cup \mathcal{P}_k \cup \Bs_{\intset{4,\pi(k,m)}}\right)(f)(0),     
    \end{equation}
    for every $j <k$ and $\nu \geq 0$.
    Hence $(\Bs_1 \cup \mathcal{P}_k \cup \Bs_{\intset{4,\pi(k,m)}} )(f)(0)
    = \mathcal{N}(f)(0)$ where
	\begin{equation} \label{eq:quad-N_ref}
		\mathcal{N} := \Bs_1 \cup \mathcal{W}_k \cup \mathcal{P}_k \cup \Bs_{\intset{4,\pi(k,m)}}
      	\quad \text{ and } \quad 
      	\mathcal{W}_k := \left\{ W_{j,\nu} ; j <k \right\}.
	\end{equation}
	
	\step{Estimates of other coordinates of the second kind} 
	Let $b \in \Bs_{\intset{1,\pi(k,m)}}$ such that $b \notin \mathcal{N} \cup \{ \bb \}$. 
	The only case which is not already treated in the proof of \cref{Lem:DPL} is $b=P_{j,l,\nu}$ with $j\geq k$. 
	Then \eqref{bound-xiPjknu/j0k0} (with $(p_1, p_2, j_0, k_0) \leftarrow (1,\infty,k,k)$) proves \eqref{eq:xib-XI-L} with $\sigma=3k+1$ and $\Xi(t;u)=t \|u_k\|_{L^2}^2$. Indeed,
	$\|u_k\|_{L^2}^2 \|u_k\|_{L^{\infty}} \leq  t \|u_k\|_{L^2}^2$ when $\|u_1\|_{\infty} \leq 1$ and $k \geq 2$.
	
	\step{Estimates of other cross terms} Let $q \geq 2$, $b_1 \geq \dotsb \geq b_q \in \Bs$ such that $n_1(b_1) + \dotsb + n_1(b_q) \leq \pi(k,m)$ and $\supp \mathcal{F}(b_1, \dotsc, b_q) \not \subset \mathcal{N}$.
	The only cases which are not already treated in the proof of \cref{Lem:DPL} are
	\begin{itemize}
		\item $q=3$ and $b_1, b_2, b_3 \in \Bs_1$, then \eqref{eq:xib-othercross-XI-L} holds with $\alpha_1=\alpha_2=\alpha_3=1/2$ (see the preliminary estimates in the Step \ref{step:quad-dom-cross} of the proof of \cref{Lem:DPL}),
		
		\item $q=2$, $b_1=W_{j_1,\nu_1}$, $b_2=M_{k_1-1}$ and 
		$\supp [b_1,b_2] \cap \{P_{j_2,k_2,\nu};j_2 \geq k\} \neq \emptyset$.
	\end{itemize}
	In this last case, \eqref{eq:xib-othercross-XI-L} holds for $i=2$ with $\alpha_2=1/2$ (see the preliminary estimates in the Step \ref{step:quad-dom-cross} of the proof of \cref{Lem:DPL}). 
	By \cref{p:1+2}, $j_1 \geq k$, thus \eqref{bound-xiWjnu/j0} (with $(p,j_0) \leftarrow (1,k)$) proves that \eqref{eq:xib-othercross-XI-L} holds for $i=1$ with $\sigma_1=2k+1$, $\alpha_1=1/2$ and $\Xi(t,u)=t \|u_k\|_{L^2}^2$. 
	Indeed, $\|u_k\|_{L^2}^2 \leq \sqrt{t} \|u_k\|_{L^2}$ when $\|u_1\|_{L^\infty} \leq 1$ and $k \geq 2$.
\end{proof}

\begin{lemma}[Vectorial relations]
	Under the assumptions of \cref{Thm:Kawski_Wm_precis-pos_crible},
	\begin{enumerate}
		\item \label{it:quad3-vecrec-1} 
		the vectors $f_{M_0}(0),\dots,f_{M_{k-1}}(0)$ are linearly independent,
		\item \label{it:quad3-vecrec-2} 
		if $\vartheta(k) \geq 2$, then 
		$\vect\{ f_{M_0}(0),\dots,f_{M_{k-1}}(0) \} \cap 
		(\mathcal{W}_k + S_{\intset{3, \vartheta(k)}})(f)(0) = \{0\}$,
	\end{enumerate}
	where $\vartheta(k)$ is defined in \cref{p:quad-loop} and
	$\mathcal{W}_k=\{W_{j,\nu} ; j<k\}$ as in \eqref{eq:quad-N_ref}.
\end{lemma}

\begin{proof}
	We adapt the proof of \cref{p:quad-vectors}. 
	For \cref{it:quad3-vecrec-1}, one replaces \eqref{S22k-2-pik} by
	\begin{equation}
		S_{2,\intset{1,2k-2}}(f)(0) \subset 
		\left(\Bs_1 \cup \mathcal{P}_k \cup \Bs_{\intset{4,\pi(k,m)}}\right)(f)(0).
	\end{equation}
	For \cref{it:quad3-vecrec-2}, $B$ is assumed to belong to $\mathcal{W}_k + S_{\intset{3, \vartheta(k)}}(X)$ thus, by \cref{p:1+2},
	\begin{equation}
		\begin{split}
			B_5
			:=\ad_{B_4}^2(X_0)
			& =[ M_{k-1}+\dots+\gamma_0 M_0 + B,M_{k}+\dots+\gamma_0 M_1 + [B,X_0] ] \\
			& \in W_k + S_{2,\intset{1,2k-2}}(X) + \mathcal{P}_k + S_{\intset{4,2\vartheta(k)}}(X),
		\end{split}
	\end{equation}
	yielding the same contradiction as in the proof of \cref{p:quad-vectors}.
\end{proof}

\begin{lemma}[Closed-loop estimates]
	Under the assumptions of \cref{Thm:Kawski_Wm_precis-pos_crible}, one has \eqref{eq:quad-loop}.
\end{lemma}

\begin{proof}
	We adapt the proof of \cref{p:quad-loop} with
	$\mathcal{N} := (\{ M_0, \dotsc, M_{k-1} \} \setminus M_i) \cup 
	\mathcal{W}_k \cup \Bs_{\intset{3,\vartheta(k)}}$.
	
	\step{Estimates of other coordinates of the second kind}
	Let $b \in \Bs_{\intset{1,\vartheta(k)}}$ such that $b \notin \mathcal{N} \cup \{ \bb \}$. 
	The only case which is not treated in the proof of \cref{p:quad-loop} is $b=W_{j,\nu}$ with $j\geq k$. 
	Then \eqref{bound-xiWjnu/j0} with $(p,j_0) \leftarrow (1,k)$ proves that \eqref{eq:xib-XI-L} holds with $\sigma=(2k+1)$ and $\Xi(t,u)=\|u_k\|_{L^2}^2 \leq \sqrt{t}\|u_k\|_{L^2}$.
	
	\step{Estimates of cross terms}
	Let $q \geq 2$, $b_1 \geq \dotsb \geq b_q \in \Bs \setminus \{ X_0 \}$ such that $n_1(b_1) + \dotsb + n_1(b_q) \leq \vartheta(k)$ and $\supp \mathcal{F}(b_1, \dotsc, b_q) \not \subset \mathcal{N}$. 
	Then $q=2$, $b_1, b_2 \in \Bs_1$ and $\supp [b_1,b_2] \cap \{ W_{j,\nu};j\geq k\} \neq \emptyset$. 
	One may assume $b_1=M_{\ell}$ with $\ell \geq k$ and then \eqref{eq:xib-othercross-XI-L} holds with $\sigma_i = k+1$, $\alpha_i = 1$ and $\Xi(t,u) = \sqrt{t}\|u_k\|_{L^2}$.
\end{proof}

\section{Kawski's refined \texorpdfstring{$W_2$}{W2} obstruction} \label{s:W2refined}

The goal of this section is to prove the case $k=2$ in \cref{Thm:CN_W123}, as a consequence of the following more precise statement.

\begin{theorem} \label{thm:W2precise}
    Assume that $f_{W_1}(0) \in \mathcal{N}_1(f)(0)$ and $f_{W_2}(0) \notin \mathcal{N}_2(f)(0)$.  Then, system \eqref{syst} has a drift along $f_{W_2}(0)$, parallel to $\mathcal{N}_2(f)(0)$, as $(t,\|u\|_{L^\infty}) \to 0$.
\end{theorem}

\subsection{Limiting examples} \label{Subsec:Ex_W2}

Let us illustrate that the set $\mathcal{N}_2$ defined in \eqref{def:mathcalE2} of brackets which can compensate $W_2$ must include $P_{1,1,\nu}$ for every $\nu \in \N$.
As an illustration, we prove the following controllability results using classical sufficient conditions due to Sussmann or Bianchini and Stefani because they are simpler to apply.
Nevertheless, the same results would follow from Agrachev and Gamkrelidze conditions as in \cref{Subsec:Ex_W3}.

\medskip

\emph{Limiting example for $P_{1,1,0}$.}
Consider the system
\begin{equation} \label{syst:Jakubczyk}
    \begin{cases}
        \dot{x}_1 = u \\
        \dot{x}_2 = x_1 \\
        \dot{x}_3 = x_2^2 + x_1^3.
    \end{cases}
\end{equation}
Written in the form \eqref{syst}, this system satisfies
\begin{equation}
    f_{M_0}(0)=e_1, \quad 
    f_{M_1}(0)=e_2, \quad 
    f_{P_{1,1,0}}(0)=6 e_3, \quad 
    f_{W_2}(0)= 2 e_3
\end{equation} 
and $f_b(0)=0$ for any $b \in \Bs \setminus\{ M_0, M_1, W_2, P_{1,1,0} \}$.

This system was proposed by Jakubczyk and is known to be $L^\infty$-STLC since\footnote{Sussmann's initial proof involves controls with $\|u\|_{L^\infty} \leq 1$, but extends easily to any bound on $u$.} \cite[p.\ 711-712]{MR710995}. 
It also satisfies Sussmann's $\mathcal{S}(\theta)$ condition (see \cite[Theorem~7.3]{MR872457} or \cite[Theorem 3.29]{zbMATH05150528}) for any $\theta > 1/2$ (see also \cite[Section 2.4.1]{JDE} for a short direct proof).

\medskip

\emph{Limiting example for $P_{1,1,\nu}$.} 
Let $\nu \in \N^*$. 
We consider the system
\begin{equation}
    \begin{cases}
        \dot{x}_1 = u \\
        \dot{x}_2 = x_1 \\
        \dot{x}_3 = x_1^3 \\
        \dot{x}_{3+i}=x_{3+i-1} \quad \text{ for } i=1,\dots,\nu-1\\
        \dot{x}_{3+\nu} = x_2^2 + x_{3+\nu-1}.
    \end{cases}
\end{equation}
Written in the form \eqref{syst}, this system satisfies
\begin{equation}
    f_{M_0}(0)=e_1, \quad 
    f_{M_1}(0)=e_2, \quad 
    f_{P_{1,1,\mu}}(0)=3! e_{3+\mu} \text{ for } \mu=0,\dots,\nu, \quad 
    f_{W_2}(0)= 2 e_{3+\nu}
\end{equation} 
and $f_b(0)=0$ for any $b \in \Bs \setminus\{ M_0, M_1, W_2, P_{1,1,\mu};\mu\in\llbracket 0 , \nu \rrbracket \}$.

For $\nu = 3$, this system corresponds to \cite[Example 2.4]{HK}.
To prove that it is $L^\infty$-STLC, the key point is to prove that $\pm e_3 = \pm 6 f_{P_{1,1,0}}(0)$ are tangent vectors.
Then, the $L^\infty$-STLC follows from the elementary remark that, if, for some $b \in \Bs$, $\pm f_b(0)$ are tangent vectors, then so are $\pm f_{(b,X_0)}(0)$ (see \cite[Theorem 6]{HK} or \cite[claim P2]{bianchini1986sufficient}). 
As in the case $\nu = 0$, the fact that $\pm e_3$ are tangent vectors can be proved using oscillating controls or Sussmann's $\mathcal{S}(\theta)$ condition with $\theta > 1/2$ as reformulated in \cite[Theorem 2]{bianchini1986sufficient} by Bianchini and Stefani.


\medskip

\emph{A non-controllable example involving $Q_{1,1,1}$.}
In \cref{Subsec_Optim_Wk}, we recalled that system \eqref{eq:ex:W2vsQ111} is small-time locally controllable with large enough controls in $L^\infty$, but not $L^\infty$-STLC in the sense of \cref{Def:WmSTLC}.
For this system, one has $6 f_{W_2}(0) = - f_{Q_{1,1,1}}(0)$ (and $Q_{1,1,1}$ is the only bracket ``compensating'' $W_2$).
But $Q_{1,1,1}$ does not belong to the set $\mathcal{N}_2$ defined in \eqref{def:mathcalE2} of brackets which can compensate $W_2$ for $L^\infty$-STLC.
Hence, the fact that \eqref{eq:ex:W2vsQ111} is not $L^\infty$-STLC can be seen as an application of the case $j=2$ of \cref{Thm:CN_W123}.

\subsection{Dominant part of the logarithm}

\begin{lemma}
    Assume that $f_{W_1}(0) \in \mathcal{N}_1(f)(0)$ and $f_{W_2}(0) \notin \mathcal{N}_2(f)(0)$. 
    Let $\mathbb{P}$ be a component along $f_{W_2}(0)$, parallel to $\mathcal{N}_2(f)(0)$.
    Then
    \begin{equation} \label{eq:W2-z3}
        \mathbb{P} \cZ{3}(t,f,u)(0) = \xi_{W_2}(t,u)
        + O\Big( |(u_1,u_2)(t)|^2 + t \|u_2\|_{L^2}^2 + \|u_1\|_{L^4}^2 \|u_2\|_{L^2}  + \|u_1\|_{L^4}^4  \Big).
    \end{equation}
\end{lemma}

\begin{proof}
    By assumption, $f_{W_1}(0) \in \mathcal{N}_1(f)(0)$.
    Since $\mathcal{N}_1$ is stable by right bracketing with $X_0$, $f_{W_{1,\nu}}(0) \in \mathcal{N}_1(f)(0)$ for every $\nu \geq 0$.
    Thus, since $\mathcal{N}_1 \subset \mathcal{N}_2$, $\mathcal{N}_2(f)(0) = \mathcal{N}(f)(0)$, where $\mathcal{N}$ is defined as
    \begin{equation} \label{eq:W2-N}
        \mathcal{N} := \mathcal{N}_2 \cup \{ W_{1,\nu} ; \nu \in \N \},
    \end{equation}
    where $\mathcal{N}_2$ is defined in \eqref{def:mathcalE2}.
    By assumption, $f_{W_2}(0) \notin \mathcal{N}_2(f)(0) = \mathcal{N}(f)(0)$.
    
    We intend to apply \cref{p:PZM-XI} with $M \gets 3$, $L \gets 6$, $\bb \gets W_2$ and $\mathcal{N}$ as in \eqref{eq:W2-N}, so that \eqref{eq:W2-z3} will follow from \eqref{eq:PZM-xibb-OXi}, for the appropriate choice of $\Xi(t,u)$.
    Let us check that the required estimates are satisfied.
    
    \step{Estimates of other coordinates of the second kind}
    Let $b \in \Bs_{\intset{1,3}}$ such that $b \notin \mathcal{N} \cup \{ \bb \}$.
    
    We investigate the different possibilities depending on $n_1(b)$.
    \begin{itemize}
        \item One cannot have $n_1(b) = 1$ since $\Bs_1 \subset \mathcal{N}_2$.
        
        \item If $n_1(b) = 2$, by \eqref{Bstar_S2} and \eqref{eq:W2-N}, one has $b = W_{j,\nu}$ with either ($j \geq 3$) or ($j = 2$ and $\nu \geq 1$).
        Thus $|b| \geq 6$.
        By estimate \eqref{bound-xiWjnu/j0} with $(p,j_0) \gets (1,2)$, \eqref{eq:xib-XI-L} holds with $\sigma = 6$ and
        \begin{equation}
            \Xi(t,u) := t \|u_2\|_{L^2}^2.
        \end{equation}
        
        \item If $n_1(b) = 3$, by \eqref{Bstar_S3} and \eqref{def:mathcalE2}, $b = P_{j,k,\nu}$ with $k \geq 2$.
        Thus $|b| \geq 5$.
        By estimate \eqref{bound-xiPjknu/j0k0} with $(p_1,p_2,j_0, k_0) \gets (2,2,1,2)$, \eqref{eq:xib-XI-L} holds with $\sigma=5$ and
        \begin{equation}
            \Xi(t,u) := \|u_1\|_{L^4}^2 \|u_2\|_{L^2}.
        \end{equation}
    \end{itemize}
    
    \step{Estimates of cross terms}
    Let $q \geq 2$, $b_1 \geq \dotsb \geq b_q \in \Bs \setminus \{ X_0 \}$ such that $n_1(b_1) + \dotsb + n_1(b_q) \leq 3$ and $\supp \mathcal{F}(b_1, \dotsc, b_q) \not \subset \mathcal{N}$.
    
    \medskip 
    \noindent We start with preliminary estimates.
    \begin{itemize}
        \item If $b_i = M_j$ for some $j \in \intset{0,1}$, by \eqref{xi_S1},
        \begin{equation}
            |\xi_{b_i}(t,u)| = |u_{j+1}(t)| = \frac{t^{|b_i|}}{|b_i|!} t^{-(j+1)} (j+1)! |u_{j+1}(t)|
        \end{equation}
        so \eqref{eq:xib-othercross-XI-L} holds with $\sigma_i = j+1$, $\alpha_i = 1/2$ and $\Xi(t,u) = |(u_1,u_2)(t)|^2$.
        
        \item If $b_i = M_j$ for $j \geq 2$, by \eqref{bound-xiMj/J0} (with $(p,j_0) \gets (2,2)$), \eqref{eq:xib-othercross-XI-L} holds with $\sigma_i = 3$, $\alpha_i = 1/2$ and $\Xi(t,u) = t\|u_2\|_{L^2}^2$.
        
        \item By \eqref{eq:xib-u1-Lk}, for each $b_i \in \Bs_2$, \eqref{eq:xib-othercross-XI-L} holds with $\sigma_i = 3$, $\alpha_i = 1/2$ and $\Xi(t,u) = t \|u_1\|_{L^4}^4$.
    \end{itemize}
    Since $n_1(b_1) + \dotsb n_1(b_q) \leq 3$ and $q \geq 2$, all the $b_i$ belong to $\Bs_{\intset{1,2}}$.
    Thanks to the preliminary estimates, $\alpha = q/2 \geq 1$.
\end{proof}

\subsection{Vectorial relation}

\begin{lemma} \label{p:W2-vectors}
	Assume that $f_{W_1}(0) \in \mathcal{N}_1(f)(0)$ and $f_{W_2}(0) \notin \mathcal{N}_2(f)(0)$. 
	Then, the vectors $f_{M_0}(0)$ and $f_{M_1}(0)$ are linearly independent.
\end{lemma}

\begin{proof}
    This statement is implied by the case $k = 2$ and $\pi(k)=2$ in \cref{p:quad-vectors}.
\end{proof}

\subsection{Closed-loop estimate}

\begin{lemma} \label{p:W2-loop}
    Assume that $f_{M_0}(0)$ and $f_{M_1}(0)$ are linearly independent.
    Then,
    \begin{equation} \label{eq:W2-loop}
        |(u_1,u_2)(t)|=O \left( |x(t;u)|+\|u_1\|_{L^2}^2 + t^{\frac 12} \|u_2\|_{L^2} \right).
    \end{equation}
\end{lemma}

\begin{proof}
    This statement is implied by the case $k=2$ and $\pi(k)=2$ in \cref{p:quad-loop}.
\end{proof}

\subsection{Interpolation inequality}

\begin{lemma}
	There exists $C > 0$ such that, for every $t > 0$ and $u \in \lone$,
	\begin{equation}
		\label{eq:w2-u1}
		\|u_1\|_{L^4}^4 \leq C \|u\|_{L^\infty}^2 \|u_2\|_{L^2}^2.
	\end{equation}
\end{lemma}

\begin{proof}
	Inequality \eqref{eq:w2-u1} follows from \cref{thm:GN} with $\phi \gets u_2$, $(p,q,r,s) \gets (4,2,\infty,2)$, $(j,l) \gets (1,2)$, $\alpha \gets 1/2$.
    The lower-order term in \eqref{eq:GNS-estimate} is absorbed using the estimate $\|u_2\|_{L^2} \leq t^{\frac 52} \|u\|_{L^\infty}$, which stems from Hölder's inequality and the equality $u_2(t) = \int_0^t (t-s) u(s) \dd s$.
\end{proof}

\subsection{Proof of the presence of the drift}

\begin{proof}[Proof of \cref{thm:W2precise}]
    Let $\mathbb{P}$ be a component along $f_{W_2}(0)$ parallel to $\mathcal{N}_2(f)(0)$.
    By \cref{thm:Key_1} with $M \gets 3$,
    \begin{equation}
        x(t;u) = \cZ{3}(t,f,u)(0) + O\left( \|u_1\|_{L^4}^4 + |x(t;u)|^{1+\frac{1}{3}} \right),
    \end{equation}
    where, by \eqref{eq:W2-z3} and \eqref{xi_Wjnu},
    \begin{equation}
        \begin{split}
        	\mathbb{P} \cZ{3}(t,f,u)(0) = \frac 12 \int_0^t u_2^2 + 
        	O\Big( |(u_1,u_2)(t)|^2 + t \|u_2\|_{L^2}^2 + \|u_1\|_{L^4}^2 \|u_2\|_{L^2}  + \|u_1\|_{L^4}^4  \Big).
        \end{split}
    \end{equation}
    Moreover, by the closed-loop estimate \eqref{eq:W2-loop},
    \begin{equation}
    	|(u_1,u_2)(t)|^2 = O\left(
    		|x(t;u)|^2
 			+ \|u_1\|_{L^4}^4
		    + t \|u_2\|_{L^2}^2	
    	 \right).
    \end{equation}
    Gathering these equalities and the interpolation estimate \eqref{eq:w2-u1} yields
    \begin{equation}
        \mathbb{P} x(t;u) = \frac 12 \int_0^t u_2^2
        + O \left( \left(t+\|u\|_{L^\infty}\right) \int_0^t u_2^2 +|x(t;u)|^{1+\frac{1}{3}} \right).
    \end{equation}
    This implies a drift along $f_{W_2}(0)$, parallel to $\mathcal{N}_2(f)(0)$, as $(t,\|u\|_{L^\infty}) \to 0$, in the sense of \cref{Def:drift}.
\end{proof}

\section{New refined \texorpdfstring{$W_3$}{W3} obstruction} \label{s:W3refined}

The goal of this section is to prove the case $k=3$ of \cref{Thm:CN_W123}, as a consequence of the following more precise statement.

\begin{theorem} \label{thm:W3precise}
    Assume that $f_{W_1}(0) \in \mathcal{N}_1(f)(0)$, $f_{W_2}(0) \in \mathcal{N}_2(f)(0)$ and $f_{W_3}(0) \notin \mathcal{N}_3(f)(0)$. 
    Then there exist a linear form $\mathbb{P}_{W_3}:\R^d\rightarrow\R$  giving a component along $f_{W_3}(0)$, another linear form $\mathbb{P}:\R^d \rightarrow \R$, $C>0$, $\beta>1$ such that, 
    for every $\varepsilon>0$, there exists 
    $\rho=\rho(\varepsilon)>0$
   such that for every 
   $t \in (0,\rho)$ and
   $u \in L^\infty((0,t),\R)$ with $\|u\|_{L^\infty}<\rho$,
\begin{equation} \label{drift_Pt}
   \left(\mathbb{P}_{W_3}+t\mathbb{P}\right)x(t;u) 
   \geq (1-\varepsilon) \xi_{W_3}(t,u)-C|x(t;u)|^{\beta}.
\end{equation}
\end{theorem}

The conclusion \eqref{drift_Pt} is not exactly a drift\footnote{One can see \eqref{drift_Pt} as a particular case of a ``weak drift'' as $(t,\|u\|_{L^\infty})\to 0$ in the sense of \cref{Def:Weak-drift} of \cref{s:m=-1}.}\footnote{This is not a technical limitation. See \cref{s:W3-comment}.} 
along $f_{W_3}(0)$, 
parallel to $\mathcal{N}_3(f)(0)$, 
as $(t,\|u\|_{L^\infty}) \rightarrow 0$, 
in the sense of \cref{Def:drift}, 
because the left-hand side of the inequality involves a linear form $\mathbb{P}_t := \mathbb{P}_{W_3} + t \mathbb{P}$ that may not give a component along $f_{W_3}(0)$. 
Nevertheless this result is still an obstruction to $L^\infty$-STLC. 
Indeed, one may assume $\mathbb{P}_{W_3}$ and $\mathbb{P}$ are linearly independent, then by considering $e \in \R^d$ such that $\mathbb{P}_{W_3}e=1$ and $\mathbb{P}e=0$, estimate \eqref{drift_Pt} prevents $x(t;u)$ from reaching targets  of the form $x^{\star}=-ae$ with $0<a \ll 1$, because this would entail $-a=(\mathbb{P}_{W_3}+t\mathbb{P})x(t;u)\geq -C|x^{\star}|^{\beta}=-Ca^{\beta}$, which fails for $a$ small enough, because $\beta>1$.

The proof of \cref{thm:W3precise} is a slight variation of the unified approach as presented in \cref{s:method-args}, in which closed-loop estimates are used not only for cross terms of coordinates of the second kind, but also for some coordinates of the second kind.

\subsection{Limiting examples} \label{Subsec:Ex_W3}

We illustrate that the set $\mathcal{N}_3$ defined in \eqref{def:mathcalE3} of brackets which can compensate $W_3$ is ``minimal'' in the following sense: for each bracket $b$ of $\mathcal{N}_3$, we construct an example of an $L^\infty$-STLC system for which there is a competition between $W_3$ and $b$.
One has
\begin{equation} \label{eq:N3-bad-good}
    \mathcal{N}_3 = \{ M_\nu, P_{1,l,\nu}, Q_{1,1,2,\nu}, R_{1,1,1,1,\nu}, R^\sharp_{1,1,1,\mu,\nu} ; l\in \N^*, \mu, \nu \in \N \} \cup \{ Q_{1,1,1}, Q^\flat_{1,0}, Q^\flat_{1,1}, Q^\flat_{1,2} \}.
\end{equation}
The brackets of the first list can be considered as ``good'', and those of the second list as ``bad'' in senses detailed below.
We treat both lists separately.

\subsubsection{Good-bad competitions}

We consider the first list of \eqref{eq:N3-bad-good}.
We skip the case of the $M_\nu$ since it is clear by the linear test that any system with $f_{W_3}(0) \in S_1(f)(0)$ and $S_1(f)(0) = \R^d$ is $L^\infty$-STLC.
For all the other brackets, we will prove the $L^\infty$-STLC property thanks to Agrachev and Gamkrelidze's sufficient condition \cite[Theorem 4]{AG93}, of which we now recall a version well-suited to our setting.

\begin{theorem} \label{thm:AG}
    Let $\sigma \in [0,1]$, $r \geq 0$ and $\Pi^1 \subset \Br(X)$ such that $\Pi^1$ generates a Lie algebra $\operatorname{Lie}(\Pi^1) \subset \mathcal{L}(X)$ with the following properties:
    \begin{itemize}
        \item $\Pi^1$ is a set of free generators of $\operatorname{Lie}(\Pi^1)$,
        \item for each $b \in \Br(X)$ with $n_1(b)$ even and $n_0(b)$ odd, $\eval(b) \in \operatorname{Lie}(\Pi^1)$.
    \end{itemize}
    For $k \in \N^*$ let $\Pi^{k+1} := [\Pi^1, \Pi^k]$ and $\Pi^\infty := \cup_{k\in \N^*} \Pi^k$.
    For $k \in \N^*$  and $\pi \in \Pi^k$, let $\omega(\pi) := |\pi|-\sigma k$.
    
    Suppose that, for all $k \in \N$, and every $\pi \in \Pi^{2k+1}$ with $n_1(\pi)$ even and $n_0(\pi)$ odd, and $\omega(\pi) \leq r$,
    \begin{equation}
        f_\pi(0) \in \vect \{ f_{\pi'}(0) ; \pi' \in \Pi^\infty, \omega(\pi') < \omega(\pi) \}.
    \end{equation}
    Assume moreover that
    \begin{equation}
        \R^d = S_1(f)(0) + \vect \{ f_\pi(0) ; \pi \in \Pi^\infty, \omega(\pi) \leq r \}.
    \end{equation}
    Then, the system is $L^\infty$-STLC.
\end{theorem}

To apply \cref{thm:AG}, the key point is thus to find a set $\Pi^1$ and a parameter $\sigma \in [0,1]$ such that the ``good'' brackets that one intends to use have a smaller weight $\omega$ than the ``bad'' ones.
All the following examples will be handled with $\sigma = 1$ and the following choice of $\Pi^1$:
\begin{equation} \label{DefPi1_W3}
    \Pi^{1} := \{\ad_{M_2}^{i_2} \ad_{M_1}^{i_1} \ad_{X_1}^{i_0}(X_0); i_0, i_1, i_2 \in \N, (i_0,i_1) \notin \{1\} \times \N^*, (i_0,i_1,i_2) \notin \{0\} \times \{1\} \times \N^* \}.
\end{equation}
By the elimination theorem \cite[Proposition 1.1]{MR516004}, $\Pi^1$ is a set of free generators of $\operatorname{Lie} (\Pi^1)$ and  $\mathcal{L}(X)=\R X_1 \oplus \R M_1 \oplus \R M_2 \oplus \operatorname{Lie} (\Pi^1)$.
In particular, $\operatorname{Lie} (\Pi^1)$ contains $\eval(b)$ for every $b \in \Br(X)$ of type (even, odd).
We compute the associated weights for all brackets of interest with $\sigma = 1$.
\begin{itemize}
    \item Since $W_3 = \ad_{M_2}^2(X_0) \in \Pi^1$, $\omega(W_3) = |W_3| - 1 = 6$.
    \item For $l \in \{ 1, 2, 3 \}$ and $\nu \in \N$, $P_{1,l,\nu} = \ad_{M_{l-1}} \ad_{M_0}^2(X_0) 0^\nu \in \Pi^{1+\nu}$, so $\omega(P_{1,l,\nu}) = 2 + l \leq 5$.
    \item For $l \geq 4$ and $\nu \in \N$, $P_{1,l,\nu} = (\ad_{M_2}(X_0) 0^{l-3}, \ad_{M_0}^2(X_0)) 0^\nu \in \Pi^{l-1+\nu}$, so $\omega(P_{1,l,\nu})  = 4$.
    \item For $\nu \in \N$, $Q_{1,1,2,\nu} = \ad_{M_1}\ad_{M_0}^3(X_0) 0^\nu \in \Pi^{1+\nu}$, so $\omega(Q_{1,1,2,\nu}) = 5$.
    \item For $\nu \in \N$, $R_{1,1,1,1,\nu} = \ad_{M_0}^5 (X_0) 0^\nu \in \Pi^{1+\nu}$, so $\omega(R_{1,1,1,1,\nu}) = 5$.
    \item For $\mu, \nu \in \N$, $R^\sharp_{1,1,1,\mu,\nu} = (\ad_{M_0}^2(X_0) 0^\mu, \ad_{M_0}^3(X_0)) 0^\nu \in \Pi^{2+\mu+\nu}$, so $\omega(R^\sharp_{1,1,1,\mu,\nu}) = 5$.
\end{itemize}
Hence, all these brackets have a smaller weight than $W_3$. 
Moreover, \cref{thm:AG} does not require for them to be compensated.
Indeed, the $P$s and $R$s have an odd $n_1$.
Moreover, $n_0(Q_{1,1,2,\nu}) = 2+\nu$ and $Q_{1,1,2,\nu} \in \Pi^{1+\nu}$ so $Q_{1,1,2,\nu}$ is never simultaneously of type (even, odd) and inside $\Pi^{2k+1}$.

We now provide limiting examples of systems, whose $L^\infty$-STLC can be established by the above argument.
These examples prove that the first list of \eqref{eq:N3-bad-good} is minimal.

\paragraph{Limiting example for $P_{1,l,\nu}$ with $l \in \{1,2,3\}$.} 
Let $\nu \in \N$. For $\nu=0$, we consider the system
\begin{equation} \label{Ex:W3vsP1l}
    \begin{cases}
        \dot{x}_1=u \\
        \dot{x}_2=x_1 \\
        \dot{x}_3=x_2 \\
        \dot{x}_4=x_3^2 + x_1^2 x_l
    \end{cases}
\end{equation}
while for $\nu > 0$ we consider the system
\begin{equation} \label{Ex:W3vsP1lnu}
    \begin{cases}
        \dot{x}_1=u \\
        \dot{x}_2=x_1 \\
        \dot{x}_3=x_2 \\
        \dot{x}_4= x_1^2 x_l \\
        \dot{x}_{4+\mu}=x_{4+\mu-1} & \text{ for } \mu=1,\dots,\nu-1 \\
        \dot{x}_{4+\nu}=x_3^2 + x_{4+\nu-1}.
    \end{cases}
\end{equation}
\medskip
Written in the form \eqref{syst}, these systems satisfy
\begin{equation}
    f_{M_{i-1}}(0)=e_i \text{ for } i\in \llbracket 1,3\rrbracket, \quad 
    f_{P_{1,l,\mu}}(0)=c e_{4+\mu} \text{ for } \mu\in\llbracket 0,\nu \rrbracket , \quad 
    f_{W_3}(0)=2 e_{4+\nu}, 
\end{equation}
where $c=6$ if $l=1$ and $c=2$ otherwise, and $f_b(0)=0$ for any other $b \in \Bs$. 

\paragraph{Limiting example for $P_{1,l,\nu}$ for $l \geq 4$.} Let $\nu \in\N$. For $\nu=0$, consider the system
\begin{equation} \label{Ex:W3vsP1l>4}
    \begin{cases}
        \dot{x}_1=u \\
        \dot{x}_{i}=x_{i-1} & \text{ for } i=2,\dots,l \\
        \dot{x}_{l+1}= x_3^2 + x_1^2 x_l
    \end{cases}
\end{equation}
while for $\nu > 0$ we consider the system
\begin{equation} \label{Ex:W3vsP1lnu>4}
    \begin{cases}
        \dot{x}_1=u \\
        \dot{x}_{i}=x_{i-1} & \text{ for } i=2,\dots,l \\
        \dot{x}_{l+1}= x_1^2 x_l \\
        \dot{x}_{l+1+\mu}=x_{l+\mu} & \text{ for } \mu=1,\dots,\nu-1 \\
        \dot{x}_{l+1+\nu}=x_3^2+x_{l+\nu}.
    \end{cases}
\end{equation}
Written in the form \eqref{syst}, these systems satisfy
\begin{equation}
    f_{M_{i-1}}(0)=e_i \text{ for } i\in \llbracket 1,l\rrbracket, \quad 
    f_{P_{1,l,\mu}}(0)=c e_{l+1+\mu} \text{ for } \mu\in\llbracket 0,\nu \rrbracket , \quad 
    f_{W_3}(0)=2 e_{l+1+\nu}, 
\end{equation}
where $c=6$ if $l=1$ and $c=2$ otherwise, and $f_b(0)=0$ for any other $b \in \Bs$. 

\paragraph{Limiting example for $Q_{1,1,2,\nu}$.} 
Let $\nu \in \N$. 
For $\nu  = 0$, we consider the system
\begin{equation}
    \begin{cases}
        \dot{x}_1 = u \\
        \dot{x}_2 = x_1 \\
        \dot{x}_3 = x_2 \\
        \dot{x}_4 = x_3^2 + x_1^3 x_2
    \end{cases}
\end{equation}
while for $\nu > 0$ we consider the system
\begin{equation}
    \begin{cases}
        \dot{x}_1 = u \\
        \dot{x}_2 = x_1 \\
        \dot{x}_3 = x_2 \\
        \dot{x}_4 = x_1^3 x_2 \\
        \dot{x}_{4+\mu}=x_{4+\mu-1} & \text{ for } \mu=1,\dots,\nu-1 \\
        \dot{x}_{4+\nu}=x_3^2+x_{4+\nu-1}.
    \end{cases}
\end{equation}
Written in the form \eqref{syst}, these systems satisfy
\begin{equation}
    f_{M_{i-1}}(0)=e_i \text{ for } i=1,2,3 
    \quad 
    f_{Q_{1,1,2,\mu}}=2 e_{4+\mu} \text{ for } \mu=0,\dots,\nu \quad
    f_{W_3}(0)=2 e_{4+\nu},
\end{equation}
and $f_b(0)=0$ for any other $b \in \Bs$. 

\paragraph{Limiting example for $R_{1,1,1,1,\nu}$.} 
Let $\nu \in \N$.
For $\nu = 0$, we consider the system
\begin{equation}
    \begin{cases}
        \dot{x}_1 = u \\
        \dot{x}_2 = x_1 \\
        \dot{x}_3 = x_2 \\
        \dot{x}_{4}= x_3^2 + x_1^5
    \end{cases}
\end{equation}
while for $\nu > 0$ we consider the system
\begin{equation}
    \begin{cases}
        \dot{x}_1 = u \\
        \dot{x}_2 = x_1 \\
        \dot{x}_3 = x_2 \\
        \dot{x}_{4}= x_1^5 \\
        \dot{x}_{4+\mu}=x_{4+\mu-1} & \text{ for } \mu=1,\dots,\nu-1 \\
        \dot{x}_{4+\nu}=x_3^2 + x_{4+\nu-1}.
    \end{cases}
\end{equation}
Written in the form \eqref{syst}, these systems satisfiy
\begin{equation}
    f_{M_{i-1}}(0)=e_i \text{ for } i = 1,2,3, 
    \quad 
    f_{R_{1,1,1,1,\mu}}(0)=5! e_{4+\mu}  \text{ for } \mu \in \llbracket 0, \nu \rrbracket,
    \quad
    f_{W_3}(0)=2 e_{4+\nu} 
\end{equation}
and $f_b(0)=0$ for any other $b \in \Bs$.

\paragraph{Limiting example for $R^{\sharp}_{1,1,1,\mu,\nu}$.}
Let $\mu, \nu \in \N$.
For $\nu = 0$, we consider the system
\begin{equation} \label{Ex:W3vsRdiese}
    \begin{cases}
        \dot{x}_1=u \\
        \dot{x}_2=x_1 \\
        \dot{x}_3=x_2 \\
        \dot{x}_4=x_1^3 \\
        \dot{x}_{4+\mu'} = x_{4+\mu'-1} & \text{ for } \mu' = 1, \dotsc, \mu \\
        \dot{x}_{5+\mu} = x_3^2 + x_1^2 x_{4+\mu}
    \end{cases}
\end{equation}
while for $\nu > 0$ we consider the system
\begin{equation}
    \begin{cases}
        \dot{x}_1=u \\
        \dot{x}_2=x_1 \\
        \dot{x}_3=x_2 \\
        \dot{x}_4=x_1^3 \\
        \dot{x}_{4+\mu'} = x_{4+\mu'-1} & \text{ for } \mu' = 1, \dotsc, \mu \\
        \dot{x}_{5+\mu} = x_1^2 x_{4+\mu} \\
        \dot{x}_{5+\mu+\nu'} = x_{5+\mu+\nu'-1} & \text{ for } \nu' = 1, \dotsc, \nu-1 \\
        \dot{x}_{5+\mu+\nu} = x_3^2 + x_{5+\mu+\nu-1}.
    \end{cases}
\end{equation}
Written in the form \eqref{syst}, this system satisfies
\begin{equation}
    \begin{split}
        & f_{M_{i-1}}(0)=e_i \text{ for } i = 1,2,3, \quad 
        f_{P_{1,1,\mu'}}(0) = 6 e_{4+\mu'} \text{ for } \mu' \in \intset{0,\mu}, \quad \\
        & f_{R^\sharp_{1,1,1,\mu,\nu'}}(0)=-12 (-1)^\mu e_{5+\mu+\nu'} \text{ for } \nu' \in \intset{0,\nu}, \quad
        f_{W_3}(0)=2 e_{5+\mu+\nu}
    \end{split}
\end{equation}
and $f_b(0)=0$ for any other $b \in \Bs$.

\subsubsection{Bad-bad competitions}
\label{s:W3-bad-bad}

The second list of \eqref{eq:N3-bad-good} consists of brackets which are associated with sign-definite coordinates of the second-kind.
Hence, they restore controllability in competition with $W_3$ only in situations where both sign-definite terms push the state in opposite directions.
Such ``bad-bad'' competitions (see e.g.\ \cite[Section 5]{zbMATH04154295} for an introduction) are not handled by classicial sufficient conditions such as \cite[Theorem~4]{AG93}.
We present the straightforward case of $Q_{1,1,1}$ here, and postpone the examples involving $Q^\flat_{1,0}, Q^\flat_{1,1}$ and $Q^\flat_{1,2}$, for which the proofs are more intricate, to \cref{s:Qflat}.

\paragraph{Limiting example for $Q_{1,1,1}$.} We consider the system
\begin{equation} \label{Ex:W3vsQ111}
    \begin{cases}
        \dot{x}_1=u \\
        \dot{x}_2=x_1 \\
        \dot{x}_3=x_2 \\
        \dot{x}_4=x_3^2 - x_1^4.
    \end{cases}
\end{equation}
Written in the form \eqref{syst}, this system satisfies
\begin{equation}
    f_{M_0}(0)=e_1, \quad 
    f_{M_1}(0)=e_2, \quad 
    f_{M_2}(0)=e_3, \quad 
    f_{W_3}(0)=2 e_4, \quad
    f_{Q_{1,1,1}}(0)=- 24 e_4
\end{equation}
and $f_b(0)=0$ for any $b \in \Bs \setminus\{ M_0, M_1, M_2, W_3, Q_{1,1,1}\}$. 
In \cite[Example 5.2]{zbMATH04154295},
Kawski proves that system \eqref{Ex:W3vsQ111} is $L^\infty$-STLC.
This can also be proved using oscillating controls as in \cref{Subsec_Optim_Wk}.

\subsection{Dominant part of the logarithm}

The following lemma is a little intricate due to the fact that the list \eqref{def:mathcalE3} is minimal.
If one only wishes to prove an easier version with $\mathcal{N}_3 \gets \mathcal{N}_3 \cup \{ Q_{1,1,1,\nu}, Q^\flat_{1,\mu,\nu} \}$, the proof could be shorter.

\begin{lemma} \label{p:W3-Z5}
    Let $\mathcal{N}_3':=\mathcal{N}_3 \setminus \{ Q^{\flat}_{1,0}, Q^{\flat}_{1,1}, Q^{\flat}_{1,2} \}$ i.e.,
    \begin{equation}
        \mathcal{N}_3' =\{ 
        M_{\nu},
        P_{1,l,\nu},
        Q_{1,1,1},
        Q_{1,1,2,\nu},
        R_{1,1,1,1,\nu},
        R^\sharp_{1,1,1,\mu,\nu}; l \in \N^*, \mu,\nu \in \N \},
    \end{equation}
    $\mathcal{N}_3'':=\mathcal{N}_3' \setminus \{Q_{1,1,1}\}$
    and 
    $\mathbb{P}_{W_3}:\R^d\rightarrow\R$ 
    be a component along $f_{W_3}(0)$ parallel to 
    $\mathcal{N}_3'(f)(0)$.
    Under the assumptions of \cref{thm:W3precise}, 
    \begin{equation} \label{eq:W3-z5}
        \begin{split}
            \mathbb{P}_{W_3} \cZ{5}(t,f,u)(0) = \xi_{W_3}(t,u)
            + O\Big( & 
            t \xi_{W_3}(t,u) 
            + \|u_2\|_{L^3}^3
            + \|u_1\|_{L^6}^3 \|u_3\|_{L^2} 
            + \|u_1\|_{L^6}^6  \\ &
            + |(u_1,u_2,u_3)(t)|^2   
            + a \xi_{W_{1}}^2(t,u) 
            + a' t \xi_{Q_{1,1,1}}(t,u)
            \Big).
        \end{split}
    \end{equation}
    where 
    \begin{itemize}
    \item $a=1$ if $f_{W_1}(0)\neq 0$ and $a=0$ otherwise,
    \item $a'=1$ if $f_{Q_{1,1,1}}(0) \notin \mathcal{N}_3''(f)(0)$
     and $a'=0$ otherwise.
    \end{itemize}
\end{lemma}

\begin{proof}
For $i \in \{1,2\}$,  $f_{W_i}(0) \in \mathcal{N}_i(f)(0)$ and $\mathcal{N}_i$ is stable by right bracketing with $X_0$ thus $f_{W_{i,\nu}}(0) \in \mathcal{N}_i(f)(0)$ for every $\nu \geq 0$. Since $\mathcal{N}_1 \subset \mathcal{N}_2 \subset \mathcal{N}_3'$ then $f_{{W}_{i,\nu}}(0) \in \mathcal{N}_3'(f)(0)$ for every $i \in \{1,2\}$ and $\nu \in \N$.  

When $a=0$, i.e.\ $f_{W_1}(0)=0$, then $f_{Q^{\flat}_{1,\mu,\nu}}(0)=0$ for every $\mu,\nu\in\N$ because it is an iterated bracket of the vector fields $f_{W_1}$ and $f_0$ that vanish at $0$, see \eqref{def:Q}. 

When $a'=0$, i.e.\ $f_{Q_{1,1,1}}(0) \in \mathcal{N}_3''(f)(0)$ then
for every $\nu\in\N$, $f_{Q_{1,1,1,\nu}}(0) \in \mathcal{N}_3''(f)(0)$ because $\mathcal{N}_3''$ is stable by right bracketing with $X_0$, thus $f_{Q_{1,1,1,\nu}}(0) \in \mathcal{N}_3'(f)(0)$.

These remarks prove that $\mathcal{N}_3'(f)(0)=\mathcal{N}(f)(0)$ where
\begin{equation} \label{eq:W3-N}
\mathcal{N} = \left\lbrace \begin{array}{l} 
\mathcal{N}_3' \cup \{ W_{1,\nu}, W_{2,\nu}, Q_{1,1,1,\nu}, Q^{\flat}_{1,\mu,\nu} ; \mu,\nu\in\N \} \text{ when } (a,a')=(0,0) , \\
\mathcal{N}_3' \cup \{ W_{1,\nu}, W_{2,\nu},  Q^{\flat}_{1,\mu,\nu} ; \mu,\nu\in\N \} \text{ when } (a,a')=(0,1) , \\
\mathcal{N}_3' \cup \{ W_{1,\nu}, W_{2,\nu}, Q_{1,1,1,\nu},  ; \mu,\nu\in\N \} \text{ when } (a,a')=(1,0) , \\
\mathcal{N}_3' \cup \{ W_{1,\nu}, W_{2,\nu};\mu,\nu\in\N \} \text{ when } (a,a')=(1,1) .
\end{array}\right.
\end{equation}
 By assumption, $f_{W_3}(0) \notin \mathcal{N}_3(f)(0)$, so $f_{W_3}(0) \notin \mathcal{N}(f)(0)$.
    
    We intend to apply \cref{p:PZM-XI} with $M \gets 5$, $L \gets 11$, $\bb \gets W_3$ and $\mathcal{N}$ as in 
    \eqref{eq:W3-N}, so that \eqref{eq:W3-z5} will follow from \eqref{eq:PZM-xibb-OXi}, for the appropriate choice of $\Xi(t,u)$ (corresponding to the quantities within the $O(\cdot)$ in \eqref{eq:W3-z5}).
    Let us check that the required estimates are satisfied.

    \step{Estimates of other coordinates of the second kind}
    Let $b \in \Bs_{\intset{1,5}}$ such that $b \notin \mathcal{N} \cup \{ \bb \}$.
    
    We investigate the different possibilities depending on $n_1(b)$.
    \begin{itemize}
        \item One cannot have $n_1(b) = 1$ since $\Bs_1 \subset \mathcal{N}_3' \subset \mathcal{N}$.
        
        \item If $n_1(b) = 2$, by \eqref{Bstar_S2} and \eqref{eq:W3-N}, one has $b = W_{j,\nu}$ with either ($j \geq 4$) or ($j = 3$ and $\nu \geq 1$).
        Thus $|b| \geq 8$.
        By estimate \eqref{bound-xiWjnu/j0} with $(p,j_0) \gets (1,3)$, \eqref{eq:xib-XI-L} holds with $\sigma = 8$ and
        \begin{equation}
            \Xi(t,u) := t \|u_3\|_{L^2}^2.
        \end{equation}
        
        \item If $n_1(b) = 3$, by \eqref{Bstar_S3} and \eqref{eq:W3-N}, $b = P_{j,l,\nu}$ with $2 \leq j \leq l$.
        Thus $|b| \geq 7$.
        By estimate \eqref{bound-xiPjknu/j0k0} with $(p_1,p_2,j_0, k_0) \gets (3/2,3,2,2)$, \eqref{eq:xib-XI-L} holds with $\sigma=7$ and
        \begin{equation}
            \Xi(t,u) := \|u_2\|_{L^3}^3.
        \end{equation}
        
        \item If $n_1(b) = 4$, by \eqref{Bstar_S4} and \eqref{eq:W3-N}, we are in one of the following cases.
        \begin{itemize}
            \item $b=Q_{1,1,1,\nu}$ with $\nu \geq 1$ and $a'=1$
            thus $|b| \geq 6$ and by estimate \eqref{bound-xiQjklnu/j0k0l0} with 
            
            $(p_1,p_2,p_3,j_0,k_0,l_0) \gets (2,4,4,1,1,1)$, \eqref{eq:xib-XI-L} holds with $\sigma=6$ and
            \begin{equation}
                \Xi(t,u) := a' t \|u_1\|_{L^4}^4 = a' 4! t\xi_{Q_{1,1,1}}(t,u).
            \end{equation}
            
            \item $b = Q_{1,1,l,\nu}$ with $l \geq 3$, thus $|b| \geq 7$ and, by estimate \eqref{bound-xiQjklnu/j0k0l0} with 
            
            $(p_1,p_2,p_3,j_0,k_0,l_0) \gets (3,6,2,1,1,3)$, \eqref{eq:xib-XI-L} holds with $\sigma=7$ and
            \begin{equation}
                \Xi(t,u) := \|u_1\|_{L^6}^3 \|u_3\|_{L^2}.
            \end{equation}
            
            \item $b = Q_{j,k,l,\nu}$ with $2 \leq  k $, thus $|b| \geq 7$ and, by estimate \eqref{bound-xiQjklnu/j0k0l0} with 
            
            $(p_1,p_2,p_3,j_0,k_0,l_0)$ $\gets (3,3,3,1,2,2)$, \eqref{eq:xib-XI-L} holds with $\sigma=7$ and
            \begin{equation}
                \Xi(t,u) := \|u_1\|_{L^6}^2 \|u_2\|_{L^3}^2.
            \end{equation}
            
            \item $b = Q^\sharp_{j,\mu,k,\nu}$, thus $|b| \geq 8$ and, by estimate \eqref{bound-xiQsharpjmuknu/j0k0} with $(p_1,p_2,j_0,k_0) \gets (3,3/2,1,2)$, \eqref{eq:xib-XI-L} holds with $\sigma=8$ and
            \begin{equation}
                \Xi(t,u) := t \|u_1\|_{L^6}^2 \|u_2\|_{L^3}^2.
            \end{equation}
            
            \item $b = Q^\flat_{j,\mu,\nu}$ and $a=1$ thus $|b| \geq 8$ and, by estimate 
            \eqref{bound-xiQflatjmunu/j0} 
            with $(p,j_0) \gets (1,1)$, \eqref{eq:xib-XI-L} holds with $\sigma=7$ and
            \begin{equation}
                \Xi(t,u) :=  a t \|u_1\|_{L^2}^4 = 4  a t \xi^2_{W_1}(t,u).
            \end{equation}
        \end{itemize}
        
        \item If $n_1(b) = 5$, by \eqref{Bstar_S5} and \eqref{def:mathcalE3},  we are in one of the following cases.        
        \begin{itemize}
            \item $b = R_{j,k,l,m,\nu}$ with $m \geq 2$, thus $|b| \geq 7$ and, by estimate \eqref{bound-xiRjklmnu/j0k0l0m0} with
            \\ $(p_1,p_2,p_3,p_4,j_0,k_0,l_0,m_0) \gets (3,6,6,3,1,1,1,2)$, \eqref{eq:xib-XI-L} holds with $\sigma=7$ and
            \begin{equation}
                \Xi(t,u) := \|u_1\|_{L^6}^4 \|u_2\|_{L^3}.
            \end{equation}
            \item $b = R^\sharp_{j,k,l,\mu,\nu}$ with $l \geq 2$, thus $|b| \geq 9$ and, by estimate \eqref{bound-xiRsharpjklmunu/j0k0l0} with $(p,p_1,p_2,j_0,k_0,l_0) \gets (3/2,3,6,1,1,2)$, \eqref{eq:xib-XI-L} holds with $\sigma=9$ and
            \begin{equation}
                \Xi(t,u) := t^{\frac 56} \|u_1\|_{L^6}^3 \|u_2\|_{L^3}^2.
            \end{equation}
        \end{itemize}
    \end{itemize}
    
    \step{Estimates of cross terms}
    Let $q \geq 2$, $b_1 \geq \dotsb \geq b_q \in \Bs \setminus \{ X_0 \}$ such that $n_1(b_1) + \dotsb + n_1(b_q) \leq 5$ and $\supp \mathcal{F}(b_1, \dotsc, b_q) \not \subset \mathcal{N}$.
    
    \medskip \noindent
    We start with preliminary estimates.
    \begin{itemize}
        \item By \eqref{eq:xib-u1-Lk}, for each $b_i \in \Bs$ with $n_1(b_i) \leq 5$, \eqref{eq:xib-othercross-XI-L} holds with $\sigma_i = n_1(b_i) + 1$, $\alpha_i = n_1(b_i) / 6$ and $\Xi(t,u) = t^{6/n_1(b_i)-1} \|u_1\|_{L^6}^6$.
        
        \item If $b_i = M_j$ for some $j \in \intset{0,2}$, by \eqref{xi_S1},
        \begin{equation}
            |\xi_{b_i}(t,u)| = |u_{j+1}(t)| = \frac{t^{|b_i|}}{|b_i|!} t^{-(j+1)} (j+1)! |u_{j+1}(t)|
        \end{equation}
        so \eqref{eq:xib-othercross-XI-L} holds with $\sigma_i = j+1$, $\alpha_i = 1/2$ and $\Xi(t,u) = |(u_1,u_2,u_3)(t)|^2$.
        
        \item If $b_i = M_j$ for $j \geq 3$, by \eqref{bound-xiMj/J0} with $(p,j_0) \gets (2,3)$, \eqref{eq:xib-othercross-XI-L} holds with $\sigma_i = 4$, $\alpha_i = 1/2$ and $\Xi(t,u) = t\|u_3\|_{L^2}^2$.
    \end{itemize}
    We now consider the different possibilities, based on the condition $n_1(b_1)+\dotsb+n_1(b_q) \leq 5$.
    \begin{itemize}
        \item \emph{Case: at least two $b_i \in \Bs_1$}.
        Then, by the preliminary steps, $\alpha \geq 1/2+1/2 = 1$.
        
        \item \emph{Case: $q = 3$, $b_1, b_2 \in \Bs_2$, $b_3 \in \Bs_1$}.
        Then, by the preliminary steps, $\alpha = 1/3+1/3+1/2 > 1$.
        
        \item \emph{Case: $q = 2$, $b_1 \in \Bs_{\intset{3,4}}$, $b_2 \in \Bs_1$}.
        Then, by the preliminary steps, $\alpha = n_1(b_1)/6+1/2 \geq 1$.
        
        \item \emph{Case: $q=2$, $b_1 \in \Bs_2$, $b_2 \in \Bs_1$}.
        Say $b_1 = W_{j,\nu}$ and $b_2 = M_{k-1}$.
        One cannot have $j=1$ because, by \eqref{eq:jacobi.rtl}, $\supp [b_1,b_2]$ would be contained in $ \{ P_{1,k',\nu'}, k' \geq 1, \nu' \geq 0 \} \subset \mathcal{N}$.
        So $j \geq 2$.
        Then, by \eqref{bound-xiWjnu/j0} with $(p,j_0) \gets (3/2,2)$, \eqref{eq:xib-othercross-XI-L} holds for $b_1$ with $\sigma_1 = 5$, $\alpha_1 = 2/3$ and $\Xi(t,u) = t^{1/2}\|u_2\|_{L^3}^3$.
        By the preliminary steps, $\alpha_1+\alpha_2 = 2/3+1/2>1$.
        
        \item \emph{Case: $q=2$, $b_1,b_2 \in \Bs_2$}.
        Say $b_1 = W_{j,\nu}$ and $b_2 = W_{j',\nu'}$.
        \begin{itemize}
        \item If $j = j' = 1$, one cannot have $a=0$ because, by \eqref{eq:jacobi.balance}, $\supp [b_1,b_2]$ would be contained in $\{ Q^{\flat}_{1,\mu',\nu'};\mu',\nu'\in\N\} \subset\mathcal{N}$, see \eqref{eq:W3-N}. So $a=1$. By \eqref{bound-xiWjnu/j0} with $(p,j_0) \leftarrow (1,1)$, \eqref{eq:xib-othercross-XI-L} holds for $\sigma_1=\sigma_2=3$ and $\Xi(t,u) = a \|u_1\|_{L^2}^2= 4 a \xi^2_{W_1}(t,u)$. 
        
        \item If $j \geq 2$, then, by \eqref{bound-xiWjnu/j0} with $(p,j_0) \gets (3/2,2)$, \eqref{eq:xib-othercross-XI-L} holds for $b_1$ with $\sigma_1 = 5$, $\alpha_1 = 2/3$ and $\Xi(t,u) = t^{1/2}\|u_2\|_{L^3}^3$.
        By the preliminary steps, $\alpha_1+\alpha_2 = 2/3+1/3=1$.
        \end{itemize}
        
        \item \emph{Case: $q=2$, $b_1 \in \Bs_3$, $b_2 \in \Bs_2$}.
        Say $b_1 = P_{j,k,\nu}$ and $b_2 = W_{l,\mu}$.
        One cannot have $j=k=l=1$ because, by \eqref{eq:jacobi.balance}, $\supp [b_1,b_2]$ would be contained in $\{ R^\sharp_{1,1,1,\mu',\nu'};\mu',\nu'\in\N \} \subset \mathcal{N}$. Thus $l\geq 2$ or $k\geq 2$.
        \begin{itemize}
        \item If $l\geq 2$, using \eqref{bound-xiWjnu/j0} with $(p,j_0) \leftarrow (3/2,2)$, \eqref{eq:xib-othercross-XI-L} holds for $b_2$ with $\sigma_2 = 5$, $\Xi(t,u)= t^{1/2} \|u_2\|_{L^3}^3$ and $\alpha_2=2/3$. By the preliminary step $\alpha_1+\alpha_2=1/2+2/3>1$.        
        \item If $k \geq 2$, then using \eqref{bound-xiPjknu/j0k0} with $(p_1,p_2,j_0,k_0) \gets (3,3,1,2)$, \eqref{eq:xib-othercross-XI-L} holds for $b_1$ with $\sigma_1 = 5$, $\alpha_1 = 2/3$ and $\Xi(t,u) = t^{1/2} \|u_1\|_{L^6}^3 \|u_2\|_{L^3}^{3/2}$. By the preliminary step $\alpha_1+\alpha_2=2/3+1/2>1$.  \qedhere        
        \end{itemize}
    \end{itemize}
\end{proof}

\subsection{Vectorial relations}

\begin{lemma} \label{p:W3-vectors}
	Under the assumptions of \cref{thm:W3precise},
	\begin{enumerate}
    	\item \label{it:W3-vectors-1} the vectors $f_{M_0}(0), f_{M_1}(0)$, $f_{M_2}(0)$ are linearly independent,
        \item \label{it:W3-vectors-2} if $f_{W_1}(0) \neq 0$ then $f_{W_1}(0) \notin \vect \{f_{M_0}(0), f_{M_1}(0)$, $f_{M_2}(0)\}$. 
	\end{enumerate}
\end{lemma}

\begin{proof}
We have $S_{2,\intset{1,4}}(f)(0) \subset \mathcal{N}_2(f)(0)$ because $S_{2,\intset{1, 4}}(X) = \vect \{W_{j,\nu}; 2j+\nu-1\leq 4 \}$, $f_{W_{1}}(0), f_{W_{2}}(0) \in \mathcal{N}_2(f)(0)$ and $\mathcal{N}_2$ is stable by right bracketing with $X_0$. Thus, since $\mathcal{N}_2 \subset \mathcal{N}_3$,
	\begin{equation} \label{S_2,1-4}
		S_{2,\intset{1,4}}(f)(0)  \subset \mathcal{N}_3(f)(0).
	\end{equation}

	\step{Proof of \cref{it:W3-vectors-1}} We assume there exists $(\beta_0,\beta_1,\beta_2)\in\R^3\setminus\{0\}$ such that $f_{B_1}(0)=0$ where $B_1=\beta_2 M_2+\beta_1 M_1+\beta_0 M_0$. One may assume that $\beta_2 = 1$; otherwise consider $[B_1,X_0]$ or $[[B_1,X_0],X_0]$ and renormalize. Then $f_{B_2}(0)=0$ where
	\begin{equation}
		B_2=\ad_{B_1}^2(X_0)=[M_2+\beta_1 M_1+\beta_0 M_0,M_3+\beta_1 M_2+\beta_0 M_1] \in  W_3 + S_{2,\intset{1,4}}(X)
	\end{equation}
	and \eqref{S_2,1-4} leads to a contradiction with the assumption $f_{W_3}(0) \notin \mathcal{N}_3(f)(0)$.

	\step{Proof of \cref{it:W3-vectors-2}} Proceeding by contradiction, we assume that there exists $(\gamma_0,\gamma_1,\gamma_2)\in\R^3\setminus\{0\}$ such that $f_{B_4}(0)=0$ where $B_4=\gamma_2 M_2+\gamma_1 M_1 + \gamma_0 M_0 + W_1$. Let $\kappa=\max\{j \in \{0,1,2\};\gamma_j \neq 0\} $. Then $f_{B_5}(0)=0$ where $B_5=[B_4 0^{2-\kappa},B_4 0^{3-\kappa}]$ i.e.\ 
    \begin{equation}
        B_5= [\gamma_{\kappa} M_2 + \dots + \gamma_0 M_{2-\kappa}  + W_{1,2-\kappa},\gamma_{\kappa} M_3 + \dots + \gamma_0 M_{3-\kappa}  + W_{1,3-\kappa}]= \gamma_{\kappa}^2 W_3 + Q^{\flat}_{1,2-\kappa} + B_6 + B_7
    \end{equation}
    where
	$B_6 \in \vect\{[M_{l},W_{1,\nu}];l \in\N,\nu\in\N \}$ 
    and 
    $B_7\in S_{2,\intset{1,4}}(X)$. 
	By \eqref{eq:jacobi.rtl}, $\supp B_6 \subset \{ P_{1,l,\nu} ; l\in\N^*,\nu\in\N\} \subset \mathcal{N}_3$.
	Together with \eqref{S_2,1-4}, this leads to a contradiction.
\end{proof}

\subsection{Closed-loop estimates}

\begin{lemma} \label{p:W3-loop}
Under the assumptions of \cref{thm:W3precise},
\begin{equation}
		\label{eq:W3-loop}
		|(u_1,u_2,u_3)(t)| =O\left(|x(t;u)|+ \|u_1\|_{L^3}^3 + \|u_2\|_{L^2}^2 + t^{\frac 12} \|u_3\|_{L^2}\right)
	\end{equation}
and if $f_{W_1}(0) \neq 0$ then
 \begin{equation}
		\label{eq:W3-loop_forW1}
		\xi_{W_1}(t,u) =O\left(|x(t;u)|+ \|u_1\|_{L^3}^3 + \|u_2\|_{L^2}^2 + t^{\frac 12} \|u_3\|_{L^2}\right).
	\end{equation}
 \end{lemma}

\begin{proof}
    By \cref{thm:Key_1} with $M \gets 2$,
    \begin{equation} \label{eq:W3-CL-1}
        x(t;u) = \cZ{2}(t,f,u)(0) + O \left( \|u_1\|_{L^3}^3 + |x(t;u)|^{1+\frac 12} \right).
    \end{equation}
    
\noindent \emph{\textbf{First case: $f_{W_1}(0)=0$.}} 
Then $f_{W_{1,\nu}}(0)=0$ for every $\nu\in\N$ because it is an iterated bracket of the vector fields $f_{W_1}$ and $f_0$ that vanish at $0$.
 Let $i \in \intset{0,2}$ and $\mathbb{P}$ be a component 
along $f_{M_i}(0)$ parallel to $\overline{\mathcal{N}}(f)(0)$
where $\overline{\mathcal{N}} = \{ M_0, M_1, M_2 \} \setminus \{M_i\}$. We have $\overline{\mathcal{N}}(f)(0)=\mathcal{N}(f)(0)$ for 
\begin{equation} \label{W3:N_u123}
\mathcal{N}=\{ M_0, M_1, M_2, W_{1,\nu} ; \nu\in\N \} \setminus \{M_i\}
\end{equation}
because $f_{W_{1,\nu}}(0)=0$ for every $\nu\in\N$.
We intend to apply \cref{p:PZM-XI} with $M \gets 2$, $L \gets 5$, $\bb \gets M_i$ and $\mathcal{N}$ as above, so that \eqref{eq:PZM-xibb-OXi}, for the appropriate choice of $\Xi(t,u)$, will yield
    \begin{equation} \label{eq:W3-CL-2}
        \mathbb{P} \cZ{2}(t,f,u)(0) = u_{i+1}(t) + O\left(|(u_1,u_2,u_3)(t)|^2 + t^{\frac 12}\|u_3\|_{L^2} + \|u_2\|_{L^2}^2 \right).
    \end{equation}
    Then, combining \eqref{eq:W3-CL-1} and \eqref{eq:W3-CL-2} concludes the proof of \eqref{eq:W3-loop}.
    Let us check that the required estimates are satisfied.

    \step{Estimates of other coordinates of the second kind}
    Let $b \in \Bs_{\intset{1,2}}$ such that $b \notin \mathcal{N} \cup \{ \bb \}$.
    \begin{itemize}
        \item If $n_1(b) = 1$, then by \eqref{Bstar_S1} and \eqref{W3:N_u123}, $b = M_j$ for $j \geq 3$.
        Thus $|b| \geq 4$. By \eqref{bound-xiMj/J0} with $(p,j_0) \gets (2,3)$, \eqref{eq:xib-XI-L} holds with $\sigma = 4$ and $\Xi(t,u) := t^{\frac 12} \|u_3\|_{L^2}$.
        
        \item If $n_1(b) = 2$, by \eqref{Bstar_S2} and \eqref{W3:N_u123}, $b = W_{j,\nu}$ with $j \geq 2$.
        Thus $|b| \geq 5$.
        By \eqref{bound-xiWjnu/j0} with $(p,j_0) \gets (1,2)$, \eqref{eq:xib-XI-L} holds with $\sigma = 5$ and $\Xi(t,u) := \|u_2\|_{L^2}^2$.
    \end{itemize}
    
    \step{Estimates of cross terms}
    Let $q \geq 2$, $b_1 \geq \dotsb \geq b_q \in \Bs \setminus \{ X_0 \}$ such that $n_1(b_1) + \dotsb + n_1(b_q) \leq 2$ and $\supp \mathcal{F}(b_1, \dotsc, b_q) \not \subset \mathcal{N}$.
    
    Thus $q = 2$ and $b_1 = M_{j_1}$, $b_2 = M_{j_2}$ for some $j_1, j_2 \in \N$.
    By the preliminary estimates of Step~2 of the proof of \cref{p:W3-Z5}, $b_1$ and $b_2$ satisfy \eqref{eq:xib-othercross-XI-L} with $\Xi(t,u) = |(u_1,u_2,u_3)(t)|^2 + t\|u_3\|_{L^2}^2$ and $\alpha_1 = \alpha_2 = 1/2$.

\medskip

\noindent \emph{\textbf{Second case: $f_{W_1}(0) \neq 0$.}}
First, we apply \cref{p:PZM-XI} with $M \gets 2$, $L \gets 5$, $\bb \gets W_1$ and $\mathcal{N}=\{M_0,M_1,M_2\}$ so that \eqref{eq:PZM-xibb-OXi}, for the appropriate choice of $\Xi$, will yield    
\begin{equation} \label{eq:W3:PZ2=W1+O}
        \mathbb{P} \cZ{2}(t,f,u)(0) = \xi_{W_1}(t,u) + O\left( 
        t \xi_{W_1}(t,u)
        + t^{\frac 12}\|u_3\|_{L^2} 
        + \|u_2\|_{L^2}^2
        + |(u_1,u_2,u_3)(t)|^2 \right)
    \end{equation}    
where $\mathbb{P}$ is a component along $f_{W_1}(0)$ parallel to $\mathcal{N}(f)(0)$. The only difference in the estimates, with respect to the first case above, concerns the estimate of coordinates of the second kind associated with 
$b \in \Bs_{2}$ such that $b \notin \mathcal{N} \cup \{ \bb \}$:
then $b = W_{j,\nu}$ with $(j,\nu)\neq(1,0)$ thus $|b| \geq 4$.
By estimate \eqref{bound-xiWjnu/j0} with $(p,j_0) \gets (1,1)$, \eqref{eq:xib-XI-L} holds with $\sigma = 4$ and $\Xi(t,u) := t \|u_1\|_{L^2}^2 = 2 t \xi_{W_1}(t,u)$.
       
\medskip

 We deduce from \eqref{eq:W3-CL-1} and \eqref{eq:W3:PZ2=W1+O} that
   \begin{equation} \label{eq:W3:W1=O}
       \xi_{W_1}(t,u) = O\left( 
        |x(t;u)| 
        + \|u_1\|_{L^3}^3        
        + t^{\frac 12}\|u_3\|_{L^2} 
        + \|u_2\|_{L^2}^2
        + |(u_1,u_2,u_3)(t)|^2 \right).
    \end{equation}

 By \cref{thm:Key_1} with $M \gets 1$,
    \begin{equation} \label{eq:W3-CL-10}
        x(t;u) = \mathcal{Z}_1(t,f,u)(0) + O \left( \|u_1\|_{L^2}^2 + |x(t;u)|^{2} \right).
    \end{equation}
Let $i \in \intset{0,2}$. By applying \cref{p:PZM-XI} with $M \gets 1$, $L \gets 4$, $\bb \gets M_i$ and $\mathcal{N}=\{M_0,M_1,M_2\} \setminus \{M_i\}$, and using \eqref{bound-xiMj/J0} with $(p,j_0) \gets (2,3)$ to prove the only required estimate, we obtain
\begin{equation} \label{PZ1=ui+O}
\mathbb{P} \mathcal{Z}_1(t,f,u)(0) = u_{i+1}(t) + O \left( t^{\frac 12} \|u_3\|_{L^2} \right).
\end{equation}
We deduce from \eqref{eq:W3-CL-10} and \eqref{PZ1=ui+O} that 
\begin{equation} \label{u123=R1}
|(u_1,u_2,u_3)(t)|=O\left( |x(t;u)|+ t^{\frac 12}\|u_3\|_{L^2}+\|u_1\|_{L^2}^2 \right).
\end{equation}
By incorporating \eqref{eq:W3:W1=O} into \eqref{u123=R1} thanks to $\xi_{W_1}(t,u)=2\|u_1\|_{L^2}^2$ one proves \eqref{eq:W3-loop}.
And by incorporating \eqref{eq:W3-loop} into \eqref{eq:W3:W1=O} one proves \eqref{eq:W3-CL-1}.
\end{proof}

\begin{lemma} \label{p:W3-loop_bis}
	Under the assumptions of \cref{thm:W3precise}, if 
 $f_{Q_{1,1,1}}(0) \notin \mathcal{N}_3''(f)(0)$  
 \begin{equation}
		\label{eq:W3-loop_bis}
  \mathbb{P}_{Q_{1}} x(t;u)= \xi_{Q_{1,1,1}}(t,u)+
  O \left(
  t \xi_{Q_{1,1,1}}(t,u)+
  \xi_{W_3}(t,u)+
  \| u_2 \|_{L^3}^3  +
  |x(t;u)|^{\frac{5}{4} } \right)
	\end{equation} 
where $\mathbb{P}_{Q_1}$ is a component along $f_{Q_{1,1,1}}(0)$ parallel to $(\mathcal{N}_3'' \cap \Bs_{\intset{1,4}})(f)(0)$.
\end{lemma}

\begin{proof}
    By \cref{thm:Key_1} with $M \gets 4$,
    \begin{equation} \label{eq:W3-CL-1_bis}
        x(t;u) = \cZ4(t,f,u)(0) + O \left( \|u_1\|_{L^5}^5 + |x(t;u)|^{1+\frac 15} \right).
    \end{equation}
For $i \in \{1,2\}$, $f_{W_i}(0) \in \mathcal{N}_i(f)(0)$ and $\mathcal{N}_i$ is stable by right bracketing with $X_0$ thus, 
for every $\nu\in\N$, $f_{W_{i,\nu}}(0) \in \mathcal{N}_i(f)(0) \subset (\mathcal{N}_3'' \cap \Bs_{\intset{1,4}})(f)(0)$. 
Thus $(\mathcal{N}_3'' \cap \Bs_{\intset{1,4}}f)(0)=\mathcal{N}(f)(0)$ where  
$\mathcal{N}=(\mathcal{N}_3'' \cap \Bs_{\intset{1,4}}) \cup \{ W_{1,\nu},W_{2,\nu};\nu\in\N\}$ i.e.
\begin{equation} \label{N=N3''&W12nu}
\mathcal{N}=\{ 
        M_{\nu},
        W_{1,\nu} 
        W_{2,\nu},
        P_{1,l,\nu},
        Q_{1,1,2,\nu},
        ; l \in \N^*, \mu,\nu \in \N \}.
 \end{equation}
We intend to apply \cref{p:PZM-XI}  with $M \gets 4$, $L \gets 11$, $\bb \gets Q_{1,1,1}$ and $\mathcal{N}$ as in \eqref{N=N3''&W12nu}, so that \eqref{eq:PZM-xibb-OXi}, for the appropriate choice of $\Xi(t,u)$, will yield
    \begin{equation} \label{eq:W3-CL-2_bis}
    \begin{split}
        \mathbb{P}_{Q_1} \cZ{4}(t,f,u)(0) = 
        \xi_{Q_{1,1,1}}(t,u)       
        + O \big( &
        t \xi_{Q_{1,1,1}}(t,u)
        + \xi_{W_3}(t,u)
            + \|u_2\|_{L^3}^3 \\ &
            + \|u_1\|_{L^6}^3 \|u_3\|_{L^2} 
            + \|u_1\|_{L^6}^6  
            + |(u_1,u_2,u_3)(t)|^2   
          \big).
    \end{split}
    \end{equation}
All the needed estimates are contained in the proof of \cref{p:W3-Z5}, except for
\begin{itemize}
\item $b \in \Bs_2$, i.e. $b=W_{j,\nu}$ with $j \geq 3$, for which $|b|\geq 7$ and \eqref{bound-xiWjnu/j0} with $(p,j_0) \leftarrow (1,3)$ proves that \eqref{eq:xib-XI-L} holds with $\sigma=7$ and $\Xi(t,u) := \|u_3\|_{L^2}^2 = 2 \xi_{W_3}(t,u)$,
\item $b=Q^{\flat}_{j,\mu,\nu}$ for which $|b| \geq 7$ and \eqref{bound-xiQflatjmunu/j0} with $(p,j_0)\leftarrow (2,1)$ proves that \eqref{eq:xib-XI-L} holds with $\sigma=7$ and
$\Xi(t,u) :=  t^2 \|u_1\|_{L^4}^4 = 4! t^2 \xi_{Q_{1,1,1}}(t,u)$,       
\item $q=2$, $b_1,b_2 \in \Bs_2$ for which $|b_i|\geq 3$ and \eqref{bound-xiWjnu/j0} with $(p,j_0) \leftarrow (2,1)$ proves that \eqref{eq:xib-othercross-XI-L} holds for $\sigma_i=3$, $\alpha_i=1/2$ and $\Xi(t,u) =  t \|u_1\|_{L^4}^4 = 4 t \xi_{Q_{1,1,1}}(t,u)$. 
\end{itemize}
Then, using $\|u_1\|_{L^6}^6=O(t^2 \|u_1\|_{L^4}^4)$, we deduce from \eqref{eq:W3-CL-2_bis} that
\begin{equation} \label{eq:W3-CL-2_bis_2}
    \begin{split}
        \mathbb{P}_{Q_1} \cZ{4}(t,f,u)(0) = 
        \xi_{Q_{1,1,1}}(t,u)       
        + O \big( &
        t \xi_{Q_{1,1,1}}(t,u)
        + \xi_{W_3}(t,u)
            + \|u_2\|_{L^3}^3 
            + |(u_1,u_2,u_3)(t)|^2   
          \big).
    \end{split}
    \end{equation}
    Then, by combining \eqref{eq:W3-CL-1_bis}, \eqref{eq:W3-CL-2_bis_2} 
    and $\|u_1\|_{L^5}^5=O(t\|u_1\|_{L^4}^4)$
    we obtain 
    \begin{equation}
        \mathbb{P}_{Q_{1}} x(t;u)= \xi_{Q_{1,1,1}}(t,u)+
  O \left(
  t \xi_{Q_{1,1,1}}(t,u)+
  \xi_{W_3}(t,u)+
  \| u_2 \|_{L^3}^3  +
  |(u_1,u_2,u_3)(t)|^2 +  
  |x(t;u)|^{\frac{5}{4} } \right).
    \end{equation}
  Finally by incorporating \eqref{eq:W3-loop} in the previous estimate we obtain
\eqref{eq:W3-loop_bis}.
\end{proof}

\subsection{Interpolation inequalities}

\begin{lemma}
	There exists $C > 0$ such that, for every $t > 0$ and $u \in \lone$,
	\begin{align}
		\label{eq:w3-u2}
		\|u_2\|_{L^3}^3 & \leq C \|u\|_{L^\infty} \|u_3\|_{L^2}^2, \\
		\label{eq:w3-u1}
		\|u_1\|_{L^6}^6 & \leq C \|u\|_{L^\infty}^4 \|u_3\|_{L^2}^2.
	\end{align}
\end{lemma}

\begin{proof}
	Inequality \eqref{eq:w3-u2} follows from \cref{thm:GN} with $\phi \gets u_3$, $(p,q,r,s) \gets (3,2,\infty,2)$, $(j,l) \gets (1,3)$, $\alpha \gets 1/3$.
	Similarly, \eqref{eq:w3-u1} follows from \cref{thm:GN} with $\phi \gets u_3$, $(p,q,r,s) \gets (6,2,\infty,2)$, $(j,l) \gets (2,3)$, $\alpha \gets 2/3$.
    In both cases, the lower-order term in \eqref{eq:GNS-estimate} is absorbed using the estimate $\|u_3\|_{L^2} \leq t^{\frac 72} \|u\|_{L^\infty}$, which stems from Hölder's inequality and the equality $u_3(t) = \frac 1 2 \int_0^t (t-s)^2 u(s) \dd s$.
\end{proof}

\subsection{Proof of the presence of the drift}

\begin{proof}[Proof of \cref{thm:W3precise}]
    In this proof, to lighten the notations, we write $x$ and $\xi_b$ instead of $x(t;u)$ and $\xi_b(t,u)$.
    By \cref{thm:Key_1} with $M \gets 5$,
    \begin{equation} \label{x=Z5+_W3}
        x = \cZ{5}(t,f,u)(0) + O\left( \|u_1\|_{L^6}^6 + |x|^{1+\frac{1}{5}} \right).
    \end{equation}
    We deduce from \eqref{x=Z5+_W3}, \eqref{eq:W3-z5}, \eqref{eq:W3-loop} and \eqref{eq:W3-loop_forW1} that
    \begin{equation}
        \mathbb{P}_{W_3} x = 
        \xi_{W_3}           
        + O\Big( 
        t \xi_{W_3}
        + \|u_2\|_{L^3}^3
        + \|u_1\|_{L^6}^3 \|u_3\|_{L^2} 
        + \|u_1\|_{L^6}^6  
        + a' t \xi_{Q_{1,1,1}}
        + |x|^{\frac 65}
        \Big).
    \end{equation}
    We deduce from the above estimate (see \cref{s:O}) the existence of $C,\rho>0$ such that, for every $t \in (0,\rho)$, $u \in \lone$ with $\|u\|_{W^{-1,\infty}}\leq \rho$,
    \begin{equation}
        \mathbb{P}_{W_3}x
        + C a' t \xi_{Q_{1,1,1}}
        \geq 
        (1-Ct) \xi_{W_3}       
        - C \Big( 
        \|u_2\|_{L^3}^3
        + \|u_1\|_{L^6}^3 \|u_3\|_{L^2} 
        \\  
        + \|u_1\|_{L^6}^6  
        + |x|^{\frac 65}
        \Big)
    \end{equation}    
    where, when $a'\neq 0$, by  \eqref{eq:W3-loop_bis},
    \begin{equation}
        \xi_{Q_{1,1,1}} \leq 2 \mathbb{P}_{Q_1} x + C \left( \xi_{W_3} + \|u_2\|_{L^3}^3 + |x|^{{\frac 54}} \right).
    \end{equation}
    Thus
    \begin{equation}
        \mathbb{P}_{W_3} x
        + 2 C a' t \mathbb{P}_{Q_1}x
        \geq 
        (1-2Ct) \xi_{W_3}            
        - 2C \Big( 
        \|u_2\|_{L^3}^3
        + \|u_1\|_{L^6}^3 \|u_3\|_{L^2} 
        + \|u_1\|_{L^6}^6  
        + |x|^{\frac 65}
        \Big).
    \end{equation}    
    Gathering this inequality and the interpolation estimates \eqref{eq:w3-u2} and \eqref{eq:w3-u1} yields
    \begin{equation}
        (\mathbb{P}_{W_3} + 2 C a' t \mathbb{P}_{Q_1}) x
        \geq {\frac 12} \int_0^t u_3^2            
        - C' \Big( 
        \left(t+\|u\|_{L^\infty}\right) \int_0^t u_3^2
        + |x|^{\frac 65}
         \Big)
    \end{equation}
    for some constant $C'$ independent of $t$ and $u$.
\end{proof}

\subsection{A comment on the time-dependent drift direction}
\label{s:W3-comment}

As already stated, the conclusion \eqref{drift_Pt} is not exactly a drift in the sense of \cref{Def:drift}, since the left-hand side of the inequality involves a time-dependent component $\mathbb{P}_t$ that may not give a component exactly along $f_{W_3}(0)$. 
This is not a technical limitation of our approach, since this phenomenon does occur.
Consider the system
\begin{equation} \label{ex:W3-time}
    \begin{cases}
        \dot{x}_1 = u \\
        \dot{x}_2 = x_1 \\
        \dot{x}_3 = x_2 \\
        \dot{x}_4 = x_1^4 + x_3^3 \\
        \dot{x}_5 = x_3^2 - x_4.
    \end{cases}
\end{equation}
One can check that this system satisfies the assumptions of \cref{thm:W3precise}.
Moreover, by Hermes' condition \cite[Theorem 3.2]{hermes1982control}, the subsystem $(x_1,x_2,x_3,x_4)$ is $L^\infty$-STLC.
However, using oscillating controls as in \cref{Subsec_Optim_Wk}, one can prove that there is no drift along $f_{W_3}(0) = 2 e_5$, parallel to $\mathcal{N}_3(f)(0) = \R e_1 + \dotsb + \R e_4$ as $(t,\|u\|_{L^\infty})\to 0$.
Heuristically, the proof of \cref{thm:W3precise} considers the quantity $y(t;u) := x_5(t;u) + t x_4(t;u)$ to obtain a drift in a direction which tends towards~$f_{W_3}(0)$.

Another way to look at this system is to consider $z(t;u) := x_5(t;u) + x_4(t;u)$.
One computes
\begin{equation}
    z(t;u) = \int_0^t (1 - (t-s)) u_1^4(s) \dd s
    + \int_0^t (1 + u_3(s) (1-(t-s)) u_3^2(s) \dd s.
\end{equation}
Hence, as $(t,\|u\|_{L^\infty}) \to 0$, 
\begin{equation}
    z(t;u) \geq (1-\varepsilon) \left( \int_0^t u_1^4 + \int_0^t u_3^2 \right),
\end{equation}
so there somehow is a strong composite drift in the (fixed) direction $e_4 + e_5$.

\section{New obstruction of the sixth order}
\label{s:sextic}

The goal of this section is to prove \cref{thm:sextic}, as a consequence of the following more precise statement.
In this section, we use the short-hand notation $D$ for the following bracket of $\Bs_6$:
\begin{equation}
    D := \ad_{P_{1,1}}^2 (X_0)
\end{equation}
and we introduce
\begin{equation}
    \mathcal{N}_D := \Bs_{\intset{1,7}} \setminus \{ D \}.
\end{equation}

\begin{theorem} \label{thm:sextic-drift}
    Assume that $f_{D}(0) \notin \mathcal{N}_D(f)(0)$.
    Then system \eqref{syst} has a drift along $f_{D}(0)$, parallel to $\mathcal{N}_D(f)(0)$, as $(t,\|u\|_{L^\infty}) \to 0$.
\end{theorem}

\subsection{Limiting examples}
\label{s:D-examples}

Let us give an example motivating the threshold 7 for this necessary condition.
In \cite[Example 6.1]{zbMATH04154295}, Kawski considers the systems
\begin{equation}
    \begin{cases}
        \dot{x}_1 = u \\
        \dot{x}_2 = x_1 \\
        \dot{x}_3 = x_1^3 \\
        \dot{x}_4 = x_3^2 - x_2^p
    \end{cases}
\end{equation}
for $p \in \{ 7, 8 \}$.
Written in the form \eqref{syst}, these systems satisfy
\begin{equation}
    f_{M_0}(0) = e_1, \quad 
    f_{M_1}(0) = e_2, \quad
    f_{P_{1,1}}(0) = 6 e_3, \quad
    f_{D}(0) = 72 e_4, \quad
    f_{\ad^{p}_{M_1}(X_0)}(0) = - p! e_4
\end{equation}
and $f_b(0) = 0$ for all $b \in \Bs \setminus \{ M_0, M_1, P_{1,1}, D, \ad^p_{M_1}(X_0) \}$.
Thus, they feature a competition between $D$ and $\ad^p_{M_1}(X_0)$.

Kawski proves that this system is $L^\infty$-STLC for $p = 7$ (see \cite[Claim 6.1]{zbMATH04154295}) but not $L^\infty$-STLC for $p = 8$ (see \cite[Claim 6.3]{zbMATH04154295}).
This both motivates and is consistent with \cref{thm:sextic}, which can be seen as a generalization of Kawski's negative claim.

\begin{remark}
    As \cref{Thm:Kawski_Wm}, \cref{thm:sextic} is a \emph{loose} condition, in the sense that we have not attempted to separate, within $\Bs_6$ and $\Bs_7$, which brackets can or cannot compensate for the drift.
    It is possible that our method could also be used to perform such a distinction.
    
    An interesting example is studied by Kawski in \cite[Example 5.3]{zbMATH04154295};
    \begin{equation}
        \begin{cases}
            \dot{x}_1 = u \\
            \dot{x}_2 = x_1 \\
            \dot{x}_3 = x_1^3 \\
            \dot{x}_4 = x_3^2 - x_2^2 x_1^4,
        \end{cases}
    \end{equation}
    which exhibits in $\Bs$ a competition between $D$ and $\ad^2_{M_1} \ad^4_{X_1} (X_0)$.
    Kawski proves that this systems is $L^\infty$-STLC.
    
    Conversely, the system
    \begin{equation}
        \begin{cases}
            \dot{x}_1 = u \\
            \dot{x}_2 = x_1 \\
            \dot{x}_3 = x_1^3 \\
            \dot{x}_4 = x_3^2 + x_3 x_1^4
        \end{cases}
    \end{equation}
    exhibits in $\Bs$ a competition between $D$ and $\ad_{P_{1,1}} \ad_{X_1}^4 (X_0)$ because $f_{\ad_{P_{1,1}} \ad_{X_1}^4 (X_0)}(0) = 144 e_4$.
    Using the estimates of the next paragraphs, one can prove that this system is not $L^\infty$-STLC.
    This hints towards the fact that it is not necessary to include the bracket $\ad_{P_{1,1}} \ad_{X_1}^4 (X_0)$ (of $\Bs_7$) in the list of brackets which can compensate $D$.
\end{remark}

\subsection{Algebraic preliminaries}
\label{s:sextic-algebra}

To lighten the proof of the following paragraph, we start with algebraic preliminaries concerning the expansions on $\Bs$ of some brackets of order 6, linked with cross terms along $D$. 
We use the trailing zero notation of \cref{def:0nu} and compute the expansions of the considered brackets on $\Bs$ using Jacobi's identity as many times as necessary (see \cite[Section 2.1]{A1} for an exposition and a more theoretical point of view on the classical recursive decomposition algorithm on Hall bases).

For $B \in \mathcal{L}(X)$, $\langle B, D \rangle$ denotes the coefficient of $D$ in the expansion of $B$ on $\Bs$.

\subsubsection{Brackets of two elements}
    
\begin{lemma} \label{p:D-dec-3+3}
    Let $a < b \in \Bs_3$ such that $\langle [a,b], D \rangle \neq 0$.
    Then $a = P_{1,1}$ and $b = P_{1,1} 0$.
\end{lemma}

\begin{proof}
    First $n_0(a) + n_0(b) = n_0(D) = 3$.
    Thus $a = P_{1,1}$ and $b \in \{ P_{1,1} 0, P_{1,2} \}$.
    Since $(P_{1,1},P_{1,2}) \in \Bs \setminus \{ D \}$, the conclusion follows.
\end{proof}

\begin{lemma} \label{p:D-dec-2+4}
    Let $a \in \Bs_2$ and $b \in \Bs_4$.
    Then $\langle [a,b], D \rangle = 0$.
\end{lemma}

\begin{proof}
    First $n_0(a) + n_0(b) = n_0(D) = 3$.
    Since $n_0(b) \geq 1$, $n_0(a) \in \intset{0,2}$ so $a \in \{ W_1, W_1 0 \}$.
    \begin{itemize}
        \item Case $a = W_1 0$.
        Then $b = \ad_{X_1}^4(X_0)$ and $[a,b] = [W_10, \ad_{X_1}^4(X_0)]$, which is in $\Bs \setminus \{ D \}$.
        
        \item Case $a = W_1$.
        Then either,
        \begin{itemize}
            \item $b = \ad_{X_1}^4(X_0) 0$ and 
            \begin{equation}
                [a,b] = [W_1, \ad_{X_1}^4(X_0)] 0 - [W_1 0, \ad_{X_1}^4(X_0)],
            \end{equation}
            both terms being in $\Bs \setminus \{ D \}$.
            \item $b = (M_1, \ad_{X_1}^3(X_0))$ and $[a,b] = [W_1, [M_1, \ad_{X_1}^3(X_0)]]$, which is in $\Bs \setminus \{ D \}$.
        \end{itemize}
    \end{itemize}
    Hence, in all cases $\langle [a,b], D \rangle = 0$.
\end{proof}

\begin{lemma} \label{p:D-dec-1+5}
    Let $a \in \Bs_1$ and $b \in \Bs_5$, such that $\langle [a,b], D \rangle \neq 0$.
    Then $a = X_1$ and $b = R^\sharp_{1,1,1,1}$.
\end{lemma}

\begin{proof}
    First $n_0(a) + n_0(b) = n_0(D) = 3$.
    Since $n_0(b) \geq 1$, $n_0(a) \in \intset{0,2}$ so $a \in \{ X_1, M_1, M_2 \}$.
    \begin{itemize}
        \item Case $a = M_2$. 
        Then $b = \ad_{X_1}^5(X_0)$ and $[a,b] = [M_2, \ad_{X_1}^5(X_0)]$, which is in $\Bs \setminus \{ D \}$.
        
        \item Case $a = M_1$.
        Then either,
        \begin{itemize}
            \item $b = \ad_{X_1}^5(X_0) 0$ and
            \begin{equation}
                [a, b] = [M_1, \ad_{X_1}^5(X_0)] 0 - [M_2, \ad_{X_1}^5(X_0)],
            \end{equation}
            both terms being in $\Bs \setminus \{D\}$.
            \item $b = (M_1, \ad_{X_1}^4(X_0))$ and $[a, b] = \ad_{M_1}^2 \ad_{X_1}^4 (X_0)$, which is in $\Bs \setminus \{ D \}$.
            \item $b = (W_1, \ad_{X_1}^3(X_0))$ and 
            \begin{equation}
                [a, b] = [W_1, [M_1, P_{1,1}]] - [P_{1,1}, P_{1,2}],
            \end{equation}
            both terms being in $\Bs \setminus \{D\}$.
        \end{itemize}
        
        \item Case $a = X_1$. Then either,
        \begin{itemize}
            \item $b = \ad_{X_1}^5(X_0) 0^2$ and
            \begin{equation}
                [a,b] = \ad_{X_1}^6 (X_0) 0^2 - 2 [M_1, \ad_{X_1}^5 (X_0)] + [M_2, \ad_{X_1}^5(X_0)],
            \end{equation}
            all terms being in $\Bs \setminus \{ D \}$.
            
            \item $b = (M_1, \ad_{X_1}^4(X_0)) 0$ and
            \begin{equation}
                [a,b] = [W_1, \ad_{X_1}^4(X_0)] 0 + [M_1, \ad_{X_1}^5(X_0)] 0 - \ad_{M_1}^2 \ad_{X_1}^4(X_0),
            \end{equation}
            all terms being in $\Bs \setminus \{ D \}$.
            
            \item $b = \ad_{M_1}^2 \ad_{X_1}^3(X_0)$ and
            \begin{equation}
                [a, b] = 2 [W_1, [M_1, P_{1,1}]] + \ad_{M_1}^2 \ad_{X_1}^4(X_0) - [P_{1,1}, P_{1,2}],
            \end{equation}
            all terms being in $\Bs \setminus \{D\}$.
            
            \item $b = (M_2, \ad_{X_1}^4(X_0))$ and
            \begin{equation}
                [a,b] = [W_10, \ad_{X_1}^4(X_0)] + [M_2, \ad_{X_1}^5(X_0)],
            \end{equation}
            both terms being in $\Bs \setminus \{ D \}$.
            
            \item $b = (W_1, \ad_{X_1}^3(X_0)) 0$ and
            \begin{equation}
                [a,b] = [W_1, \ad_{X_1}^4(X_0)] 0 - [W_1, [M_1, P_{1,1}]] + [P_{1,1}, P_{1,2}]
            \end{equation}
            all terms being in $\Bs \setminus \{ D \}$.
            
            \item $b = (W_1 0, \ad_{X_1}^3(X_0))$ and
            \begin{equation}
                [a,b] = - D + [P_{1,1}, P_{1,2}] + [W_1 0, \ad_{X_1}^4(X_0)],
            \end{equation}
            so $\langle [a,b], D \rangle =  -1$.
            
            \item $b = (W_1, (M_1, W_1))$ and
            \begin{equation}
                [a, b] = [P_{1,1}, P_{1,2}] + [W_1,[M_1,\ad_{X_1}^3(X_0)]
            \end{equation}
            both terms being in $\Bs \setminus \{D\}$.
        \end{itemize}
    \end{itemize}
    Hence, the only case where $\langle [a,b], D \rangle =  -1 \neq 0$ is $a = X_1$ and $b = (W_1 0, \ad_{X_1}^3(X_0)) = R^\sharp_{1,1,1,1}$.
\end{proof}

\subsubsection{Brackets of three elements}

\begin{lemma} \label{p:D-dec-2+2+2}
    For every $a, b, c \in \Bs_2$, $\langle [a,[b,c]], D \rangle = 0$.
\end{lemma}

\begin{proof}
    By contradiction, assume that $\langle [a,[b,c]], D \rangle \neq 0$. 
    Then $n_0(a)+n_0(b)+n_0(c) = 3$.
    Thus $a = b = c = W_1$, so $[a,[b,c]] = 0$.
\end{proof}

\begin{lemma} \label{p:D-dec-1+2+3}
    Let $a \in \Bs_1$, $b \in \Bs_2$, $c \in \Bs_3$ such that $\langle [a, [b, c]], D \rangle \neq 0$ or $\langle [[a,b],c], D \rangle \neq 0$.
    Then $a = X_1$ and, either ($b = W_1 0$ and $c = P_{1,1}$) or ($b = W_1$ and $c = P_{1,1} 0$).
\end{lemma}

\begin{proof}
    First $n_0(a)+n_0(b)+n_0(c) = 3$.
    
    \step{First form: $[a,[b,c]]$}
    \begin{itemize}
        \item Case $a = M_1$. 
        Then $b = W_1$, $c = P_{1,1}$ and
        \begin{equation}
            [a,[b,c]] = -[P_{1,1}, P_{1,2}] + [W_1, [M_1, P_{1,1}]],
        \end{equation}
        both terms being in $\Bs \setminus \{ D \}$.
        
        \item Case $a = X_1$.
        
        \begin{itemize}
            \item Case $b = W_1 0$. Then $c = P_{1,1}$ and $\langle [a, [b,c]], D \rangle = -1$.
            
            \item Case $b = W_1$. Then either, 
            \begin{itemize}
                \item $c = P_{1,1} 0$ and $\langle [a, [b,c]], D \rangle = +1$.
                \item $c = P_{1,2}$ and $[a,[b,c]] = [W_1, [M_1, P_{1,1}]] + [P_{1,1}, P_{1,2}]$
                both terms being in $\Bs \setminus \{ D \}$.
            \end{itemize}
        \end{itemize}
    \end{itemize}
    
    \step{Second form: $[[a,b],c]$}
    \begin{itemize}
        \item Case $a = M_1$.
        Then $b = W_1$ and $c = P_{1,1}$ and $[[a,b],c]=-[P_{1,1},P_{1,2}]$ which is in $\Bs \setminus \{ D \}$.
        
        \item Case $a = X_1$.
        \begin{itemize}
            \item Case $b = W_1 0$.
            Then $c = P_{1,1}$ and $\langle [[a,b],c], D \rangle = -1$.
            
            \item Case $b = W_1$. Then either
            \begin{itemize}
                \item $c = P_{1,1} 0$ and $\langle [[a,b],c], D \rangle = 1$.
                \item $c = P_{1,2}$ and $[[a,b],c] = [P_{1,1}, P_{1,2}]$, which is in $\Bs \setminus \{ D \}$.
            \end{itemize}
        \end{itemize}
    \end{itemize}
    This concludes the case disjunction.
\end{proof}

\begin{lemma} \label{p:D-dec-1+1+4}
    Let $a, b \in \Bs_1$ and $c \in \Bs_4$ such that $\langle [a, [b, c]], D \rangle \neq 0$, or $\langle [[a,b],c] \rangle \neq 0$.
    Then $a = b = X_1$.
\end{lemma}

\begin{proof}
    First $n_0(a)+n_0(b)+n_0(c) = 3$.
    
    \step{First form: $[a,[b,c]]$ with $a \leq b$}
    \begin{itemize}
        \item Case $a = b = M_1$. Then $c = \ad_{X_1}^4(X_0)$ and $[a,[b,c]] = \ad_{M_1}^2 \ad_{X_1}^4(X_0)$, which is in $\Bs \setminus \{ D \}$.
        
        \item Case $a = X_1$, $b = M_2$. Then $c = \ad_{X_1}^4(X_0)$ and 
        \begin{equation}
            [a,[b,c]] = [W_1 0, \ad_{X_1}^4(X_0)] + [M_2, \ad_{X_1}^5(X_0)],
        \end{equation}
        both terms being in $\Bs \setminus \{ D \}$.
        
        \item Case $a = X_1$, $b = M_1$.
        Then either,
        \begin{itemize}
            \item $c = \ad_{X_1}^4(X_0) 0$ and
            \begin{equation}
                \begin{split}
                    [a,[b,c]] = [W_1, \ad_{X_1}^4(X_0)]0 & - [W_1 0, \ad_{X_1}^4(X_0)] - [M_2, \ad_{X_1}^5(X_0)] \\
                    & - \ad_{M_1}^2 \ad_{X_1}^4(X_0) + [M_1, \ad_{X_1}^5(X_0)]0 ,
                \end{split}
            \end{equation}
            all terms being in $\Bs \setminus \{ D \}$.
            
            \item $c = (M_1, \ad_{X_1}^3(X_0))$ and 
            \begin{equation}
                [a,[b,c]] = - [P_{1,1}, P_{1,2}] + 2 [W_1, [M_1, \ad_{X_1}^3(X_0)]] + \ad_{M_1}^2 \ad_{X_1}^4 (X_0),
            \end{equation}
            all terms being in $\Bs \setminus \{ D \}$.
        \end{itemize}
        
        \item Case $a = b = X_1$.
        One may have $\langle [a,[b,c]], D \rangle \neq 0$.
        Since the conclusion of the lemma does not concern $c$, we do not need to study all possible cases.
    \end{itemize}
    Thus, the only case leading to a (possibly) nonzero value of $\langle [a,[b,c]], D \rangle$ is $a = b = X_1$.
    
    \step{Second form: $[[a,b],c]$ with $a < b$}
    Since $n_0(a) + n_0(b) \leq 2$, $a = X_1$ and $b = M_1$.
    Thus $[a,b] = W_1$.
    By \cref{p:D-dec-2+4}, $\langle [W_1,c], D \rangle = 0$.
    
    \step{Third form: $[a,[b,c]]$ with $a > b$}
    Then $[a,[b,c]] = [[a,b],c] + [b,[a,c]]$ so the conclusions of the previous forms apply.
\end{proof}

\subsection{Dominant part of the logarithm}

\begin{lemma}
    Assume that $f_{D}(0) \notin \mathcal{N}_D(f)(0)$.
    Let $\mathbb{P}$ be a component along $f_D(0)$ parallel to $\mathcal{N}_D(f)(0)$.
    Then
    \begin{equation} \label{eq:D-Z7}
        \begin{split}
            \mathbb{P} \cZ{7}(t,f,u)(0) = \xi_{D}(t,u) + O\Big(
            |u_1(t)|^4 & +
            |\xi_{P_{1,1}}(t,u)|^2 +
            |\xi_{P_{1,1} 0}(t,u)|^2 \\
            & + |u_1(t) \xi_{R^\sharp_{1,1,1,1}}(t,u)| +
            \|u_1\|_{L^8}^8
            \Big).
        \end{split}
    \end{equation}
\end{lemma}

\begin{proof}
    We start with a preliminary estimate.
    By \eqref{eq:xib-u1-Lk} and Hölder's inequality, there exists $c > 0$ such that, for every $t \leq 1$, $u \in \lone$ and $b \in \Bs_{\intset{1,6}} \setminus \{ X_1 \}$,
    \begin{equation} \label{eq:D-xib-univ}
        |\xi_b(t,u)| \leq c \|u_1\|_{L^8}^{n_1(b)}.
    \end{equation}
    By \eqref{eq:ZM-etab} and definition of $\mathbb{P}$, 
    \begin{equation}
        \mathbb{P} \cZ{7}(t,f,u)(0) = \eta_D(t,u).
    \end{equation}
    To apply \cref{p:etab-xib-XI}, let us prove that, for every $q \geq 2$, $b_1 \geq \dotsb \geq b_q \in \Bs$ such that $D \in \supp \mathcal{F}(b_1,\dotsc,b_q)$, for every $t > 0$ and $u \in \lone$, the estimate \eqref{eq:prod-xibi-XI} holds, for an appropriate choice of $\Xi$.
    We split cases depending on $q$.
    
    \step{Case $q = 2$}
    \begin{itemize}
        \item Case $n_1(b_1) = 5$ and $n_1(b_2) = 1$.
        By \cref{p:D-dec-1+5}, $b_1 = R^\sharp_{1,1,1,1}$ and $b_2 = X_1$ so \eqref{eq:prod-xibi-XI} holds with $\Xi(t,u) := |u_1(t) \xi_{R^\sharp_{1,1,1,1}}(t,u)|$.
        
        \item Case $n_1(b_1) = 4$ and $n_1(b_2) = 2$.
        By \cref{p:D-dec-2+4}, $D \notin \supp \mathcal{F}(b_1,b_2)$ in this case.
        
        \item Case $n_1(b_1) = 3$ and $n_1(b_2) = 3$.
        By \cref{p:D-dec-3+3}, $b_1 = P_{1,1} 0$ and $b_2 = P_{1,1}$ so \eqref{eq:prod-xibi-XI} holds with $\Xi(t,u) := |\xi_{P_{1,1}}(t,u) \xi_{P_{1,1}0}(t,u)|$.
        
    \end{itemize}
    
    \step{Case $q = 3$}
    \begin{itemize}
        \item Case $n_1(b_1) = 4$, $n_1(b_2) = 1$, $n_1(b_3) = 1$.
        By \cref{p:D-dec-1+1+4}, $b_2 = b_3 = X_1$.
        Hence, using \eqref{eq:D-xib-univ},
        \eqref{eq:prod-xibi-XI} holds with $\Xi(t,u) := c |u_1(t)|^2 \|u_1\|_{L^8}^4$.
        
        \item Case $n_1(b_1) = 3$, $n_1(b_2) = 2$, $n_1(b_3) = 1$.
        By \cref{p:D-dec-1+2+3}, $b_3 = X_1$ and $b_1 \in \{ P_{1,1}, P_{1,1} 0 \}$.
        Hence, using \eqref{eq:D-xib-univ}, \eqref{eq:prod-xibi-XI} holds with $\Xi(t,u) := c |u_1(t)| \|u_1\|_{L^8}^2 (|\xi_{P_{1,1}}(t,u)| + |\xi_{P_{1,1}0}(t,u)|)$.
        
        \item Case $n_1(b_1) = 2$, $n_1(b_2) = 2$, $n_1(b_3) = 2$.
        By \cref{p:D-dec-2+2+2}, $D \notin \supp \mathcal{F}(b_1,b_2,b_3)$ in this case.
    \end{itemize}
    
    \step{Case $q = 4$}
    \begin{itemize}
        \item Case $n_1(b_1) = 3$, $n_1(b_2) = 1$, $n_1(b_3) = 1$, $n_1(b_4) = 1$.
        Counting the occurrences of $X_0$ and using  \eqref{eq:D-xib-univ} implies that either,
        \begin{itemize}
            \item $b_3 = b_4 = X_1$, and \eqref{eq:prod-xibi-XI} holds with $\Xi(t,u) :=  c \|u_1\|_{L^8}^4 |u_1(t)|^2$.
            \item $b_1 = P_{1,1}$, $b_2 = b_3 = M_1$ and $b_4 = X_1$, and thus \eqref{eq:prod-xibi-XI} holds with $\Xi(t,u) := |\xi_{P_{1,1}}(t,u)| \|u_1\|_{L^8}^2 |u_1(t)|$.
        \end{itemize}
        
        \item Case $n_1(b_1) = 2$, $n_1(b_2) = 2$, $n_1(b_3) = 1$, $n_1(b_4) = 1$.
        Counting the occurrences of $X_0$ and using  \eqref{eq:D-xib-univ} implies that either,
        \begin{itemize}
            \item $b_1 = b_2 = W_1$, $b_3 = M_1$ and $b_4 = X_1$ and $D \notin \supp \mathcal{F}(b_1,b_2,b_3,b_4)$. 
            Indeed, a non-zero bracket of $W_1$, $W_1$, $M_1$ and $X_1$ is either a bracket over ($M_1$, $W_1$ and $(X_1, W_1)$) or over ($X_1$, $W_1$ and $(M_1,W_1)$).
            But such brackets have a vanishing coefficient along $D$ by \cref{p:D-dec-1+2+3}.
            
            \item $b_1 = W_1 0$, $b_2 = W_1$, $b_3 = b_4 = X_1$ and \eqref{eq:prod-xibi-XI} holds with $\Xi(t,u) := \|u_1\|_{L^8}^4 |u_1(t)|^2$.
        \end{itemize}
    \end{itemize}
    
    \step{Case $q \in \{ 5, 6 \}$}
    Counting the occurrences of $X_0$ implies that $b_{q-1} = b_q = X_1$.
    Using \eqref{eq:D-xib-univ} implies that \eqref{eq:prod-xibi-XI} holds with $\Xi(t,u) := (1+c)^4 |u_1(t)|^k \|u_1\|_{L^8}^{6-k}$ for some $k \in \intset{2,5}$.
    
    \step{Conclusion}
    Gathering the previous estimates and using Young's inequality proves \eqref{eq:D-Z7}.
\end{proof}

\subsection{Vectorial relations}

\begin{lemma} \label{p:D-vectors}
    Assume that $f_{D}(0) \notin \mathcal{N}_D(f)(0)$.
    Then
    \begin{enumerate}
        \item $f_{X_1}(0) \notin \vect \{ f_{b}(0) ; b \in \Bs_1 \setminus \{ X_1 \} \}$,
        \item $f_{P_{1,1}}(0) \notin \vect \{ f_{b}(0) ; b \in \Bs_{\intset{1,3}} \setminus \{ P_{1,1} \} \}$.
    \end{enumerate}
\end{lemma}

\begin{proof}
    We proceed by contradiction.
    
    \step{First statement}
    Assume that $f_{X_1}(0) = \sum_{j \geq 1} \alpha_j f_{M_j}(0)$ where $\alpha_j \in \R$ and the sum is finite.
    Hence $f_{B_1}(0) = 0$ where $B_1 := X_1 - \sum_{j \geq 1} \alpha_j M_j \in S_1(X)$.
    Let $B_2 := \ad^2_{\ad^3_{B_1}(X_0)}(X_0)$.
    Then $f_{B_2}(0) = 0$.
    Moreover, by definition of $B_1$ and $B_2$, one checks that $B_2 = D + B_3$ where $B_3 \in \vect \{ b \in \Bs_6 ; n_0(b) \geq 4 \}$.
    The equality $f_D(0) = - f_{B_3}(0)$ contradicts the assumption on $f_D(0)$.
    
    \step{Second statement}
    Assume that there exists $B_0 \in \vect \{ b \in \Bs_{\intset{1,3}} ; n_1(b) < 3 \text{ or } n_0(b) > 1 \}$ such that $f_{P_{1,1}}(0) = f_{B_0}(0)$.
    Let $B_1 := P_{1,1} - B_0$ so that $f_{B_1}(0) = 0$.
    Then $f_{B_2}(0) = 0$ where $B_2 := \ad_{B_1}^2(X_0)$. 
    Thus $f_D(0) = f_{B_3}(0)$ where $B_3 \in \vect \{ b \in \Bs_{\intset{1,6}} ; n_1(b) \leq 5 \text{ or } n_0(b) \geq 4 \}$, which contradicts the assumption on $f_D(0)$.
\end{proof}

\subsection{Closed-loop estimates}

\begin{lemma}
    Assume that $f_D(0) \notin \mathcal{N}_D(f)(0)$.
    Then
    \begin{align}
        \label{eq:D-CL-u1}
        |u_1(t)| & = O\left( |x(t;u)| + \|u_1\|_{L^2}^2 \right), \\
        \label{eq:D-CL-u1^3}
        \left| \xi_{P_{1,1}}(t,u) \right| & = O\left(|x(t;u)| + \|u_1\|_{L^4}^4 \right).
    \end{align}
\end{lemma}

\begin{proof}
    We rely on \cref{p:D-vectors}.
    
    \step{First estimate}
    By \cref{thm:Key_1} with $M \gets 1$,
    \begin{equation} \label{eq:D-CL-u1-1}
        x(t;u) = \mathcal{Z}_1(t,f,u)(0) + O \left( \|u_1\|_{L^2}^2 + |x(t;u)|^{1+1} \right).
    \end{equation}
    By \cref{p:D-vectors}, we can consider $\mathbb{P}$, a component along $f_1(0)$, parallel to $\mathcal{N}(f)(0)$ where $\mathcal{N} := \Bs_1 \setminus \{ X_1 \}$.
    Hence $\mathbb{P} \mathcal{Z}_1(t,f,u)(0) = u_1(t)$.
    Thus \eqref{eq:D-CL-u1-1} yields \eqref{eq:D-CL-u1}.
    
    \step{Second estimate}
    By \cref{thm:Key_1} with $M \gets 3$,
    \begin{equation} \label{eq:D-CL-u1^3-1}
        x(t;u) = \cZ{3}(t,f,u)(0) + O \left( \|u_1\|_{L^4}^4 + |x(t;u)|^{1+\frac 1 3} \right).
    \end{equation}
    By \cref{p:D-vectors}, we can consider $\mathbb{P}$, a component along $f_{P_{1,1}}(0)$, parallel to $\mathcal{N}(f)(0)$ where $\mathcal{N} := \Bs_{\intset{1,3}} \setminus \{ P_{1,1} \}$.
    By \eqref{eq:ZM-etab},
    \begin{equation} \label{eq:D-CL-u1^3-2}
        \mathbb{P} \cZ{3}(t,f,u)(0) = \eta_{P_{1,1}}(t,u).
    \end{equation}
    We apply \cref{p:etab-xib-XI} (see below) to obtain
    \begin{equation} \label{eq:D-CL-u1^3-3}
        \eta_{P_{1,1}}(t,u) = \xi_{P_{1,1}}(t,u) + O\left( |u_1(t)| \|u_1\|_{L^2}^2 + |u_1(t)|^2 \|u_1\|_{L^1} \right).
    \end{equation}
    Then \eqref{eq:D-CL-u1^3-1}, \eqref{eq:D-CL-u1^3-2} and \eqref{eq:D-CL-u1^3-3}, combined with the previous estimate \eqref{eq:D-CL-u1}, yield \eqref{eq:D-CL-u1^3}.
    
    Let us check the required conditions to obtain \eqref{eq:D-CL-u1^3-3}.
    Let $q \geq 2$, $b_1 \geq \dotsb \geq b_q \in \Bs$ such that $P_{1,1} \in \supp \mathcal{F}(b_1, \dotsc, b_q)$.
    Since $n_1(P_{1,1}) = 3$ and $n_0(P_{1,1}) = 1$, the only possibilities are 
    \begin{itemize}
        \item $q = 2$, $b_1 = W_1$, $b_2 = X_1$, in which case
        \begin{equation}
            |\xi_{b_1}(t,u)\xi_{b_2}(t,u)| = |u_1(t)| \int_0^t \frac{u_1^2}{2} \leq |u_1(t)| \|u_1\|_{L^2}^2.    
        \end{equation}
        \item $q = 3$, $b_1 = M_1$, $b_2 = b_3 = X_1$, in which case
        \begin{equation}
            |\xi_{b_1}(t,u)\xi_{b_2}(t,u)\xi_{b_3}(t,u)| = |u_1(t)|^2 |u_2(t)| \leq |u_1(t)|^2 \|u_1\|_{L^1}.
        \end{equation}
    \end{itemize}
    This concludes the proof of \eqref{eq:D-CL-u1^3-3} by \cref{p:etab-xib-XI}.
\end{proof}

\subsection{Interpolation inequalities}
\label{s:sextic-interpolation}

\begin{lemma}
    There exits $C > 0$ such that, for every $t > 0$ and $u \in \lone$,
    \begin{align}
        \label{eq:D-interpol-u1-L8}
        \|u_1\|_{L^8}^8 & \leq C t |u_1(t)|^8 + C \| u \|_{L^\infty}^2 \xi_D(t,u), \\
        \label{eq:D-interpol-P111}
        |\xi_{P_{1,1}0}(t,u)|^2 & \leq 2 t \xi_D(t,u), \\
        \label{eq:D-interpol-R}
        |\xi_{R^\sharp_{1,1,1,1}}(t,u)| & \leq C t \|u_1\|_{L^2}^2 |\xi_{P_{1,1}}(t,u)| + C t^{\frac 12} \|u_1\|_{L^2}^2 \xi_D(t,u)^{\frac 12}.
    \end{align}
\end{lemma}

\begin{proof}
    \step{First estimate}
    By integration by parts,
    \begin{equation}
        \int_0^t u_1^8 = u_1^5(t) \int_0^t u_1^3 - 5 \int_0^t u(s) u_1^4(s) \left(\int_0^s u_1^3\right) \dd s.
    \end{equation}
    By Cauchy--Schwarz and Hölder inequalities and \eqref{eq:def:coord2}, we obtain
    \begin{equation}
        \| u_1 \|_{L^8}^8 \leq t^{\frac 58} |u_1(t)|^5 \|u_1\|_{L^8}^3 + 30 \sqrt{2} \|u\|_{L^\infty} \|u_1\|_{L^8}^4 \xi_D(t,u)^{\frac 12},    
    \end{equation}
    which proves \eqref{eq:D-interpol-u1-L8} using Young's inequality.
    
    \step{Second estimate}
    By \eqref{eq:def:coord2}, $\xi_{P_{1,1} 0} = \int \xi_{P_{1,1}}$ and $\xi_D = \frac 12 \int \xi_{P_{1,1}}^2$ so \eqref{eq:D-interpol-P111} follows directly from the Cauchy--Schwarz inequality.
    
    \step{Third estimate}
    By \eqref{eq:def:coord2} and since $R^\sharp_{1,1,1,1} = (W_1 0, P_{1,1})$, integration by parts yields
    \begin{equation}
        \xi_{R^\sharp_{1,1,1,1}}(t,u) 
        = \int_0^t \xi_{W_1 0} \dot{\xi}_{P_{1,1}}
        = \xi_{W_1 0}(t) \xi_{P_{1,1}}(t) - \int_0^t \xi_{W_1} \xi_{P_{1,1}}.
    \end{equation}
    Then \eqref{eq:D-interpol-R} follows by the Cauchy--Schwarz inequality and the estimates $\xi_{W_1}(0)(t) \leq t \|u_1\|_{L^2}^2$ and $\xi_{W_1}(t) \leq \|u_1\|_{L^2}^2$.
\end{proof}

\begin{remark}
    Estimate \eqref{eq:D-interpol-u1-L8} is not exactly of the form \eqref{eq:coerc-interp} since it involves a ``boundary term'' $t |u_1(t)|^8$.
    In our context, this boundary term is harmless since it will immediately be absorbed thanks to the closed-loop estimate \eqref{eq:D-CL-u1}.
    Moreover, it is likely that \eqref{eq:D-interpol-u1-L8} also holds without this additional term, up to a slightly more complex proof.
\end{remark}

\subsection{Proof of the presence of the drift}

\begin{proof}[Proof of \cref{thm:sextic-drift}]
    Let $\mathbb{P}$ be a component along $f_D(0)$ parallel to $\mathcal{N}_D(f)(0)$.
    By \cref{thm:Key_1} with $M \gets 7$,
    \begin{equation}
        x(t;u) = \cZ{7}(t,f,u)(0) + O\left( \|u_1\|_{L^8}^8 + |x(t;u)|^{1+\frac{1}{7}} \right),
    \end{equation}
    where $\mathbb{P} \cZ{7}(t,f,u)(0)$ satisfies \eqref{eq:D-Z7}.
    Combining the closed-loop estimate \eqref{eq:D-CL-u1} and the interpolation estimate \eqref{eq:D-interpol-u1-L8}, one obtains
    \begin{equation}
        \|u_1\|_{L^8}^8 = O\left(|x(t;u)|^8+\|u\|_{L^\infty}^2 \xi_D(t,u) \right).
    \end{equation}
    Substituting in the closed-loop estimate \eqref{eq:D-CL-u1} yields
    \begin{equation}
        |u_1(t)|^4 = O\left(|x(t;u)|^4+\|u\|_{L^\infty}^2 \xi_D(t,u) \right)
    \end{equation}
    and in the closed-loop estimate \eqref{eq:D-CL-u1^3} yields
    \begin{equation}
        |\xi_{P_{1,1}}(t,u)|^2 = O\left(|x(t;u)|^2+\|u\|_{L^\infty}^2 \xi_D(t,u) \right).
    \end{equation}
    Eventually, using \eqref{eq:D-interpol-R} and Young's inequality,
    \begin{equation}
        \begin{split}
            |u_1(t) \xi_{R^\sharp_{1,1,1,1}}(t,u)|
            & = O \left( |\xi_{P_{1,1}}(t,u)|^2 + \|u_1\|_{L^2}^8 + |u_1(t)|^4 + t \xi_D(t,u) \right) \\
            & = O \left( |x(t;u)|^2 + (t+\|u\|_{L^\infty}^2) \xi_D(t,u) \right).
        \end{split}
    \end{equation}
    Gathering all these equalities in \eqref{eq:D-Z7} and the interpolation estimate \eqref{eq:D-interpol-P111} yields
    \begin{equation}
        \mathbb{P} x(t;u) = \xi_D(t,u)
        + O \left( \left(t+\|u\|_{L^\infty}^2\right) \xi_D(t,u) +|x(t;u)|^{1+\frac{1}{7}} \right).
    \end{equation}
    This implies a drift along $f_{D}(0)$, parallel to $\mathcal{N}_D(f)(0)$, as $(t,\|u\|_{L^\infty}) \to 0$, in the sense of \cref{Def:drift}.
\end{proof}

\section{Embedded semi-nilpotent systems and the \texorpdfstring{$m=-1$}{m=-1} case}
\label{s:m=-1}

As announced in \cref{s:method-1}, for $m=-1$ and particular choices of $\bb$, estimate \eqref{eq:coerc-interp} may fail (even for arbitrarily large $M$) and then the remainder $\|u\|_{W^{-1,M+1}}^{M+1}$ in the representation formula~\eqref{repform1} cannot be absorbed by interpolation. 
This is for example the case for \cref{Thm:Kawski_Wm} with $m=-1$ and $k \geq 2$.
In this section, we describe extensions of our unified approach which can be used to deal with such cases, and in particular we use them to prove \cref{Thm:Kawski_Wm} and \cref{thm:S2-S3-intro} for $m = -1$ and $k \geq 2$.

First, we present a general methodology relying on the notion of \vocab{weak drifts} (\cref{s:weak-drift}), which are a slightly weaker version of \cref{Def:drift}.
We show in \cref{s:weak-semi-quad} that our general strategy entails weak drifts for semi-nilpotent vector fields (see \cref{s:semi-nilpotent}), because they satisfy an approximate representation formula without the remainder term $\|u\|_{W^{-1,M+1}}^{M+1}$, see \eqref{eq:x=ZM+O+Nilp}.
Then, using an argument of embedded semi-nilpotent system
(see \cref{subsec:embedded_seminilp}), we extend the weak drift conclusion to systems that do not satisfy the semi-nilpotency assumption (see \cref{subsc:Drift=seminilp->all}).

Eventually, in \cref{s:weak-to-strong}, we prove that one can also work even more precisely and prove drifts in the strong sense of \cref{Def:drift} even when $m=-1$.

\subsection{Weak drifts}
\label{s:weak-drift}

We start with the following weaker version of \cref{Def:drift}.

\begin{definition}[Weak drift] \label{Def:Weak-drift}
    Let $\bb \in \Bs$ and $\mathcal{N} \subset \Br(X)$.
    We say that system \eqref{syst} has a \emph{weak drift along $f_{\bb}(0)$, parallel to $\mathcal{N}(f)(0)$, as $(t,\|u\|_{W^{-1,\infty}}) \to 0$} when, for every $\varepsilon>0$, there exists $\rho=\rho(\varepsilon) > 0$ such that, for every $t \in (0,\rho)$ and every $u \in L^1((0,t);\R)$ with $\|u\|_{W^{-1,\infty}}\leq\rho$,
    \begin{equation} \label{eq:def-weak-drift}
        \mathbb{P} x(t;u) \geq (1-\varepsilon) \xi_{\bb}(t,u) - \varepsilon |x(t;u)|
    \end{equation}
    where $\mathbb{P}$ gives a component along $f_{\bb}(0)$ parallel to $\mathcal{N}(f)(0)$ and $(\xi_{b})_{b\in\Bs}$ are the coordinates of the second kind associated with $\Bs$ (see \cref{Def:Coord2} and \cref{Prop:Coord_Bstar}).
\end{definition}

For instance, the conclusion of \cref{thm:W3precise} can be interpreted as a weak drift along $f_{W_3}(0)$, parallel to $\mathcal{N}_3(f)(0)$, as $(t,\|u\|_{L^\infty}) \rightarrow 0$.

A drift, in the strong sense of \cref{Def:drift}, implies a weak drift in the sense of \cref{Def:Weak-drift}. 
The reciprocal may not be true. 
Nevertheless, a weak drift is sufficient to prevent STLC, and one can prove the following lemma as \cref{lem:drift-stlc}. 

\begin{lemma} \label{lem:weak-drift-stlc}
    Let $\bb \in \Bs$ and $\mathcal{N} \subset \Br(X)$.
    Assume that $\xi_\bb(t,u) \geq 0$ for all $u \in L^1((0,t);\R)$ and that system \eqref{syst} has a weak drift along $f_{\bb}(0)$, parallel to $\mathcal{N}(f)(0)$, as $(t,\|u\|_{W^{-1,\infty}}) \to 0$.
    Then system \eqref{syst} is not $W^{-1,\infty}$-STLC.
\end{lemma}


Estimate \eqref{eq:def-weak-drift} proves that, in the limit $(t,\|u\|_{W^{-1,\infty}}) \to 0$, the ``ultimately unreachable'' set contains a half space.
But for $t$ and $\|u\|_{W^{-1,\infty}}$ fixed (even small), estimate \eqref{eq:def-weak-drift} only guarantees that the unreachable set contains a convex cone, which is slightly weaker than \cref{Def:drift} as commented in \cref{rk:beta>1}.

\subsection{Semi-nilpotent systems and their representation formula}
\label{s:semi-nilpotent}

As sketched in \cref{s:method-1}, our methodology for $m=-1$ relies on the notion of ``semi-nilpotent'' systems, which enjoy the representation formula \eqref{eq:x=ZM+O+Nilp} below.

\begin{definition}[Semi-nilpotent family of vector fields] 
    \label{Def:semi-nilpotent}
    Let $\Omega$ be an open subset of $\R^d$, $f_0, f_1 \in \CC^\infty(\Omega;\R^d)$ and $M \in \N^*$. 
    We say that the vector field $f_1$ is \emph{semi-nilpotent (resp. semi-nilpotent at zero) of index $M$ with respect to $f_0$} when
    \begin{equation} \label{eq:nilpotent}
        \forall b \in \Br(X), n_1(b) \geq M \Rightarrow f_b = 0 \text{ on } \Omega
    \end{equation}
     \begin{equation} \label{eq:nilpotent-0}
       \left( \text{resp. } \forall b \in \Br(X), n_1(b) \geq M \Rightarrow f_b(0) = 0  \right)
    \end{equation}    
    and $M$ is the smallest positive integer for which this property holds\footnote{Condition \eqref{eq:nilpotent} is equivalent to its variant with only $n_1(b) = M$, as can be checked by writing any bracket with $n_1(b) > M$ as a left-nested one by Jacobi's identity.
This does not hold for condition \eqref{eq:nilpotent-0}.}.
\end{definition}

Clearly, the semi-nilpotency implies the semi-nilpotency at zero. 
The converse is true for analytic vector fields under the Lie algebra rank condition at $0$, as stated in \cref{LemB2} below.

\begin{lemma} \label{LemB2-ideal}
    Let $\Omega \subset \R^d$ be a connected open neighborhood of $0$, $f_0, f_1 \in \CC^\omega(\Omega;\R^d)$.
    Let $\mathcal{I}$ be an ideal of $\mathcal{L}(X)$.
    Assume that $\mathcal{L}(f)(0)=\R^d$ and that, for every $B \in \mathcal{I}$, $f_B(0) = 0$.
    
    Then, for every $B \in \mathcal{I}$, $f_B \equiv 0$ on $\Omega$.
\end{lemma}

\begin{proof}
    \step{We prove that, if 
    $f \in \CC^\omega(\Omega;\R^d)$ and 
    $g_1,\dotsc,g_d \in \CC^\infty(\Omega;\R^d)$ satisfy 
    $\vect \{ g_i(0) ; i \in \intset{1,d} \} = \R^d$ and 
    	\begin{equation}
    		\forall n \in \N,
    		\forall i_1, \dotsc, i_n \in \intset{1,d}^n,
    		\quad 
    		[g_{i_n}, [g_{i_{n-1}}, \dotsc, [g_{i_1}, f]\dotsb]](0) = 0,
    	\end{equation}
    	then $f \equiv 0$ on $\Omega$}
    Since $\Omega$ is connected and $f$ is analytic, it is sufficient to prove that, for every $n \in \N$, its $n$-th differential vanishes at zero: $D^n f(0)=0$. We work by induction on $n \in \N$. The initialization for $n=0$ holds. Let $n \in \N^*$ and assume $D^k f(0)=0$ for $k=0,\dots,n-1$. Then, for every $i_1, \dotsc, i_n \in \intset{1,d}^n$, by expanding the iterated Lie brackets
    \begin{equation}
        D^n f(0) \cdot (g_{i_n}(0),\dotsc,g_{i_1}(0)) 
        = [g_{i_n}, [g_{i_{n-1}}, \dotsc, [g_{i_1}, f]\dotsb]](0) 
        = 0.
    \end{equation}
    This implies $D^n f (0)=0$ because $\vect \{ g_i(0) ; i \in \intset{1,d} \} = \R^d$. 

    \step{Proof of \cref{LemB2}}
    By the Lie algebra rank condition there exist $b_1,\dots,b_d \in \Br(X)$ such that $\vect \{ f_{b_i}(0) ; i \in \intset{1,d} \} = \R^d$. 
    Let $B \in \mathcal{I}$. 
    We prove that $f_B \equiv 0$ on $\Omega$ by applying Step~1. 
    Indeed, for $n \in \N$ and $i_1, \dotsc, i_n \in \intset{1,d}^n$, we have $[f_{b_{i_n}}, \dotsc , [f_{b_{i_1}}, f_B]\dotsb](0) = f_{\underline{B}}(0)=0$, where $\underline{B}:=[b_{i_n},\dotsc [b_{i_1},B] \dotsb ] \in \mathcal{I}$ since $\mathcal{I}$ is an ideal of $\mathcal{L}(X)$.
\end{proof}

\begin{corollary}
    \label{LemB2}
    Let $\Omega \subset \R^d$ be a connected open neighborhood of $0$, $f_0, f_1 \in \CC^\omega(\Omega;\R^d)$ and $M \in \N^*$.
    If $\mathcal{L}(f)(0)=\R^d$ and $f_1$ is semi-nilpotent at zero of index $M$ with respect to $f_0$, then $f_1$ is semi-nilpotent of index $M$ with respect to $f_0$.
\end{corollary}

\begin{proof}
    This follows from \cref{LemB2-ideal} applied with the ideal $\mathcal{I} = S_{\llbracket M,\infty \llbracket}(X)$.
\end{proof}

\begin{proposition} \label{thm:Key_2}
    Let $f_0$, $f_1$ be analytic vector fields on a neighborhood of $0$ with $f_0(0) = 0$ and $M \in\N^*$.
    Assume that $\mathcal{L}(f)(0)=\R^d$ and
    $f_1$ is semi-nilpotent at zero of index $M+1$ with respect to~$f_0$.
    Then
    \begin{equation} \label{eq:x=ZM+O+Nilp}
        x(t;u) = \cZ{M}(t,f,u)(0) + O \left(\|u_1\|_{L^\infty} |x(t;u)| \right).
    \end{equation}
\end{proposition}

\begin{proof}
    By \cref{LemB2}, $f_1$ is semi-nilpotent of index $M+1$ with respect to $f_1$ on a neighborhood of $0$. Then, the third item of \cite[Corollary 163]{P1} gives the estimate.
\end{proof}

\subsection{Weak quadratic drift for semi-nilpotent systems}
\label{s:weak-semi-quad}

We now prove that, thanks to the modified representation formula \eqref{eq:x=ZM+O+Nilp}, one can prove \cref{Thm:Kawski_Wm,thm:S2-S3-intro} for $m = -1$ and $k \geq 2$ for such semi-nilpotent systems.

\begin{theorem} \label{Thm:Kawski_Wm_precis--1_nilp}
    Let $k \geq 2$. 
    Assume\footnote{This assumption is not restrictive as one can always work within the integral manifold generated by $f_0$ and $f_1$.\label{fn:LARC}} that $\mathcal{L}(f)(0)=\R^d$, $f_1$ is semi-nilpotent at zero with respect to $f_0$ and $k$ is the minimal value for which $f_{W_k}(0) \notin S_{\N^* \setminus\{2\}}(f)(0)$. 
    Then system \eqref{syst} has a weak drift along $f_{W_k}(0)$, parallel to $S_{\N^* \setminus\{2\}}(f)(0)$, as $(t,\|u\|_{W^{-1,\infty}}) \rightarrow 0$.
\end{theorem}

The proof is the same as the one of \cref{Thm:Kawski_Wm_precis-pos}, merely replacing \eqref{eq:x=ZM+O} by \eqref{eq:x=ZM+O+Nilp} everywhere.
In particular, this yields the following closed-loop estimate.



\begin{lemma} \label{p:quad-loop-1}
    Under the assumptions of \cref{Thm:Kawski_Wm_precis--1_nilp},
    \begin{equation} \label{eq:quad-loop-1}
        |(u_1,\dotsc,u_k)(t)| = O \left( |x(t;u)| + t^{\frac 12} \|u_k\|_{L^2} \right).
    \end{equation}
\end{lemma}

\begin{proof}
    The proof is performed along the same lines as in \cref{p:quad-loop}.
    Instead of $M = \vartheta(k)$, one uses $M$ such that $f_1$ is semi-nilpotent of index $(M+1)$ with respect to $f_0$.
    One replaces \eqref{eq:quad-CL-1} by \eqref{eq:x=ZM+O+Nilp} and concludes as previously. Note that the vectorial relations of \cref{p:quad-vectors} hold with $\pi(k)=\mathcal{V}(k)=\infty$.
\end{proof}

\begin{proof}[Proof of \cref{Thm:Kawski_Wm_precis--1_nilp}]
    Let $\mathbb{P}$ be a component along $f_{W_k}(0)$ parallel to 
    $S_{\N^* \setminus\{2\}}(f)(0)$.
    Let $M\in\N^*$ be such that $f_1$ is semi-nilpotent of index $(M+1)$ with respect to $f_0$ (see \cref{Def:semi-nilpotent}).
    By \cref{thm:Key_2},
    \begin{equation} \label{x=Z+O_nilpotent}
        x(t;u) = \cZ{M}(t,f,u)(0) + O\left( \|u_1\|_{L^\infty} |x(t;u)| \right),
    \end{equation}
    where, by \eqref{eq:quad-zpi} and \eqref{xi_Wjnu},
    \begin{equation} \label{PZM_m=-1}
        \mathbb{P} \cZ{M}(t,f,u)(0)
        = \frac{1}{2} \int_0^t u_k^2 + O\left( |(u_1,\dotsc,u_k)(t)|^2 + t \|u_k\|_{L^2}^2 \right).
    \end{equation}
    Gathering \eqref{x=Z+O_nilpotent} and \eqref{eq:quad-loop-1} yields
    \begin{equation} \label{eq:Px-Wk-W-1}
        \mathbb{P} x(t;u) = \frac 12 \int_0^t u_k^2
        + O \left( t \|u_k\|_{L^2}^2 + \|u_1\|_{L^\infty} |x(t;u)| \right),
    \end{equation}
    proving the presence of the weak drift.
\end{proof}

\begin{theorem} \label{Thm:Kawski_Wm_precis--1_nilp_crible}
    Let $k \geq 2$. 
    Assume\cref{fn:LARC} that $\mathcal{L}(f)(0)=\R^d$, $f_1$ is semi-nilpotent at zero with respect to $f_0$ and $k$ is the minimal value for which 
    $f_{W_k}(0) \notin (\Bs_1 \cup \mathcal{P}_k \cup \Bs_{\llbracket 4 , \infty \llbracket} )(f)(0)$. 
    Then system \eqref{syst} has a weak drift along $f_{W_k}(0)$, parallel to $(\Bs_1 \cup \mathcal{P}_k \cup \Bs_{\llbracket 4 , \infty \llbracket} )(f)(0)$, as $(t,\|u\|_{W^{-1,\infty}}) \rightarrow 0$.
\end{theorem}

\begin{proof}
    The proof follows the same steps as the proof of \cref{Thm:Kawski_Wm_precis-pos_crible}, the truncated formula \eqref{eq:x=ZM+O} being replaced by \eqref{eq:x=ZM+O+Nilp}.
    Let $\mathbb{P}$ be a component along $f_{W_k}(0)$ parallel to $(\Bs_1 \cup \mathcal{P}_k \cup \Bs_{\llbracket 4 , \infty \llbracket} )(f)(0)$.
    Let $M\in\N^*$ be such that $f_1$ is semi-nilpotent of index $(M+1)$ with respect to $f_0$ (see \cref{Def:semi-nilpotent}).
    The dominant part of the logarithm satisfies \eqref{PZM_m=-1}.
    The proof of the vectorial relations in \cref{s:S2-S3} holds with $\pi(k)=\mathcal{V}(k)=\infty$, which proves the closed-loop estimate \eqref{eq:quad-loop-1}. 
    Thus, the proof ends as above.
\end{proof}

\subsection{Embedded semi-nilpotent systems}
\label{subsec:embedded_seminilp}

In this section, we explain how to extract from a possibly rich large system, smaller parts of which the controllability may be easier to analyze.
Our motivation is to extract from a large system a semi-nilpotent one to which we can apply the results of the previous subsection.
We apply this idea in the next subsection.

\begin{definition} \label{def:embedded}
    Let $f_0, f_1$ be analytic vector fields in a neighborhood of $0 \in \R^d$.
    Let $r \in \intset{0,d}$ and $g_0, g_1$ be analytic vector fields in a neighborhood of $0 \in \R^r$.
    We say that the smaller system $\dot{y} = g_0(y) + u g_1(y)$ is \emph{embedded} in the larger system $\dot{x} = f_0(x) + u f_1(x)$ when there exist an open neighborhood $\Omega_x$ (resp. $\Omega_y$) of $0$ in $\R^d$ (resp. $\R^r$) and an analytic map $\lambda:\Omega_x \rightarrow \Omega_y$ with $\lambda(0)=0$ such that
    \begin{equation} \label{lambda-related}
        \forall j \in \{0,1\}, \forall x \in \Omega_x, \qquad D\lambda(x) f_j(x)=g_j(\lambda(x)).
    \end{equation}
    In this case, their evaluated Lie brackets satisfy
    \begin{equation} \label{gb=dlfb}
        \forall b \in \Br(X), \qquad g_b(0)=D\lambda(0) f_b(0)
    \end{equation}
    and for every $T>0$ and $u \in L^1(0,T)$ such that $x([0,T];u) \subset \Omega_x$ then $y(t;u)=\lambda(x(t;u))$ on $[0,T]$.
\end{definition}

Equality \eqref{lambda-related} corresponds to the notion of $\lambda$-related fields (see \cite[Page 182]{lee2012smooth}) and implies that, for every $b\in \Br(X)$, $g_b$ and $f_b$ are also $\lambda$-related (see \cite[Proposition 8.30]{lee2012smooth}), which entails~\eqref{gb=dlfb}. 
The equality $y = \lambda(x)$ along trajectories follows from the chain rule and \eqref{lambda-related}.

Moreover, by Krener's result \cite[Theorem 1]{zbMATH03385496}, given $g_0, g_1$, the existence of $(\Omega_x,\Omega_y,\lambda)$ is equivalent to the existence of a linear map $L : \R^d \to \R^r$ such that, for all $b \in \Br(X)$, $g_b(0) = L f_b(0)$ (in which case, one has $L = D\lambda(0)$).

We now proceed in the converse direction, and derive a sufficient condition on such a linear map $L$ to guarantee the existence of $g_0, g_1$ satisfying $g_b(0) = L f_b(0)$ for all $b \in \Br(X)$.

\begin{proposition} \label{PropB5}
    Let $f_0, f_1$ be analytic vector fields in a neighborhood of $0 \in \R^d$.
    Let $\mathcal{I}$ be an ideal of $\mathcal{L}(X)$ and $r := \codim ( \mathcal{I}(f)(0) )$.
    Then there exists an embedded system $(g_0,g_1,\lambda)$ set on~$\R^r$ such that $\ker D\lambda(0) = \mathcal{I}(f)(0)$.
\end{proposition}

\begin{remark} \label{embedded_nilpotent_syst}
    If $V_g$ and $V_f$ denote the following linear maps
    \begin{equation}
        V_g:b \in \mathcal{L}(X) \mapsto g_b(0) \in \R^r,
        \qquad
        V_f:b\in\mathcal{L}(X) \mapsto f_b(0) \in \R^d
    \end{equation}
    then $\ker(V_g)=\ker(V_f)+\mathcal{I}$.
    In particular, if $\mathcal{I}=S_{\llbracket n , \infty \llbracket}(X)$ for some $n \in \N^*$ then $g_1$ is semi-nilpotent at zero with respect to $g_0$.
\end{remark}

\begin{remark}
    \cref{PropB5} only provides a sufficient condition on $\ker D\lambda(0)$ for the existence of an embedded system, which is however not necessary as illustrated by the following example.
    
    Consider on $\R^3$ the system $\dot{x} = (u, x_1, x_1^2)$ and on $\R^2$ the system $\dot{y} = (u, y_1^2)$.
    One has $y(t,u) = \lambda(x(t;u))$ with $\lambda(x) = (x_1, x_3)$.
    Hence $\ker D\lambda(0) = \R e_2$.
    By contradiction, consider $\mathcal{I}$  an ideal of $\mathcal{L}(X)$ such that $\ker D\lambda(0) = \mathcal{I}(f)(0)$.
    Take $b \in \mathcal{I}$ such that $f_b(0) = e_2$.
    Expanding $b$ on $\Bs$, $b = \alpha_0 X_0 + \alpha_1 X_1 + \beta [X_1, X_0] + \gamma [X_1, [X_1, X_0]] + \delta [[X_1, X_0],X_0] + B$ where $B$ is a sum of brackets of length at least 4.
    Since $f_b(0) = e_2$, $\alpha_1 = \gamma = 0$ and $\beta = 1$.
    Since $\mathcal{I}$ is an ideal, $[X_1, b] \in \mathcal{I}$ so $f_{[X_1,b]}(0)  \in \ker D\lambda(0)$.
    But $f_{[X_1,b]}(0) = \alpha_0 e_2 + e_3$ so $D \lambda (0) f_{[X_1,b]}(0) = e_3 \neq 0$.

    It could be interesting to derive a necessary and sufficient condition for the existence of an embedded system.
    Such a result might be linked with the theory of realization of control systems (see e.g.\ \cite{fliess1983realisation,Jakubczyk2000,reutenauer1986local}).
\end{remark}

Our proof of \cref{PropB5} is inspired by the one of \cite[Theorem 1]{zbMATH03385496}. 
It relies on the following classical expansion, proved for instance in \cite[Lemma 90, item 3]{P1}.

\begin{lemma} \label{Lem90.3}
Let $\delta>0$,
$f_0, f_1 \in C^{\omega}(B_{2\delta},\R^d)$ and
$\Phi_0(t,p) := e^{tf_0}(p)$ the flow associated with $f_0$.
If $t$ is small enough then, for each $p \in B_\delta$,
    \begin{equation}\label{tool_series}
        {\left(\partial_p \Phi_0 (t,p)\right)}^{-1} f_1\left( \Phi_0(t,p) \right) 
        = \sum_{k=0}^{+\infty} \frac{t^k}{k!} \ad_{f_0}^k(f_1)(p).
    \end{equation}
\end{lemma}

\begin{proof}[Proof of \cref{PropB5}]
    By the assumptions on $f_0, f_1$ and $\mathcal{I}$ there exist $b_1, \dotsc b_d \in \mathcal{L}(X)$ such that the vectors $f_{b_i}(0)$ form a basis of $\R^d$ for $1 \leq i \leq d$ and a basis of $\mathcal{I}(f)(0)$ with $b_i \in \mathcal{I}$ for $r < i \leq d$. 
    (One can assume that $1 \leq r < d$, since $\lambda = 0$ works when $r = 0$ and $\lambda = \operatorname{Id}$ works when $r = d$.)

    For $s=(s_1,\dots,s_d) \in \R^d$ small enough, let 
	\begin{equation}
		F(s) := e^{s_d f_{b_d}} \dotsb e^{s_1 f_{b_1}}(0).
	\end{equation}
    Since $f_0, f_1$ are analytic near $0$, $F$ is an analytic map on a neighborhood of $0$. Moreover, since the $f_{b_i}(0)$ are a basis of $\R^d$, $F$ is a local diffeomorphism of $\R^d$ around $0$. 
    Let $P : s=(s_1,\dots,s_d) \in \R^d \mapsto (s_1,\dots,s_r,0,\dots,0)\in\R^d$.
    Then $\lambda(x) := P F^{-1}(x)$ defines an analytic map $\lambda$ on a neighborhood of $0$ in $\R^d$, taking values in $\R^r \times \{0\} \subset \R^d$, such that $\lambda(0)=0$.

    In the rest of this proof, we will implicitly consider $s, x\in\R^d$ small enough for the formulas/statements to hold.
    We write $s' := (s_1, \dotsc, s_r)$.
    
    \step{We introduce vectors $h_i(s) \in \R^d$ and a linear map $Q(s') :\R^d \rightarrow \R^d$ such that:}
    \begin{equation} \label{Q(s')hi_utile}
    \forall 1 \leq i \leq d, \quad Q(s') h_i(s)=e_i \delta_{i \leq r}
    \qquad 
    \text{ and }
    \qquad
    \ker(Q(s'))=\mathcal{I}(f)(0).
    \end{equation}
    For $1 \leq i \leq d$, we introduce 
    the flow $\Phi_i$ associated with $f_{b_i}$, i.e. $\Phi_i(t,p)=e^{t f_{b_i}}(p)$,
    the linear map $L_i(s):\R^d \rightarrow \R^d$  and the vector $h_i(s) \in \R^d$
    defined by
    \begin{equation}
    L_i(s) := \partial_p \Phi_i (s_i, e^{s_{i-1} f_{b_{i-1}}} \dots e^{s_1 f_{b_1}} (0) )\,,
    \end{equation}
    \begin{equation}
    		h_i(s) := L_1(s)^{-1} \dotsb L_i(s)^{-1} f_{b_i} \left( e^{s_i f_{b_i}} \dotsb e^{s_1 f_{b_1}}(0) \right).
    	\end{equation}
    These vectors are analytic functions of $s$ such that 
    $h_i(0) = f_{b_i}(0)$. 
    For $1 \leq i \leq r$, then only depend on $s'=(s_1,\dots,s_r)$, thus may be denoted $h_i(s')$. 
    
    The family $(h_i(s'))_{1\leq i \leq r}$ is linearly independent and 
    in direct sum with the vector space $\mathcal{I}(f)(0)=\vect \{ h_i(0) ; r < i \leq d \}$ thus one may consider the linear map
    $Q(s'):\R^d \rightarrow \R^d$ such that 
    \begin{equation} \label{Def:Q(s')}
    Q(s') h_i(s')=e_i \text{  for  } 1 \leq i \leq r \quad \text{and} \quad \ker(Q(s'))=\mathcal{I}(f)(0).
    \end{equation}
    To end Step 1, it remains to prove that, for every $r < i \leq d$, $h_i(s) \in \ker(Q(s'))$. By \cref{Lem90.3}, 
    \begin{equation}
        h_i(s)= \sum_{k_1, \dotsc, k_i = 0}^{+\infty} \frac{s_1^{k_1} \dotsb s_i^{k_i}}{k_1! \dotsb k_i!} \ad^{k_1}_{f_{b_1}} \dotsb \ad^{k_i}_{f_{b_i}} (f_{b_i}) (0).
    \end{equation}
    For $r<i\leq d$, $b_i \in \mathcal{I}$ and $\mathcal{I}$ is an ideal of $\mathcal{L}(X)$ thus, for every $k \in \N^i$,
    $b:=\ad^{k_1}_{b_1} \dotsb \ad^{k_i}_{b_i} (b_i) \in \mathcal{I}$ and $f_{b}(0) \in \mathcal{I}(f)(0)=\ker(Q(s'))$.
    
    \step{We prove that
    $D\lambda(x)=Q(s') L_1(s)^{-1} \dots L_d(s)^{-1}$.}
    Since both sides are linear maps on $\R^d$, it suffices to check that they coincide on a basis of $\R^d$.
    Since the vectors $\frac{\partial F}{\partial s_i}(0)=f_{b_i}(0)$ for $1 \leq i \leq d$ form a basis of $\R^d$, then so do the vectors $\frac{\partial F}{\partial s_i}(s)$ for $s$ small enough.
    Using successively 
    the definitions of $F$ and $h_i$, 
    Step 1, 
    the definition of $P$ 
    and the chain rule in $\lambda$, one obtains
    \begin{equation}
        Q(s') L_1(s)^{-1} \dots L_d(s)^{-1} \frac{\partial F}{\partial s_i}(s)  = Q(s') h_i(s) = e_i \delta_{i \leq r} = \frac{\partial P}{\partial s_i}(s) = D \lambda (F(s)) \frac{\partial F}{\partial s_i}(s).
    \end{equation}
    This ends Step 2, which, together with Step 1 proves $D\lambda(0)=Q(0)$ and $\ker(D\lambda(0))=\mathcal{I}(f)(0)$.

    \step{We prove that, for $j \in \{0,1\}$, $D\lambda(x)f_j(x)$ depends only on $\lambda(x)$ or equivalently on $s'=(s_1,\dots,s_r)$}
    Using Step 2 and \cref{Lem90.3}
    \begin{equation}
    \begin{split}
    D\lambda(x) f_j(x) 
    & 
    = Q(s') L_1(s)^{-1} \dots L_d(s)^{-1}f_j(e^{s_d} f_{b_d} \dots e^{s_1 f_{b_1}}(0))
    \\ & 
    = Q(s') \sum_{k_1, \dotsc, k_d = 0}^{+\infty} \frac{s_1^{k_1} \dotsb s_d^{k_d}}{k_1! \dotsb k_d!} \ad^{k_1}_{f_{b_1}} \dotsb \ad^{k_d}_{f_{b_d}} (f_j) (0)
    \\ & 
    = \sum_{k_1, \dotsc, k_r = 0}^{+\infty} \frac{s_1^{k_1} \dotsb s_r^{k_r}}{k_1! \dotsb k_r!} 
    Q(s') f_{\ad^{k_1}_{b_1} \dotsb \ad^{k_r}_{b_r} (X_j)} (0) 
    =:g_j(\lambda(x)).
    \end{split}
    \end{equation}
    Indeed any term involving $k_i>0$ for $r<i\leq d$ vanishes because $b_i \in \mathcal{I}$ thus 
    $\ad^{k_1}_{b_1} \dotsb \ad^{k_d}_{b_d} (X_j) \in \mathcal{I}$ by the ideal property, and $f_{\ad^{k_1}_{b_1} \dotsb \ad^{k_d}_{b_d} (X_j)}(0) \in \mathcal{I}(f)(0)=\ker (Q(s'))$.
\end{proof}

\subsection{Drift: from semi-nilpotent systems to all systems}
\label{subsc:Drift=seminilp->all}

We now explain how embedded semi-nilpotent systems can be used to prove drift results.

\begin{lemma} \label{lem:embedded-weak-drift}
    Let $\bb \in \Bs$ and $\mathcal{N} \subset \Br(X)$.
    With the notations of \cref{def:embedded}, if $\dot{y} = g_0 + u g_1$ has a weak drift along $g_{\bb}(0)$, parallel to $\mathcal{N}(g)(0)$, as $(t,\|u\|_{W^{-1,\infty}}) \rightarrow 0$, then $\dot{x} = f_0 + u f_1$ has a weak drift along $f_{\bb}(0)$ parallel to $\mathcal{N}(f)(0)$ as $(t,\|u\|_{W^{-1,\infty}}) \rightarrow 0$.
\end{lemma}

\begin{proof}
    By \cref{p:small-state}, there exist $r,\rho_1>0$ such that, for every $t \in (0,\rho_1)$ and $u \in L^1(0,t)$ with $\|u\|_{W^{-1,\infty}}\leq\rho_1$ then $x(t;u) \in \overline{B}_{\R^d}(0,r) \subset \Omega_x$. 
    The map $\lambda$ is $C^1$ thus locally Lipschitz: there exists $C \geq 1$ such that, for every $x \in \overline{B}_{\R^d}(0,r)$, $|\lambda(x)| \leq C |x|$.

    Let $\varepsilon>0$. The weak drift assumption on $\dot{y} = g_0 + u g_1$ gives $\rho_2 \in (0,\rho_1)$ such that, for every $t \in (0,\rho_2)$ and $u \in L^1(0,t)$ with $\|u\|_{W^{-1,\infty}}\leq\rho_2$ then
     \begin{equation}
            \mathbb{P} y(t;u) \geq \left( 1-\frac{\varepsilon}{2C} \right) \xi_{\bb}(t,u) - \frac{\varepsilon}{2C} |y(t;u)|.
    \end{equation}
    Since $\lambda$ is of class $C^2$,  by   \cref{p:small-state} 
    \begin{equation}
            y(t;u)
            = \lambda(x(t;u))
             = D\lambda(0) x(t;u) + O\left( |x(t;u)|^2 \right).
    \end{equation}
    Thus there exists $\rho_3 \in (0,\rho_2)$ such that, for every $t \in (0,\rho_3)$ and $u \in L^1(0,t)$ with $\|u\|_{W^{-1,\infty}}\leq\rho_3$ then
    \begin{align*}
    \mathbb{P} D\lambda(0) x(t;u)  
    &  \geq 
    \mathbb{P}  y(t;u) - \frac{\varepsilon}{2} |x(t;u)|
      \geq 
    \left( 1-\frac{\varepsilon}{2C} \right) \xi_{\bb}(t,u) - \frac{\varepsilon}{2C} |\lambda(x(t;u))|- \frac{\varepsilon}{2} |x(t;u)|
    \\ & \geq 
    ( 1- \varepsilon ) \xi_{\bb}(t,u) - \varepsilon |x(t;u)|.
    \end{align*}
    Finally, $\mathbb{P} D\lambda(0)$ is a projection on $f_{\bb}(0)$ parallel to $\mathcal{N}(f)(0)$ because $\mathbb{P} D\lambda(0) f_{\bb}(0)= \mathbb{P} g_{\bb}(0)=1$ and for all $b \in \mathcal{N}(X)$, $\mathbb{P} D\lambda(0) f_{b}(0)=\mathbb{P} g_b(0)=0$.
\end{proof}

\begin{definition} \label{def:X1-truncable}
    Let $\mathcal{H}$ be a boolean property on subsets of $\mathcal{L}(X)$.
    We say that $\mathcal{H}$ is \emph{$X_1$-truncable} when there exists $n \in \N^*$ such that, 
    for every vectorial subspace $F$ of $\mathcal{L}(X)$ such that $F$ satisfies~$\mathcal{H}$ then $F+S_{\llbracket n , \infty \llbracket}(X)$ satisfies~$\mathcal{H}$.
\end{definition}

\begin{lemma} \label{lem:X1-truncable}
    For subsets $F$ of $\mathcal{L}(X)$, any finite boolean combination of conditions of the form $(E + S_{\llbracket q, \infty \llbracket}) \cap F \neq \emptyset$ with $E \subset \mathcal{L}(X)$ and $q \in \N^*$ is $X_1$-truncable.
\end{lemma}

\begin{proof}
Let $F$ be a vector subspace of $\mathcal{L}(X)$. 
First, we consider the case of a single such condition. We claim that, for every $n \geq q$,
\begin{equation}
    (E + S_{\llbracket q, \infty \llbracket}) \cap F \neq \emptyset
    \quad \Leftrightarrow \quad 
    (E + S_{\llbracket q, \infty \llbracket}(X)) \cap (F+S_{\llbracket n, \infty \llbracket}(X)) \neq \emptyset.
\end{equation}
Indeed, $\Rightarrow$ holds because $0 \in S_{\llbracket n, \infty \llbracket}(X)$. For $\Leftarrow$, if $e \in E$, $B \in S_{\llbracket q, \infty \llbracket}(X)$, $f \in F$, $B' \in S_{\llbracket n, \infty \llbracket}(X)$ and $e+B=f+B'$ then
$f=e+B-B' \in (E + S_{\llbracket q, \infty \llbracket}(X)) \cap F$ because $n \geq q$.

The same equivalence follows easily for a boolean combination of such elementary conditions, by taking for $n$ the maximum value of the $q$.
\end{proof}

\begin{proposition} \label{Prop:Drift=seminilp->all}
    Let $\mathcal{H}$ be an $X_1$-truncable property.
    We assume that for every $d\in\N^*$, $f_0, f_1$ analytic vector fields on a neighborhood of $0$ in $\R^d$ such that $f_1$ is semi-nilpotent at zero with respect to $f_0$ and $\ker (b \in \mathcal{L}(X) \mapsto f_b(0)\in\R^d )$ satisfies $\mathcal{H}$, system~\eqref{syst} has a weak drift along $f_{\bb}(0)$, parallel to $\mathcal{N}(f)(0)$, as $(t, \|u\|_{W^{-1,\infty}} ) \rightarrow 0$. 
    
    Then, the same conclusion holds without the semi-nilpotency assumption.
\end{proposition}

\begin{proof}
    Let $d\in\N^*$, $f_0, f_1$ analytic vector fields on a neighborhood of $0$ in $\R^d$ such that $\ker(V_f)$ satisfies $\mathcal{H}$, where $V_f : b \in \mathcal{L}(X) \mapsto f_b(0)\in\R^d$.
    Let $n \in \N^*$ be given by \cref{def:X1-truncable}. 
    By \cref{PropB5} applied with $\mathcal{I} = S_{\llbracket n,\infty \llbracket}(X) $, we obtain an embedded system $(g_0, g_1,\lambda)$ for which $g_1$ is semi-nilpotent with respect to $g_0$ (see \cref{embedded_nilpotent_syst}).
    Moreover, the kernel of the linear map $V_g: b \in \mathcal{L}(X) \mapsto g_b(0)$ is $\ker(V_g)=
    \ker(V_f)+S_{\llbracket n,\infty \llbracket}(X)$ thus it satisfies the property $\mathcal{H}$.
    Hence the result follows from \cref{lem:embedded-weak-drift}.
\end{proof}

All the necessary conditions for $W^{-1,\infty}$-STLC proved in this article take the form of conditions of \cref{lem:X1-truncable} with $F=\ker(V_f)$:
\begin{itemize}
    \item $f_{W_k}(0) \notin S_{\N^* \setminus\{2\}}(f)(0)$ corresponds to $E=W_k+S_1(X)$ and $q=3$,
    \item $f_{W_k}(0) \notin (\Bs_1 \cup \mathcal{P}_k \cup \Bs_{\llbracket 4 , \infty \llbracket} )(f)(0)$ corresponds to $E=W_k+S_1(X)+\vect(\mathcal{P}_k)$ and $q=4$.
\end{itemize}

In particular, this entails that the results of \cref{s:weak-semi-quad} prove \cref{Thm:Kawski_Wm,thm:S2-S3-intro} for $m = -1$ and $k \geq 2$ without the semi-nilpotency assumption.



\subsection{Drift or weak drift?}
\label{s:weak-to-strong}

The notions of weak drift, semi-nilpotent vector fields, and $X_1$-truncable properties are very convenient because they provide a systematic way to generalize our unified approach to the case $m=-1$.
In this section, we show that, for \cref{Thm:Kawski_Wm,thm:S2-S3-intro} for $m=-1$ and $k \geq 2$, one can actually cleverly manipulate the representation formula to obtain (strong) drifts.
Of course, as done in \cref{lem:embedded-weak-drift} for weak drifts, one can transfer a drift from an embedded system to the large system.

\begin{lemma} \label{lem:embedded-drift}
    Let $\bb \in \Bs$ and $\mathcal{N} \subset \Br(X)$.
    With the notations of \cref{def:embedded}, if $\dot{y} = g_0 + u g_1$ has a drift along $g_{\bb}(0)$, parallel to $\mathcal{N}(g)(0)$, as $(t,\|u\|_{W^{-1,\infty}}) \rightarrow 0$, then $\dot{x} = f_0 + u f_1$ has a  drift along $f_{\bb}(0)$ parallel to $\mathcal{N}(f)(0)$ as $(t,\|u\|_{W^{-1,\infty}}) \rightarrow 0$.
\end{lemma}

\begin{theorem}
    Let $k \geq 2$. 
    Assume\cref{fn:LARC} that $\mathcal{L}(f)(0) = \R^d$ and that $k$ is the minimal value for which 
    $f_{W_k}(0) \notin (\Bs_1 \cup \mathcal{P}_k \cup \Bs_{\llbracket 4 , \infty \llbracket} )(f)(0)$. 
    Then system \eqref{syst} has a drift along $f_{W_k}(0)$, parallel to $(\Bs_1 \cup \mathcal{P}_k \cup \Bs_{\llbracket 4 , \infty \llbracket} )(f)(0)$, as $(t,\|u\|_{W^{-1,\infty}}) \to 0$.
\end{theorem}

\begin{proof}
    Recalling $\mathcal{W}_k = \{ W_{j,\nu} ; j < k \}$ of \eqref{eq:quad-N_ref} and $\mathcal{P}_k = \{ P_{j,l,\nu} ; j < k \}$ of \eqref{eq:P_k}, let
    \begin{equation}
        \mathcal{I} := \vect \mathcal{W}_k + \vect \mathcal{P}_k + S_{\llbracket 4, \infty \llbracket}(X).
    \end{equation}
    Using \cref{p:1+2}, one checks that $\mathcal{I}$ is an ideal of $\mathcal{L}(X)$.
    Let $(g_0,g_1,\lambda)$ be the embedded system associated with $\mathcal{I}$ given by \cref{PropB5}.
    In particular, since $\mathcal{L}(f)(0) = \R^d$, $\mathcal{L}(g)(0) = \R^r$ where $r := \codim \mathcal{I}(f)(0)$.
    By \cref{LemB2-ideal}, for every $B \in \mathcal{I}$, $g_B \equiv 0$ in a neighborhood of $0 \in \R^r$.
    Since $S_{\llbracket 4, \infty \llbracket}(X) \subset \mathcal{I}$, $g_1$ is semi-nilpotent of index at most 4 with respect to $g_0$ (in the sense of \eqref{eq:nilpotent}).
    By \cite[Corollary 122]{P1}, the solution to $\dot{y} = g_0(y) + u g_1(y)$ satisfies
    \begin{equation} \label{eq:y-Z3}
        y(t;u) = (\exp \mathcal{Z}_{\intset{1,3}}(t,g,u)) (0)
    \end{equation}
    where
    \begin{equation}
        \mathcal{Z}_{\intset{1,3}}(t,g,u)
        = \sum_{i=1}^\infty u_i(t) g_{M_i} + \sum_{k \leq j, \nu \in \N} \eta_{W_{j,\nu}}(t,u) g_{W_{j,\nu}} + \sum_{k \leq j \leq l, \nu \in \N} \eta_{P_{j,l,\nu}}(t,u) g_{P_{j,l,\nu}}.
    \end{equation}
    Using that $k \leq j$, we obtain\footnote{To avoid repeating once more similar arguments, we skip here the verification of the assumptions of these black-box estimates, which can be carried out as in \cref{s:Quad}. Moreover, \cref{p:sum-xi-XI,p:sum-cross-XI} conclude to analytic estimates, which of course entail $\CC^1$ estimates.} from \cref{p:sum-xi-XI,p:sum-cross-XI} that
    \begin{equation} \label{eq:Z3-C1}
        \| \mathcal{Z}_{\intset{1,3}}(t,g,u) \|_{\CC^1} = O\left( |(u_1,\dotsc,u_k(t))| + \| u_k \|_{L^2} \right).
    \end{equation}
    Since $k$ is assumed to be minimal, one proves as in \cref{lem:dominant-S2-S3} that $f_{W_k}(0) \notin \mathcal{N}(f)(0)$ where $\mathcal{N} := S_1(X) + \mathcal{I}$.
    Since $\ker (V_g) = \ker (V_f) + \mathcal{I}$ (see \cref{embedded_nilpotent_syst}), one checks that 
    \begin{equation}
        g_{W_k}(0) \notin (\Bs_1 \cup \mathcal{P}_k \cup \Bs_{\llbracket 4 , \infty \llbracket} )(g)(0)
    \end{equation}
    and that $k$ is the minimal such integer.
    In particular, by \cref{p:quad-loop-1},
    \begin{equation} \label{eq:quad-loop-y-Z3}
        |(u_1,\dotsc,u_k(t))| = O(|y(t;u)| + t^{\frac 12} \|u_k\|_{L^2}).
    \end{equation}
    From \eqref{eq:y-Z3}, we derive that (see \cite[Lemma 160]{P1})
    \begin{equation}
        y(t;u) = \mathcal{Z}_{\intset{1,3}}(t,g,u)(0) + O\left( \| \mathcal{Z}_{\intset{1,3}}(t,g,u) \|_{\CC^1} |y(t;u)| \right).
    \end{equation}
    Thus, by \eqref{eq:quad-loop-y-Z3} and \eqref{eq:Z3-C1},
    \begin{equation}
        y(t;u) = \mathcal{Z}_{\intset{1,3}}(t,g,u)(0) + O\left( (|y(t;u)|+\|u_k\|_{L^2}) |y(t;u)| \right).
    \end{equation}
    Let $\mathbb{P}$ denote a component along $g_{W_k}(0)$ parallel to $(\Bs_1 \cup \mathcal{P}_k \cup \Bs_{\llbracket 4 , \infty \llbracket} )(g)(0)$.
    The same arguments as in \cref{lem:dominant-S2-S3} prove that
    \begin{equation}
        \mathbb{P} \mathcal{Z}_{\intset{1,3}}(t,g,u)(0) = \xi_{W_k}(t,u) + O\left(|(u_1,\dotsc,u_k)(t)|^2+t\|u_k\|_{L^2}^2\right).
    \end{equation}
    Thus
    \begin{equation}
        \mathbb{P} y(t;u) = \frac 12 \int_0^t u_k^2 + O(t\|u_k\|_{L^2}^2 + |(u_1,\dotsc,u_k(t))|^2 + |y(t;u)|^2 + |y(t;u)| \|u_k\|_{L^2} ).
    \end{equation}
    Using \eqref{eq:quad-loop-y-Z3} once more and writing 
    \begin{equation}
        |y(t;u)| \|u_k\|_{L^2} 
        = O \left( \| u_k \|_{L^2}^3 + |y(t;u)|^{\frac 32} \right)
    \end{equation}
    proves the presence of a drift along $g_{W_k}(0)$, parallel to $(\Bs_1 \cup \mathcal{P}_k \cup \Bs_{\llbracket 4 , \infty \llbracket} )(g)(0)$ as $(t,\|u\|_{W^{-1,\infty}}) \to 0$, in the strong sense of \cref{Def:drift}.
    This concludes the proof by \cref{lem:embedded-drift}.
\end{proof}

\section{Obstructions without analyticity}
\label{s:C^k}

Except for this section, all our paper is written with an analyticity assumption on the vector fields $f_0$ and $f_1$.
This allows to work with convergent series.
However, as announced in the introduction, the obstruction mechanisms on which our necessary conditions for controllability rely are sufficiently robust to absorb an approximation scheme for non-analytic vector fields.

Let $\delta > 0$.
For smooth vector fields $f_0$ and $f_1$ in $\CC^\infty(B_\delta;\R^d)$, one can still define all Lie brackets $f_b \in \CC^\infty(B_\delta;\R^d)$ for $b \in \Br(X)$.
The arguments of the next paragraphs will prove that all\footnote{The only exception is the case $m = -1$ of \cref{Thm:Kawski_Wm} which is not included in \cref{cor:wk}.} the statements of \cref{s:main-results} remain true without any change under this (weaker) regularity setting.

Furthermore, even in a finite regularity setting, one can give a sense to some Lie brackets, once evaluated at zero.
This stems from the equilibrium assumption $f_0(0) = 0$.
More precisely, the value of $f_b(0)$ only depends on the coefficients of the Taylor expansion at $0$ of $f_0$ up to order $n_1(b)$ and of $f_1$ up to order $n_1(b)-1$ (see \cref{p:fb-hatfb} below).
This leads to the following definition.

\begin{definition}
    Let $M \in \N^*$, $\delta > 0$, $f_0 \in \CC^{M}(B_\delta;\R^d)$ with $f_0(0) = 0$ and $f_1 \in \CC^{M-1}(B_\delta;\R^d)$.
    Let $\hat{f}_0 := \mathrm{T}_M f_0$ (respectively $\hat{f}_1 := \mathrm{T}_{M-1} f_1$) be the truncated Taylor series at $0$ of $f_0$ (resp.\ $f_1$) of order $M$ (resp.\ $M-1$).
    For $b \in \Br(X)$ with $n_1(b) \in \intset{1,M}$, we define $f_b(0) := \hat{f}_b(0)$.
\end{definition}

With this notation, we will prove that the following corollaries of the main theorems of \cref{s:main-results} hold. 
As a rule of thumb, the theorems continue to hold as soon as the vector fields have enough regularity for the involved Lie brackets to be defined as above.
More rigorously, we assume one extra derivative to be able to estimate the truncation error properly (see \cref{p:error-xhat}).

We make the blanket hypothesis that $f_0(0) = 0$.

\begin{corollary}
    Let $M \in \N^*$.
    Assume that $f_0 \in \CC^{M+1}$, $f_1 \in \CC^M$.
    If system \eqref{syst} is $W^{-1,\infty}$-STLC, then, for every $k \in \N^*$ such that $2 k \leq M$,
    \begin{equation}
        ad_{f_1}^{2k}(f_0)(0) \in S_{\intset{1,2k-1}}(f)(0).
    \end{equation}
\end{corollary}

\begin{corollary} \label{cor:wk}
    Let $M \in \N^*$.
    Assume that $f_0 \in \CC^{M+1}$, $f_1 \in \CC^M$.
    Let $m \in \N$.
    If system \eqref{syst} is $W^{m,\infty}$-STLC, then, for every $k \in \N^*$ such that $\pi(k,m) \leq M$,
    \begin{equation}
        f_{W_k}(0) \in S_{\intset{1,\pi(k,m)} \setminus \{2\}}(f)(0),
    \end{equation}
    where $\pi(k,m)$ is defined in \eqref{eq:pikm}, or, more precisely when $\pi(k,m) \geq 3$, \eqref{eq:S2-S3-comp} holds.
\end{corollary}

\begin{corollary}
    Assume that system \eqref{syst} is $L^\infty$-STLC.
    If $f_0 \in \CC^4$ and $f_1 \in \CC^3$, then $f_{W_2}(0) \in \mathcal{N}_2(f)(0)$ (see \eqref{def:mathcalE2}).
    If $f_0 \in \CC^6$ and $f_2 \in \CC^5$, then $f_{W_3}(0) \in \mathcal{N}_3(f)(0)$ (see \eqref{def:mathcalE3}).
\end{corollary}

\begin{corollary}
    Assume that $f_0 \in \CC^8$ and $f_1 \in \CC^7$.
    Then \cref{thm:sextic} holds.
\end{corollary}

All these corollaries follow form the main theorems and the approximation result \cref{p:error-xhat}.
One writes $x \approx \hat{x}$, where $\hat{x}$ is the solution to a system driven by the truncated Taylor expansions of $f_0$ and $f_1$.
For the $\hat{x}$ system, one can apply the drift results of the previous sections.
Since the truncation error is of the same size as (or smaller than) the error terms which were already absorbed by the drift, the drift conclusion remains true on the state $x$.

\subsection{Brackets at zero only depend on low-order Taylor coefficients}

\begin{lemma} \label{p:fb-hatfb}
    Let $M \in \N^*$, $\delta > 0$, $f_0 \in \CC^\infty(B_\delta;\R^d)$ with $f_0(0) = 0$ and $f_1 \in \CC^\infty(B_\delta;\R^d)$.
    Let $\hat{f}_0 := \mathrm{T}_M f_0$ (respectively $\hat{f}_1 := \mathrm{T}_{M-1} f_1$) be the truncated Taylor series at $0$ of $f_0$ (resp.\ $f_1$) of order~$M$ (resp.\ $M-1$).
    For all $b \in \Br(X)$ with $n_1(b) \leq M$, $f_b(0) = \hat{f}_b(0)$.
\end{lemma}

\begin{proof}
    \step{Notations and preliminary remarks}
    As in \cite[Section 3.1]{JDE}, for two vector fields $g, h \in \CC^\infty(B_\delta;\R^d)$ and $k \in \N^*$, we write $g =_{[k]} h$ when the Taylor expansions of $g$ and $h$ at $0$ are equal up to order $k-1$.
    When $k \geq 2$, $g =_{[k]} \hat{g}$ and $h =_{[k]} \hat{h}$, straightforward computations prove that $[g,h] =_{[k-1]} [\hat{g},\hat{h}]$.
    When $k \geq 1$, $g(0) = 0$, $g =_{[k+1]} \hat{g}$ and $h =_{[k]} \hat{h}$, straightforward computations prove that $[g,h] =_{[k]} [\hat{g},\hat{h}]$, so that there is ``no loss of derivative'' in this weak sense.
    
    \step{Computation of brackets}
    We now proceed by induction on $n_1(b) \in \intset{1,M}$, proving that, for every $b \in \Br(X)$ with $1 \leq n_1(b) \leq M$, $f_b =_{[M+1-n_1(b)]} \hat{f}_b$.
    
    When $n_1(b) = 1$, by symmetry, we can assume that $b = X_1 0^\nu$ for some $\nu \in \N$.
    Since $f_1 =_{[M]} \hat{f}_1$, iterating the previous remarks yields $f_{X_1 0^\nu} =_{[M]} \hat{f}_{X^1 0^\nu}$, which gives the initialization.
    
    Now let $b \in \Br(X)$.
    By symmetry, we can assume that $b = (b_1,b_2) 0^\nu$ for some $\nu \in \N$, with $b_1, b_2 \neq X_0$.
    By the induction hypothesis $f_{b_1} =_{[M+1-n_1(b_1)]} \hat{f}_{b_1}$ and $f_{b_2} =_{[M+1-n_1(b_2)]} \hat{f}_{b_2}$.
    Hence, by the preliminary remark, $f_{(b_1,b_2)} =_{[M+1-n]} \hat{f}_{(b_1,b_2)}$ with $n := 1 + \max n_1(b_1), n_1(b_2) \leq n_1(b)$.
    And by the preliminary remark, bracketing with $f_0$ preserves this approximation level, so we have proved that $f_b =_{[M+1-n_1(b)]} \hat{f}_b$.
    
    \step{Evaluation at zero}
    When $b = X_0$, $f_0(0) = \hat{f}_0(0)$.
    When $b \in \Br(X)$ with $1 \leq n_1(b) \leq M$, we have proved that $f_b =_{[M+1-n_1(b)]} \hat{f}_b$ so $f_b =_{[1]} \hat{f}_b$ and thus $f_b(0) = \hat{f}_b(0)$.
\end{proof}

\subsection{Estimate of the approximation error}

\begin{lemma} \label{p:error-xhat}
    Let $M \in \N^*$, $\delta > 0$, $f_0 \in \CC^{M+1}(B_\delta;\R^d)$ with $f_0(0) = 0$ and $f_1 \in \CC^{M}(B_\delta;\R^d)$.
    Let $\hat{f}_0 := \mathrm{T}_M f_0$ (respectively $\hat{f}_1 := \mathrm{T}_{M-1} f_1$) be the truncated Taylor series at $0$ of $f_0$ (resp.\ $f_1$) of order $M$ (resp.\ $M-1$).
    Then
    \begin{equation} \label{eq:x-hatx}
        x(t;u) - \hat{x}(t;u) = O \left(\|u_1\|_{L^{M+1}}^{M+1}+|u_1(t)|^{M+1}\right),
    \end{equation}
    where $\hat{x}(t;u)$ denotes the solution with initial data $0$ to
    \begin{equation}
        \dot{\hat{x}} = \hat{f}_0(\hat{x}) + u(t) \hat{f}_1(\hat{x}).
    \end{equation}
\end{lemma}

\begin{proof}
    Such an estimate is straightforward to prove when the right-hand side of \eqref{eq:x-hatx} is replaced by $\|u\|_{L^1}^{M+1}$.
    To obtain an estimate involving only $u_1$, we need to consider an appropriate ``auxiliary system'' as in \cite[Section 7]{P1} or \cite[Section 6.3]{JDE}.
    
    \step{Computations on the auxiliary system}
    Let $\Phi_1$ denote the flow of $f_1$, which is well-defined locally.
    We then introduce
    \begin{equation}
        y(t;u) := \Phi_1(-u_1(t),x(t;u)).
    \end{equation}
    This new unknown satisfies $y(0;u) = 0$ and
    \begin{equation}
        \dot{y} = \left(\Phi_1(-u_1(t))_* f_0\right)(\hat{y}),
    \end{equation}
    where $\Phi_1(-u_1(t))_* f_0$ is the push-forward of the vector field $f_0$ by the diffeomorphism $\Phi_1(-u_1(t),\cdot)$.
    In particular, for $v \in \R$ and $p \in \R^d$ small enough, (see e.g.\ \cite[equation (169)]{P1}, albeit with swapped indexes),
    \begin{equation} \label{eq:Phi1f0p-sum}
        (\Phi_1(-v)_* f_0)(p)
        = \sum_{k=0}^{M-1} \frac{v^k}{k!} \ad^k_{f_1}(f_0)(p)
        + \int_0^v \frac{(v-v')^{M-1}}{(M-1)!} \left(\Phi_1(-v')_* \ad^{M}_{f_1}(f_0)\right)(p) \dd v'.
    \end{equation}
    By \cref{p:tool-1} (with $k \gets M$) and \cref{p:tool-2} (with $g \gets \ad^M_{f_1}(f_0)$ and $\nu \gets 0$),
    \begin{equation} \label{eq:adf1f0-O}
        \ad_{f_1}^M(f_0)(p) = \ad_{f_1}^M(f_0)(0) + \underset{|p|\to 0}{O}(|p|).
    \end{equation}
    Moreover, since $f_1 \in \CC^1$,
    \begin{equation}
        \Phi_1(v,p) = 0 + \underset{v \to 0, |p| \to 0}{O}(|v|+|p|)
    \end{equation}
    and
    \begin{equation}
        \left(\partial_p \Phi_1(v,p)\right)^{-1}
        = \operatorname{Id} + \underset{v \to 0, |p| \to 0}{O}(|v|+|p|).
    \end{equation}
    Thus, combining the last three estimates proves that, for $|v'| \leq |v|$
    \begin{equation}
        \left(\Phi_1(-v')_* \ad^{M}_{f_1}(f_0)\right)(p)
        = \ad_{f_1}^M(f_0)(0) + \underset{v \to 0, |p| \to 0}{O}(|v|+|p|).
    \end{equation}
    Substituting in \eqref{eq:Phi1f0p-sum} and using Young's inequality proves that
    \begin{equation} \label{eq:Phi1vf0p-Taylor}
        \left(\Phi_1(-v)_* f_0\right)(p)
        = \sum_{k=0}^{M-1} \frac{v^k}{k!} \ad^k_{f_1}(f_0)(p)
        + \frac{v^M}{M!} \ad^M_{f_1}(f_0)(0)
        + \underset{v \to 0, |p| \to 0}{O}(|v|^{M+1}+|p|^{M+1}).
    \end{equation}
    
    \step{Grönwall estimate for the auxiliary systems}
    We introduce similarly $\hat{y}(t;u)$ using $\hat{\Phi}_1$ (the flow of $\hat{f}_1$) and $\hat{f}_0$.
    Then the counterpart for \eqref{eq:Phi1vf0p-Taylor} holds, \emph{mutatis mutandis}, since $\hat{f}_0$ and $\hat{f}_1$ are smooth.
    Using these estimates, one obtains
    \begin{equation}
        \begin{split}
            \dot{y} - \dot{\hat{y}}
            = \sum_{k=0}^{M-1} \frac{u_1^k(t)}{k!} \left( \ad_{f_1}^k(f_0) (y) - \ad_{\hat{f}_1}^k(\hat{f}_0)(\hat{y}) \right) & + \frac{u_1^M(t)}{M!}  \left( \ad_{f_1}^M(f_0) (0) - \ad_{\hat{f}_1}^M(\hat{f}_0)(0) \right) \\
            & + O\left( |u_1(t)|^{M+1} + |y|^{M+1} + |\hat{y}|^{M+1} \right).
        \end{split}
    \end{equation}
    For $k = 0$, since $f_0 \in \CC^{M+1}$ and $\hat{f}_0 = \mathrm{T}_M f_0$,
    \begin{equation}
        f_0(y) - \hat{f}_0(\hat{y})
        = f_0(y) - \hat{f}_0(y) + \hat{f}_0(y) - \hat{f}_0(\hat{y})
        = O\left(|y|^{M+1} + |y-\hat{y}|\right).
    \end{equation}
    For $k \in \intset{1,M-1}$, one has
    \begin{equation}
        \begin{split}
        \ad_{f_1}^k(f_0) (y) - \ad_{\hat{f}_1}^k(\hat{f}_0)(\hat{y})
       &  = 
        \ad_{f_1}^k(f_0) (y) - \ad_{\hat{f}_1}^k(\hat{f}_0)(y)
        + \ad_{\hat{f}_1}^k(\hat{f}_0)(y) - \ad_{\hat{f}_1}^k(\hat{f}_0)(\hat{y}) \\
        & = O\left(|y|^{M+1-k} + |y-\hat{y}|\right),
        \end{split}
    \end{equation}
    where we used the estimate 
    \begin{equation}
        \ad^k_{f_1}(f_0)(p) = (\mathrm{T}_{M-k} \ad^k_{f_1}(f_0))(p) + O(|p|^{M+1-k}),
    \end{equation}
    which follows from \cref{p:tool-1} and \cref{p:tool-2} (with $g \gets \ad^k_{f_1}(f_0)$ and $\nu \gets M-k$).
    Eventually, we obtain
    \begin{equation}
        \dot{y} - \dot{\hat{y}} =  O\left( |y-\hat{y}| + |u_1(t)|^{M+1} + |y|^{M+1} + |\hat{y}|^{M+1} \right).
    \end{equation}
    Moreover, from classical estimates $|y| = O(\|u_1\|_{L^1})$ and $|\hat{y}| = O(\|u_1\|_{L^1})$ (see e.g. \cite[Lemma 9]{JDE}).
    Thanks to Grönwall's lemma, we conclude
    \begin{equation}
        y(t;u) - \hat{y}(t;u) = O\left(\|u_1\|_{L^{M+1}}^{M+1} \right).
    \end{equation}
    
    \step{Conclusion}
    First, using similar estimates as above, one proves that
    \begin{equation}
        \Phi_1(v,p) - \hat{\Phi}_1(v,p) =  \underset{v \to 0, |p| \to 0}{O} \left(|p|^{M+1} + |v|^{M+1}\right).
    \end{equation}
    (For example, one can bound the difference between the trajectories to $\dot{z} = f_1(z)$, $z(0) = p$ and $\dot{\hat{z}} = \hat{f}_1(\hat{z})$, $\hat{z}(0) = p$, at time $v$, using a Grönwall estimate, then apply Young's inequality).
    Therefore, we obtain
    \begin{equation}
        \begin{split}
        x - \hat{x} & = \Phi_1(u_1(t),y) - \Phi_1(u_1(t),\hat{y}) + \Phi_1(u_1(t),\hat{y}) - \hat{\Phi}_1(u_1(t),\hat{y}) \\
        & = O\left(|y - \hat{y}| + |\hat{y}|^{M+1} + |u_1(t)|^{M+1}\right),
        \end{split}
    \end{equation}
    where we can use again the estimate $\hat{y} = O(\|u_1\|_{L^1})$, which concludes the proof of \eqref{eq:x-hatx}.
\end{proof}

\begin{lemma}
    \label{p:tool-1}
    Let $M \in \N^*$ and $\delta > 0$.
    Let $f_0 \in \CC^{M+1}(B_\delta;\R^d)$ and $f_1 \in \CC^{M}(B_\delta;\R^d)$.
    For each $k \in \intset{1,M}$, there exists $h_k \in \CC^{M+1-k}(B_\delta;\R^d)$ such that 
    \begin{equation}
        \ad_{f_1}^k(f_0) = - D^k f_1 \cdot (f_0, f_1, \dotsc, f_1) + h_k.
    \end{equation}
\end{lemma}

\begin{proof}
    For $k = 1$, this holds with $h_1 := Df_0 \cdot f_1 \in \CC^M$.
    Then the general formula follows by induction on $k$.
\end{proof}

\begin{lemma}
    \label{p:tool-2}
    Let $\nu \in \N$ and $\delta > 0$.
    Assume that $g \in \CC^\nu(B_\delta;\R^d)$ is of the form $g = A f_0 + h$ where $A \in \CC^\nu(B_\delta;\mathcal{M}_d(\R))$ and $h \in \CC^{\nu+1}(B_\delta;\R^d)$.
    Then, if $f_0(0) = 0$ and $f_0 \in \CC^{\nu+1}(B_\delta;\R^d)$,
    \begin{equation}
        g(p) = (\mathrm{T}_\nu g)(p) + \underset{p \to 0}{O}(|p|^{\nu+1}),
    \end{equation}
    where $\mathrm{T}_\nu g$ denotes the truncated Taylor series at $0$ of $g$. 
\end{lemma}

\begin{proof}
    The claimed estimate is straightforward when $g \in \CC^{\nu+1}$.
    In particular, by linearity, one can assume that $h = 0$.
    When $\nu = 0$, $A \in \CC^0$ so is locally bounded, and, since $f_0 \in \CC^1$ with $f_0(0) = 0$, $f_0(p) = \underset{p \to 0}{O}(|p|)$ and $g(p) = A(p) f_0(p) = \underset{p \to 0}{O}(|p|)$.
    Then, one proceeds by induction.
    Assuming \cref{p:tool-2} holds for some $\nu \in \N$, let us prove it at step $\nu+1$.
    Using Taylor's formula
    \begin{equation}
        g(p) = g(0) + \int_0^1 (Dg(sp)) p \dd s.
    \end{equation}
    Moreover, $Dg = (DA) f_0 + A (Df_0)$, where $DA \in \CC^\nu$, $f_0 \in \CC^{\nu+2}$, $A Df_0 \in \CC^{\nu+1}$.
    In particular, the induction assumption applies and
    \begin{equation}
        Dg(sp) = (\mathrm{T}_\nu (Dg)) (sp) + \underset{p \to 0}{O}(|sp|^{\nu+1}).
    \end{equation}
    Combining both equalities yields
    \begin{equation}
        g(p) = g(0) + \int_0^1 (\mathrm{T}_\nu (Dg)) (sp) p \dd s
        + \underset{p \to 0}{O}(|p|^{\nu+2})
        = (\mathrm{T}_{\nu+1}g)(p) + \underset{p \to 0}{O}(|p|^{\nu+2}),
    \end{equation}
    which concludes the proof.
\end{proof}

\appendix

\section{Proofs of technical results and estimates}

\subsection{On the differences between STLC definitions}
\label{s:chain}

In this paragraph, we prove the claim made in \cref{s:intro-def} concerning the invalidity of the three reciprocal implications in~\eqref{Wm-STLC_implications}.
For $k \in \{ 3, 4, 5 \}$, consider the system
\begin{equation} \label{eq:x22-x1k}
    \begin{cases}
        \dot{x}_1 = u \\
        \dot{x}_2 = x_1 \\
        \dot{x}_3 = x_2^2 - x_1^k.
    \end{cases}
\end{equation}

\begin{proposition}
    The following results hold for system \eqref{eq:x22-x1k}:
    \begin{itemize}
        \item For $k = 3$, it is $L^\infty$-STLC but not $W^{1,\infty}$-STLC.
        \item For $k = 4$, it is not $L^\infty$-STLC, but there exists $\rho > 0$ such that it is $\rho$-bounded-STLC.
        \item For $k = 5$, it is $W^{-1,\infty}$-STLC, but not $\rho$-bounded-STLC for any $\rho > 0$.
    \end{itemize}
\end{proposition}

\begin{proof}
    \step{Case $k=3$}
    With $k = 3$, the $L^\infty$-STLC of system \eqref{eq:x22-x1k} follows from the same arguments as for system \eqref{syst:Jakubczyk} (with the opposite sign), studied in \cref{Subsec:Ex_W2}.
    The fact that it is not $W^{1,\infty}$-STLC follows from \cref{Thm:Kawski_Wm} with $m = 1$ and $k = 2$ (or \cite[Theorem 3]{JDE}), which states that $f_{W_3}(0) \in S_1(f)(0)$ is a necessary condition of $W^{1,\infty}$-STLC (see also \cite[Section 2.3.3]{JDE}).
    
    \step{Case $k=4$} 
    Both claims are proved in \cref{Subsec_Optim_Wk} concerning system \eqref{eq:ex:W2vsQ111}.

    \step{Case $k=5$}
    The fact that system \eqref{eq:x22-x1k} is $W^{-1,\infty}$-STLC for $k=5$ follows from Sussmann's $\mathcal{S}(\theta)$ condition\footnote{Sussmann's initial proof requires $\theta \leq 1$, but, if one is only interested in $W^{-1,\infty}$-STLC, it is possible to work with any $\theta \in [0,+\infty)$.} (see \cite[Theorem~7.3]{MR872457} or \cite[Theorem 3.29]{zbMATH05150528}) with $\theta > 3/2$ (here $f_{W_2}(0)$ of type (2,3) is compensated by $f_{R_{1,1,1,1,0}}(0)$ of type (5,1)).
    It can also be proved explicitly using oscillating controls as in \cite[Section 2.4.1]{JDE} or \cref{Subsec_Optim_Wk}.
    
    Let $\rho > 0$. 
    Let us prove that \eqref{eq:x22-x1k} is not $\rho$-bounded-STLC.
    Let $u \in L^\infty((0,T);\R)$ with $\|u\|_{L^\infty} \leq \rho$ such that $x_1(T;u) = 0$,
    \begin{equation}
        \int_0^T u_1^4 = - 3 \int_0^T u u_1^2 u_2 
        \leq 3 \rho \left(\int_0^T u_1^4 \int_0^T u_2^2 \right)^{\frac 12}.
    \end{equation}
    Thus
    \begin{equation}
        \left| \int_0^T u_1^5 \right|
        \leq \|u_1\|_{L^\infty} \int_0^T u_1^4
        \leq \rho T \| u_1 \|_{L^4}^4
        \leq 9 \rho^3 T \int_0^T u_2^2.
    \end{equation}
    In particular, if $T \leq 1/(9\rho^3)$, one has $x_3(T;u) \geq 0$ and the system is not controllable.
\end{proof}

\subsection{On the structure of ``bad'' brackets}
\label{s:bad-structure}

In this paragraph, we prove the claim made in \cref{s:main-results} that, due to the structure of the brackets $\ad^{2k}_{X_1}(X_0)$, $W_k$ and $D := \ad^2_{P_{1,1}}(X_0)$, they are always required to be compensated by the Agrachev--Gamkrelidze sufficient condition of \cref{thm:AG}.
Since they are of type (even, odd), this follows from the following claims.

\begin{lemma} \label{lem:b-pi}
    Let $\Pi^1 \subset \Br(X)$ as in \cref{thm:AG}.
    Then $X_0 \in \Pi^1$.
\end{lemma}

\begin{lemma}
    Let $\Pi^1 \subset \Br(X)$ as in \cref{thm:AG}.
    Let $k \in \N^*$. Then
    \begin{equation} \label{eq:eval2k-evalPi}
        \eval(\ad^{2k}_{X_1}(X_0)) \notin \vect \{ \eval(\pi) ; \pi \in \Pi^{\mathrm{even}} \}
        \quad \text{in} \quad \mathcal{L}(X).
    \end{equation}
\end{lemma}

\begin{proof}
    By contradiction, assume that there exist $r \in \N^*$, $\pi_1, \dotsc, \pi_r \in \Pi^\mathrm{even}$ and $\alpha_1, \dotsc, \alpha_r \in \R$ such that $\eval(\ad^{2k}_{X_1}(X_0)) = \alpha_1 \eval(\pi_1) + \dotsb + \alpha_r \eval(\pi_r)$ in $\mathcal{L}(X)$.
    Thus, there exists $\pi = \pi_j \in \Pi^\mathrm{even}$ such that $\ad^{2k}_{X_1}(X_0) \in \supp_{\Bs} \eval(\pi)$.
    Since $n_0(\pi) = 1$ and $n_0(\pi) = 2k$, $\eval(\pi) = \pm \eval(\ad^{2k}_{X_1}(X_0))$.
    Since $\eval(\pi) \neq 0$ and $n_0(\pi) = 1$, one has $\lambda(\pi) = X_1$ or $\mu(\pi) = X_1$.
    Thus $X_1 \in \Pi^1$, and $\ad^{2k}_{X_1}(X_0) \in \Pi^{1 + 2k}\subset \Pi^\mathrm{odd}$, which contradicts the initial assumption since $\Pi^1$ freely generates $\operatorname{Lie}(\Pi^1)$.
\end{proof}

\begin{lemma}
    Let $\Pi^1 \subset \Br(X)$ as in \cref{thm:AG}.
    Let $k \in \N^*$.
    Then
    \begin{equation} \label{eq:evalWk-evalPi}
        \eval(W_k) \notin \vect \{ \eval(\pi) ; \pi \in \Pi^{\mathrm{even}} \}
        \quad \text{in} \quad \mathcal{L}(X).
    \end{equation}
\end{lemma}

\begin{proof}
    By contradiction, let $r \in \N^*$, $\pi_1, \dotsc, \pi_r \in \Pi^\mathrm{even}$ and $\alpha_1, \dotsc, \alpha_r \in \R$ such that $\eval(W_k) = \alpha_1 \eval(\pi_1) + \dotsb + \alpha_r \eval(\pi_r)$ in $\mathcal{L}(X)$.
    Thus, there exists $\pi = \pi_j \in \Pi^\mathrm{even}$ such that $W_k \in \supp_{\Bs} \eval(\pi)$.
    Since $W_k$ is a germ (see \cref{def:germ}), one cannot have $\lambda(\pi) = X_0$ or $\mu(\pi) = X_0$.
    Thus $n_1(\lambda(\pi)) = n_1(\mu(\pi)) = 1$.
    This implies the existence of a unique $l \in \intset{0,k-1}$ and $\pi^* \in \Pi^1$ such that $\eval(\pi^*) = \pm \eval (M_l)$.
    Thus $W_k \in \vect \{ \eval(\pi) ; \pi \in \Pi^{1 + 2(k-l)} \subset \Pi^\mathrm{odd} \}$, which contradicts the initial assumption since $\Pi^1$ freely generates $\operatorname{Lie}(\Pi^1)$.
\end{proof}

\begin{lemma}
    Let $\Pi^1 \subset \Br(X)$ as in \cref{thm:AG}.
    Then
    \begin{equation} \label{eq:evalD-evalPi}
        \eval(D) \notin \vect \{ \eval(\pi) ; \pi \in \Pi^{\mathrm{even}} \}
        \quad \text{in} \quad \mathcal{L}(X).
    \end{equation}
\end{lemma}

\begin{proof}
    By contradiction, let $r \in \N^*$, $\pi_1, \dotsc, \pi_r \in \Pi^\mathrm{even}$ and $\alpha_1, \dotsc, \alpha_r \in \R$ such that $\eval(D) = \alpha_1 \eval(\pi_1) + \dotsb + \alpha_r \eval(\pi_r)$ in $\mathcal{L}(X)$.
    Thus, there exists $\pi = \pi_j \in \Pi^\mathrm{even}$ such that $D \in \supp_{\Bs} \eval(\pi)$.
    Thus $n_1(\pi) = 6$ and $n_0(\pi) = 3$.
    We write $\pi = (a,b)$.
    Since $D$ is a germ (see \cref{def:germ}) and $\Bs$ is stable by bracketting with $X_0$, one cannot have $a = X_0$ or $b = X_0$.
    By symmetry, one can assume $1 \leq n_1(a) \leq n_1(b) \leq 5$.

    \step{Case $n_1(a) = 1$ and $n_1(b) = 5$}
    If $n_0(a) \geq 1$, up to decomposing $a$ and $b$ on $\Bs$, by \cref{p:D-dec-1+5}, $D \notin \supp_{\Bs} \eval((a,b))$, which contradicts $D \in \sup_{\Bs} \eval(\pi)$.
    Thus $a = X_1 \in \Pi^1$ and $\Pi^1 = \{ X_0, X_1 \}$.
    Hence $D \in \vect \{ \eval(\pi) ; \pi \in \Pi^{9} \subset \Pi^\mathrm{odd} \}$ which contradicts the initial assumption since $\Pi^1$ freely generates $\operatorname{Lie}(\Pi^1)$.
    
    \step{Case $n_1(a) = 2$ and $n_1(b) = 4$}
    Then, up to decomposing $a$ and $b$ on $\Bs$, by \cref{p:D-dec-2+4}, $D \notin \supp_{\Bs} \eval((a,b))$, which contradicts $D \in \sup_{\Bs} \eval(\pi)$.
    Hence this case doesn't happen.
    
    \step{Case $n_1(a) = n_1(b) = 3$}
    By symmetry, one can assume that $n_0(a) = 1$ and $n_0(b) = 2$.
    Thus $\eval(a) = \pm \eval(P_{1,1})$ and $\eval(b) = \pm \eval(P_{1,2})$ or $\eval(b) = \pm \eval((P_{1,1},X_0))$.
    Since $(P_{1,1},P_{1,2}) \in \Bs$ and $D \in \supp_{\Bs} \eval(\pi)$, we are in the second case, i.e.\ $\pi = (P_{1,1}, (P_{1,1}, X_0))$ (or a permutation thereof).
    In particular $\eval(\pi) = \pm \eval(D)$.
    Thus $P_{1,1} = \lambda(\pi)$ (or a permutation thereof) belongs to $\Pi^1$, and thus $\pi \in \Pi^3$, which contradicts the initial assumption.
\end{proof}

\subsection{Universal rough estimate for coordinates of the second kind}
\label{s:proof-bound-xi-univ}

\begin{proof}[Proof of \cref{p:xib-u1-Lk}]
    The proof is by induction on $k \in \N^*$. 
    
    \step{Case $k = 1$}
    Then $b=X_1 0^\nu$ for some $\nu \geq 1$ and $|b| = \nu + 1$.
    Thus, for every $t>0$ and $u \in \lone$,
    \begin{equation}
        |\xi_b(t,u)| = \left|\int_0^t \frac{(t-s)^{\nu-1}}{(\nu-1)!} u_1(s) \dd s \right| \leq \frac{t^{\nu-1}}{(\nu-1)!} \|u_1\|_{L^1}
        \leq 2 \frac{(2t)^{\nu+1}}{(\nu+1)!} t^{-2} \|u_1\|_{L^1},
    \end{equation}
    which gives the conclusion with $c(1):=4$. 
    
    \step{Case $k \geq 2$}
    To simplify notations, we write $c$ instead of $c(k-1)$ and, without loss of generality, we assume that $1 \leq c(1) \leq \dots \leq c(k-1)=c$. 
    Let $b \in \Bs \setminus\{X_1\}$ with $n_1(b)=k$. 
    Then $b=b^* 0^\nu$ for some $\nu \geq 0$ and there exists $j \in \N^*$, $m_1,\dots,m_j \in \N^*$, $m \in \N$ and $b_1>\dots>b_j>X_1 \in \Bs_{\intset{1,k-1}}$ such that $b^*=\ad_{b_1}^{m_1} \dots \ad_{b_j}^{m_j} \ad_{X_1}^m(X_0)$. 
    In particular, $k=n_1(b)=n_1(b^*)=m_1 n_1(b_1)+\dots+m_j n_1(b_j)+m$ and $|b| = m_1 |b_1| + \dots +m_j |b_j|+m+\nu+1$.

    First, for each $i \in \intset{1,j}$, using the induction assumption and Hölder's inequality,
    \begin{equation}
        |\xi_{b_i}(t,u)| \leq \frac{(ct)^{|b_i|}}{|b_i|!} t^{-n_1(b_i) \left( 1 + \frac{1}{k} \right)} \|u_1\|_{L^k}^{n_1(b_i)}.
    \end{equation}
    Thus, by \eqref{eq:def:coord2},
    \begin{equation}
        \begin{split}
            |\xi_{b^*}(t,u)| & = \left| \int_0^t \frac{\xi_{b_1}^{m_1}(s,u)}{m_1!} \dots  \frac{\xi_{b_j}^{m_j}(s,u)}{m_j!} \frac{u_1^m(s)}{m!} \dd s \right|
            \\ &  \leq \frac{(ct)^{1+m+\sum m_i |b_i|}}{m!|b_1|!^{m_1} \dotsb |b_j|!^{m_j}} t^{-(m+\sum m_i n_1(b_i)) \left(1+\frac 1 k\right)} \|u_1\|_{L^k}^{m+\sum m_i n_1(b_i)} \\
            & \leq \frac{\left(2^{k+1}ct\right)^{|b^*|}}{|b^*|!} t^{-(1+k)} \|u_1\|_{L^k}^k,
        \end{split}
    \end{equation}
    where we used $\|u_1\|_{L^m}^m \leq t^{1+m} t^{-m\left(1+\frac 1k\right)}\|u_1\|_{L^k}^m$ and
    \begin{equation}
        |b^*|! = \left( 1 + m + \sum_{i=1}^j m_i |b_i| \right)!
        \leq 2^{((\sum m_i + 2) - 1)|b^*|} 1! m! \prod_{i=1}^j |b_i|!^{m_i}
    \end{equation}
    which follows from \eqref{factorielle} and the estimate $m + m_1 + \dotsb + m_j \leq k$.
    
    Finally, if $\nu \geq 1$, using \cref{Lem:adnuX0} and \eqref{factorielle},
    \begin{equation}
        \begin{split}
            |\xi_b(t,u)| & = \left| \int_0^t \frac{(t-s)^{\nu-1}}{(\nu-1)!} \xi_{b^*}(s,u) \dd s  \right| 
            \\ & \leq 
            \frac{t^\nu}{\nu!} 
            \frac{(2^{k+1}ct)^{|b^*|}}{|b^*|!}
            t^{-(1+k)}
            \|u_1\|_{L^k}^k
            \\ & \leq 
            \frac{(2^{k+2}ct)^{|b|}}{|b|!} t^{-(1+k)} \|u_1\|_{L^k}^k,
        \end{split}
    \end{equation}
    which gives the conclusion with $c(k) := 2^{k+2} c$.
\end{proof}

\subsection{Precise estimates of coordinates up to the fifth order}
\label{s:proof-bound-xi-12345}

We start with an elementary estimate.

\begin{lemma}
    For every $p \in [1,\infty]$, $j_0 \leq j \in \N^*$, $t>0$ and $u \in \lone$,
    \begin{equation} \label{majo_uj_Lp/Lpj0}
        \|u_j\|_{L^p} \leq \frac{t^{j-j_0}}{(j-j_0)!} \|u_{j_0}\|_{L^p}.
    \end{equation}
\end{lemma}

\begin{proof}
    One can assume $j > j_0$
    By definition, $u_j$ is the $(j-j_0)$-th primitive of $u_{j_0}$ vanishing iteratively at zero, i.e.\
    \begin{equation}
        u_j(s) = \int_0^s \frac{(s-s')^{j-j_0-1}}{(j-j_0-1)!} u_{j_0}(s') \dd s'.
    \end{equation}
    Thus $u_j = g_{j-j_0-1} * \bar{u}_{j_0}$, where $\bar{u}_{j_0}$ is the extension of $u_{j_0}$ from $(0,t)$ to $\R$ by zero and $g_{\nu}(s) := s^\nu / \nu!$ for $s \in (0,t)$ and $0$ elsewhere, so that $\|g_\nu\|_{L^1} = t^{\nu+1}/(\nu+1)!$.
    Hence, \eqref{majo_uj_Lp/Lpj0} follows from Young's convolution inequality.
\end{proof}

This leads to the following estimates.

\begin{proof}[Proof of \cref{p:xi-bounds}]
    We prove the bounds one by one.
    \begin{enumerate}
        \item By \eqref{xi_S1}, Hölder's inequality, \eqref{majo_uj_Lp/Lpj0} and \eqref{factorielle},
        \begin{equation}
            \begin{split}
                |\xi_{M_j}(t,u)| \leq \| u_j \|_{L^1} 
                & \leq t^{1-\frac 1 p}\| u_j \|_{L^p} \\
                & \leq \frac{t^{j-j_0}}{(j-j_0)!} t^{1 - \frac 1 p} \|u_{j_0}\|_{L^p} \\
                & \leq (j_0+1)! \frac{(2t)^{j+1}}{(j+1)!} t^{-(j_0+1)} t^{1 - \frac 1 p} \| u_{j_0} \|_{L^p},
            \end{split}
        \end{equation}
        which proves \eqref{bound-xiMj/J0} with $c := 2 (j_0+1)!$ since $|M_j| = j+1$.
        
        \item By \eqref{xi_Wjnu}, Hölder's inequality, \eqref{majo_uj_Lp/Lpj0} and \eqref{factorielle},
        \begin{equation}
            \begin{split}
                |\xi_{W_{j,\nu}}(t,u)| 
                \leq \frac{t^\nu}{\nu!} t^{1-\frac 1 p} \|u_j\|_{L^{2p}}^2
                & \leq \frac{t^\nu}{\nu!} \frac{t^{2(j-j_0)}}{(j-j_0)!^2} t^{1-\frac 1 p} \|u_{j_0}\|_{L^{2p}}^2 \\
                & \leq (2j_0+1)! \frac{(2^3 t)^{2j+\nu+1}}{(2j+\nu+1)!} t^{-(2j_0+1)} t^{1-\frac 1 p} \|u_{j_0}\|_{L^{2p}}^2,
            \end{split}
        \end{equation}
        which proves \eqref{bound-xiWjnu/j0} with $c := 2^2 (2j_0+1)!$ since $|W_{j,\nu}| = 2j+\nu+1$.
        
        \item For \eqref{bound-xiPjknu/j0k0}, we proceed as in the second item, starting from \eqref{xi_S3}.
        
        \item For \eqref{bound-xiQjklnu/j0k0l0}, we proceed as in the second item, starting from \eqref{xi_Qljknu}.
        
        \item By \eqref{xi_Qbemolknu} and \eqref{bound-xiWjnu/j0}, there exists $c_2 > 0$ such that
        \begin{equation}
            \begin{split}
                |\xi_{Q^\flat_{j,\mu,\nu}}(t,u)| 
                & =
                \frac 1 2 \left| \int_0^t \frac{(t-s)^\nu}{\nu!} \xi_{W_{j,\mu}}^2(s,u) \dd s \right|
                \\& \leq \frac{t^{\nu+1}}{(\nu+1)!} \left( \frac{(c_2 t)^{|W_{j,\mu}|}}{|W_{j,\mu}|!} t^{-(2j_0+1)} t^{1-\frac 1p} \|u_{j_0}\|_{L^{2p}}^2 \right)^2
                \\ & \leq \frac{(2^2 c_2 t)^{2|W_{j,\mu}|+\nu+1}}{(2|W_{j,\mu}|+\nu+1)!} t^{-(4j_0+3)} t^{3 - \frac 2 p} \|u_{j_0}\|_{L^{2p}}^4
            \end{split}
        \end{equation}
        using \eqref{factorielle}, which proves \eqref{bound-xiQflatjmunu/j0} with $c := 2^2 c_2$ since $|Q^\flat_{j,\mu,\nu}| = 2 |W_{j,\mu}|+\nu+1$.
        
        \item By \eqref{xi_Qdieseknujmu}, \eqref{majo_uj_Lp/Lpj0} and \eqref{bound-xiWjnu/j0}, there exists $c_2 > 0$ such that
        \begin{equation}
            \begin{split}
                |\xi_{Q^\sharp_{j,\mu,k,\nu}}(t,u)| 
                & =
                \frac 1 2 \left| \int_0^t \frac{(t-s)^\nu}{\nu!} \xi_{W_{j,\mu}}(s,u) u_k^2(s) \dd s \right|
                \\& \leq \frac{t^\nu}{\nu!} \left( \frac{(c_2 t)^{|W_{j,\mu}|}}{|W_{j,\mu}|!} t^{-(2j_0+1)} t^{1-\frac 1{p_1}} \|u_{j_0}\|_{L^{2p_1}}^2 \right) t^{1-\frac{1}{p_2}}
                \left(\frac{t^{k-k_0}}{(k-k_0)!}  \|u_{k_0}\|_{L^{2p_2}}\right)^2 \\
                & \leq 
                (2k_0+1)! \frac{(2^4 c_2 t)^{|W_{j,\mu}|+2k+\nu+1}}{(|W_{j,\mu}|+2k+\nu+1)!} \\
                & \qquad \qquad \qquad \times
                t^{-(2j_0+2k_0+2)} t^{2-\frac{1}{p_1}-\frac{1}{p_2}} \|u_{j_0}\|_{L^{2p_1}}^2 \|u_{k_0}\|_{L^{2p_2}}^2
            \end{split}
        \end{equation}
        using \eqref{factorielle}, which proves \eqref{bound-xiQflatjmunu/j0} with $c := 2^4 c_2 (2k_0+1)!$ since $|Q^\sharp_{j,\mu,k,\nu}| = 2 |W_{j,\mu}|+2k+\nu+1$.
        
        \item For \eqref{bound-xiRjklmnu/j0k0l0m0}, we proceed as in the second item, starting from \eqref{xi_Rl3l2l1knu}.
        
        \item By \eqref{xi_Rdiesel3l2l1knu}, Hölder's inequality, \eqref{majo_uj_Lp/Lpj0}, and \eqref{bound-xiWjnu/j0}, there exists $c_2 > 0$ such that
        \begin{equation}
            \begin{split}
            |\xi_{R^\sharp_{j,k,l,\mu,\nu}}(t,u)| 
            & = \alpha_{j,k} \left| \int_0^t \frac{(t-s)^\nu}{\nu!} \xi_{W_{l,\mu}}(s,u) u_k(s) u_j^2(s) \dd s \right| 
            \\
            & \leq \frac{t^\nu}{\nu!} 
            \left( \frac{(c_2t)^{|W_{l,\mu}|}}{|W_{l,\mu}|!} t^{-(2l_0+1)} t^{1-\frac 1 p}\|u_{l_0}\|_{L^{2p}}^2 \right) 
            t^{1-\frac{1}{p_1}-\frac{1}{p_2}} \\
            & \qquad \qquad\qquad\qquad
            \times \left(\frac{t^{j-j_0}}{(j-j_0)!} \|u_{j_0}\|_{L^{2p_1}}\right)^2
             \frac{t^{k-k_0}}{(k-k_0)!} \|u_{k_0}\|_{L^{p_2}}
            \\
            & \leq (2j_0+k_0+1)! \frac{(2^4 c_2)t^{|W_{l,\mu}|+2j+k+\nu+1}}{(|W_{l,\mu}|+2j+k+\nu+1)!} t^{-(2l_0+2j_0+k_0+2)} 
            \\ & \qquad \qquad\qquad\qquad
            \times t^{2-\frac 1 p -\frac{1}{p_1}-\frac{1}{p_2}} \|u_{l_0}\|_{L^{2p}}^2 \|u_{j_0}\|_{L^{2p_1}}^2 \|u_{k_0}\|_{L^{p_2}}
            \end{split}
        \end{equation}
        using \eqref{factorielle}, which proves \eqref{bound-xiRsharpjklmunu/j0k0l0} with $c := 2^4 c_2 (2j_0+k_0+1)!$ since $|R^\sharp_{j,k,l,\mu,\nu}| = |W_{l,\mu}|+2j+k\nu+1$.
        \qedhere
    \end{enumerate}
\end{proof}

\subsection{Black-box estimates for the dominant part of the logarithm}
\label{s:proof-black-box}

We prove \cref{p:sum-xi-XI,p:sum-cross-XI} of \cref{s:black-box}.

\begin{proof}[Proof of \cref{p:sum-xi-XI}]
    We have, using \eqref{eq:xib-XI-L}, \eqref{eq:bracket.analytic} and $|\{b \in \Bs ; |b| = \ell \}| \leq 2^\ell$,
    \begin{equation}
        \begin{split}
            \sum_{b \in \mathcal{E}} \opnorm{ \xi_b(t,u) f_b }_{r'} 
            & \leq 
            \sum_{\sigma = 1}^{L} \sum_{\ell = \sigma}^{+\infty} \sum_{\substack{b \in \mathcal{E}\\|b|=\ell,\sigma(b) = \sigma}} \frac{(ct)^{\ell}}{\ell!} t^{-\sigma} \Xi(t,u) \frac{r}{9} (\ell-1)! \left(\frac{9\opnorm{f}_r}{r}\right)^\ell \\
            & \leq \frac{r}{9} \Xi(t,u) \sum_{\sigma = 1}^{L} \sum_{\ell = \sigma}^{+\infty} \left(\frac{18 ct \opnorm{f}_r}{r}\right)^\ell t^{-\sigma}
        \end{split}
    \end{equation}
    which converges provided that $18 ct \opnorm{f}_r < r$, and is then bounded by $C \Xi(t,u)$ for an appropriate constant $C$ depending on $r,c$ and $\opnorm{f}_r$ (independent of $t$ for $t \ll 1$, e.g.\ when $18 ct \opnorm{f}_r < r/2$).
\end{proof}

The proof of \cref{p:sum-cross-XI} relies on the following algebraic lemmas.

\begin{lemma} \label{lem:A1-itere}
    Let $M \in \N^*$.
    There exists $C_1(M) > 0$ such that, for all $a \in \Br(\Bs)$ and $b \in \Bs_{\intset{1,M}}$,
    \begin{equation} \label{eq:est-IA-b}
        | \langle \textup{\textsc{i}}(a), b \rangle_{\Bs}| \leq C_1^{(|a|_{\Bs}-1)|b|},
    \end{equation}
    where $\textup{\textsc{i}}$ denotes the canonical morphism of magmas from $\Br(\Br(X))$ to $\Br(X)$.
    For example, when $a = (X_1,M_1) \in \Br(\Bs)$, $|a|_{\Bs} = 2$, and $\textup{\textsc{i}}(a) = (X_1, (X_1, X_0)) \in \Br(X)$ with $|\textup{\textsc{i}}(a)| = 3$.
\end{lemma}

\begin{proof}
    For $B \in \mathcal{L}(X)$, let $\| B \| := \sum_{c \in \Bs} |\langle B, c \rangle_{\Bs}|$.
    By \cite[Theorem 1.9]{A1}, there exists $C_1(M) > 0$ such that, for all $b_1, b_2 \in \Bs$ with $n_1(b_1) + n_1(b_2) \leq M$,
    \begin{equation}
        \label{eq:A1-geom}
        \|[b_1,b_2]\| \leq C_1^{|b_1|+|b_2|}.
    \end{equation}
    Let us prove by induction on $|a|_{\Bs} \geq 1$ that, when $n_1(\textsc{i}(a)) \leq M$, 
    \begin{equation} \label{eq:est-IA}
        \| \textsc{i}(a) \| \leq C_1^{(|a|_{\Bs}-1)|\textsc{i}(a)|}.
    \end{equation}
    The case $|a|_{\Bs} = 1$ is trivial since then $\langle \textsc{i}(a), b \rangle_{\Bs} \in \{ 0, 1 \}$.
    The case $|a|_{\Bs} = 2$ is exactly \eqref{eq:A1-geom}.
    Now if $|a|_{\Bs} > 2$, $a = (a_1, a_2)$ with $a_1, a_2 \in \Br(\Bs)$, where, by the induction assumption
    \begin{equation}
        \eval (\textsc{i}(a_i)) = \sum_{c \in \Bs} \alpha_i^c \eval(c) 
        \quad \text{where} \quad \sum |\alpha_i^c| \leq C_1^{(|a_i|_{\Bs}-1)|\textsc{i}(a_i)|}.
    \end{equation}
    Then, using \eqref{eq:A1-geom} once more,
    \begin{equation}
        \| \textsc{i}(a) \| 
        \leq \sum_{c_1,c_2 \in \Bs} |\alpha_1^{c_1}| |\alpha_2^{c_2}| \| [c_1, c_2] \|
        \leq C_1^{(|a_1|_{\Bs}-1+|a_2|_{\Bs}-1+1)|\textsc{i}(a)|},
    \end{equation}
    proving \eqref{eq:est-IA}.
    Eventually, when $\langle \textsc{i}(a), b \rangle_{\Bs} \neq 0$, $n_1(\textsc{i}(a))=n_1(b) \leq M$ and $|\textsc{i}(a)| = |b|$ so that \eqref{eq:est-IA-b} follows from \eqref{eq:est-IA}.
\end{proof}

\begin{lemma}
    Let $M \in \N^*$.
    There exists a constant $C_M > 0$ such that, for every $q \geq 2$, $h \in (\N^*)^q$, $b_1 < \dotsb < b_q \in \Bs$ and $b \in \Bs_{\intset{1,M}}$,
    \begin{equation} \label{eq:Fqh-CM}
        | \langle F_{q,h}(b_1,\dotsc,b_q), b \rangle_{\Bs} | \leq C_M^{|b|}.
    \end{equation}
\end{lemma}

\begin{proof}
    By homogeneity, $\langle F_{q,h}(b_1,\dots,b_q), b \rangle_{\Bs} \neq 0$ entails that
    \begin{align}
        \label{eq:sum-hini}
        h_1 n_1(b_1) + \dotsb + h_q n_1(b_q) & = n_1(b) \leq M, \\
        \label{eq:sum-hielli}
        h_1 |b_1| + \dotsb + h_q |b_q| & = |b|.
    \end{align}
    For each $q \geq 2$ and $h \in (\N^*)^q$, there exists a finite subset $\mathcal{A} \subset \Br(\{ Y_1, \dotsc, Y_q \})$ and coefficients $(\alpha_a)_{a \in \mathcal{A}}$ such that $F_{q,h}(Y_1,\dotsc,Y_q) = \sum \alpha_a \eval(a)$.
    By \eqref{eq:sum-hini}, the set of considered $q$ and $h$ is finite, so there exists a constant $C_F(M) > 0$ (depending only on $M$) such that $\sum |\alpha_a| \leq C_F$.
    Let $a \in \mathcal{A}$ of homogeneity $h_1, \dotsc, h_q$ with respect to $Y_1, \dotsc, Y_q$. 
    In particular $|a| = \sum h_i \leq M$ by \eqref{eq:sum-hini}.
    By \cref{lem:A1-itere},
    \begin{equation}
        |\langle a(b_1, \dotsc, b_q), b \rangle_{\Bs}| \leq C_1^{(|a|-1)|b|} \leq C_1^{M |b|},
    \end{equation}
    where $a(b_1,\dotsc,b_q)$ denotes the image of $a$ through the homomorphism of Lie algebras sending $Y_i$ to $b_i \in \mathcal{L}(X)$.
    Eventually, we conclude that \eqref{eq:Fqh-CM} holds with $C_M := C_F C_1^M$.
\end{proof}

\begin{proof}[Proof of \cref{p:sum-cross-XI}]
    By \cref{Prop:ZM_Bstar}, for $b \in \Bs_{\intset{1,M}}$,
    \begin{equation}
        \eta_b(t,u) - \xi_b(t,u) 
        = 
        \sum_{\substack{q \geq2, h\in (\N^*)^q \\ b_1 > \dots > b_q \in \Bs \setminus\{ X_0\} }} \xi_{b_1}^{h_1}(t,u) \dots \xi_{b_q}^{h_q}(t,u) \langle F_{q,h}(b_1,\dots,b_q), b \rangle_{\Bs}.
    \end{equation}
    When $\langle F_{q,h}(b_1,\dots,b_q), b \rangle_{\Bs}\neq 0$, there are at most $M$ choices for $q$, $M^M$ choices for $h$ and $(2^{|b|})^M$ choices for the $b_i$ (since $|\{ c \in \Bs ; |c| \leq \ell \}| \leq 2^\ell$).
    Hence, using \eqref{eq:Fqh-CM}, \eqref{factorielle}, \eqref{eq:xib-othercross-XI-L} and $\Xi = O(1)$, we obtain 
    \begin{equation}
        \begin{split}
        |\eta_b(t,u)-\xi_b(t,u)|
        & \leq 
        M^M 2^{M |b|} C_M^{|b|} \sum_{q=1}^M  \sum_{\sigma = 1}^{\min \{ q L, |b| \}} 2^{(M-1)|b|} \frac{(ct)^{|b|}}{|b|!} t^{-\sigma} \Xi(t,u) \\
        & \leq \frac{(D_M t)^{|b|}}{|b|!} t^{-\sigma_b} \Xi(t,u),
        \end{split}
    \end{equation}
    where $\sigma_b := \min \{ M L, |b| \}$ and $D_M > 0$, and the summations over $b \in \mathcal{E}$ follows exactly as in the proof of \cref{p:sum-xi-XI} above.
\end{proof}

\subsection{Bad-bad limiting examples involving \texorpdfstring{$W_3$ versus $Q^\flat$}{W3 vs Qb}}
\label{s:Qflat}

As announced in \cref{s:W3-bad-bad}, in this paragraph, we give examples illustrating that one must include $Q^\flat_{1,0}$, $Q^\flat_{1,1}$ and $Q^\flat_{1,2}$ in the list $\mathcal{N}_3$ of \eqref{def:mathcalE3}.

\paragraph{Limiting example for $Q^\flat_{1,0}$.}
Consider the system
\begin{equation} \label{Ex:W3vsQb10}
    \begin{cases}
        \dot{x}_1=u \\
        \dot{x}_2=x_1 \\
        \dot{x}_3=x_2+x_1^2 \\
        \dot{x}_4=x_3 \\
        \dot{x}_5=x_3^2+2x_1^2 x_4.
    \end{cases}
\end{equation}
Written in the form \eqref{syst}, this system satisfies
\begin{equation}
    \begin{split}
        & f_{M_{i-1}}(0)=e_i \text{ for } i \in \intset{1,4}, \quad 
        f_{W_1}(0) = 2 e_3, \quad
        f_{W_{1,1}}(0) = 2 e_4, \\
        & f_{Q^\flat_{1,0}}(0) = - 8e_5, \quad
        f_{W_3}(0)=2 e_5
    \end{split}
\end{equation}
and $f_b(0) = 0$ for any other $b \in \Bs$.

\begin{proposition}
	System \eqref{Ex:W3vsQb10} is $L^\infty$-STLC (but not $W^{1,\infty}$-STLC).
\end{proposition}

\begin{proof}
	By \cref{Thm:Kawski_Wm}, if a system is $W^{1,\infty}$-STLC, then $f_{W_3}(0) \in S_3(f)(0)$.
	Since this condition is not satisfied by system \eqref{Ex:W3vsQb10}, it is not $W^{1,\infty}$-STLC.
	We now prove that it is $L^\infty$-STLC.
	
	\step{Computation of the state}
    Let us fix $T > 0$.
	Explicit integrations lead to $x_1(T) = u_1(T)$, $x_2(T) = u_2(T)$ and
	\begin{equation} \label{eq:syst-W3-Qb100-x1x5}
		\begin{split}
			x_3(T) & = u_3(T) + \int_0^T u_1^2(t) \dd t, \\
			x_4(T) & = u_4(T) + \int_0^T (T-t) u_1^2(t) \dd t, \\
			x_5(T) & = \int_0^T u_3^2(t) \dd t - \int_0^T \left(\int_0^t u_1^2(s) \dd s\right)^2 \dd t + 2 x_4(T) \int_0^T u_1^2(t) \dd t.
		\end{split}
	\end{equation}
	
	\step{Motions in the linear directions}
    Let $i \in \intset{1,4}$.
	By the usual linear theory, there exists $\bar{u}^i \in L^\infty((0,T);\R)$ such that, for $a \in [-1,1]$, $x(T;a \bar{u}^i) = a e_i + O(a^2)$.
	
	\step{Motion in the easy quadratic direction $+e_5$}
	Take a non-zero function $\chi \in C^\infty_c((0,T);\R)$, normalized such that $\|\chi^{(1)}\|_{L^2} = 1$ and define, for $a \in [0,1]$,
	\begin{equation}
		\mathcal{U}_{e_5}(a) := a \chi^{(4)} - a^2 \left( \int_0^T (\chi^{(3)})^2(t) \dd t \right) \bar{u}^3 - a^2 \left( \int_0^T (T-t) (\chi^{(3)})^2(t) \dd t \right) \bar{u}^4.
	\end{equation}
	Using \eqref{eq:syst-W3-Qb100-x1x5}, one checks that $x(T,\mathcal{U}_{e_5}(a)) = a^2 e_5 + O(a^3)$.
	
	\step{Motion in the difficult quartic direction $-e_5$}
	In order to benefit from the quartic term, we use oscillating controls.
	Let $\phi \in C^\infty(\R;\R)$ be a fixed non-zero $T$-periodic function with $\phi(0)=\phi'(0)=\phi''(0) = 0$, $\langle \phi \rangle = 0$ and $\langle \phi^2 \rangle = \frac 14$.
	Let $\chi \in C^\infty_c((0,T];\R)$ such that $\chi'(T)=\chi''(T)=0$ to be chosen later.
	For $a \in [0,1]$ small enough, we use controls of the form
	\begin{equation} \label{eq:w3qb10-u3}
		u_3(t) = a^4 \chi(t) \left(1 + \phi((T-t)/a) \right).
	\end{equation}
	In particular, $u_i(0) = u_i(T) = 0$ for $i = 1,2$.
	Moreover, the map $(a,\chi) \to u$ is continuous from $[0,1] \times C^3([0,T];\R)$ to $L^\infty((0,T);\R)$ and 
	\begin{equation}
		\| u \|_{L^\infty} = O \left( a \| \chi \|_{C^3} \right).
	\end{equation}
	Differentiating \eqref{eq:w3qb10-u3} twice yields
	\begin{equation}
		u_1(t) = a^4 \chi''(t) \left(1 + \phi((T-t)/a) \right) 
		- 2 a^3 \chi'(t) \phi' ((T-t)/a)
		+ a^2 \chi(t) \phi''((T-t)/a).
	\end{equation}
	Hence, heuristically, $u_1 \approx a^2 \chi \phi''$.
	Substituting in \eqref{eq:syst-W3-Qb100-x1x5} and using a Riemann--Lebesgue-type argument (see \cref{p:RL} in \cref{s:RL} below) yields
	\begin{align}
		x_3(T) & = a^4 \left( G_3(\chi) + F_3(a,\chi) \right), && G_3(\chi) := \chi(T) + \langle (\phi'')^2 \rangle \int_0^T \chi^2 \\
		x_4(T) & = a^4 \left( G_4(\chi) + F_4(a,\chi) \right), && G_4(\chi) := \int_0^T \chi + \langle (\phi'')^2 \rangle \int_0^T (T-t) \chi^2(t) \dd t
	\end{align}
	where $F_3$ and $F_4$ are $C^1$ maps on $[0,1] \times C^5_c((0,T];\R)$ with $F_3(0,\cdot) = F_4(0,\cdot) = 0$.
	Similar arguments prove that
	\begin{equation}
		\begin{split}
		x_5(T) & = a^8 \left( G_5(\chi) + F_5(a,\chi) \right)
		+ 2 x_4(T) \int_0^T u_1^2, \\
		G_5(\chi) & := \langle (1+\phi)^2 \rangle \int_0^T \chi^2 -  \langle (\phi'')^2 \rangle^2 \int_0^T \left(\int_0^t \chi^2 \right)^2 \dd t
		\end{split}
	\end{equation}
	where $F_5$ is a $C^1$ map on $[0,1] \times C^5_c((0,T];\R)$ with $F_5(0,\cdot) = 0$.
	
	\medskip
	
	We now prove that there exists $\chi^*$ such that $G_3(\chi^*) = G_4(\chi^*) = 0$, $G_5(\chi^*) < 0$, and $\chi \mapsto (G_3(\chi),G_4(\chi))$ is locally onto.
	More precisely, given $0 < \tau \ll 1$, let $\bar\chi \in C^\infty_c((0,T);[0,1])$ with $\bar \chi \equiv 1$ on $[2\tau,T-2\tau]$, $\supp \chi \subset [\tau,T-\tau]$, and $\bar h \in C^\infty_c([0,1];[0,1])$ with $\bar h'$ compactly supported in $(0,1)$ and $\bar h (0) = 1$.
	We look for $\chi$ under the form
	\begin{equation}
		\chi(t) = c_0 \bar \chi(t) + c_1 \bar h((T-t)/\tau)
	\end{equation}
	Let $\Lambda := \langle (\phi'')^2 \rangle$.
	Using the assumptions on $\phi$,
	\begin{align}
		\notag
		G_3(\chi) & = c_1 + \Lambda c_0^2 T + O(\tau), \\
		G_4(\chi) & = c_0 T + \Lambda c_0^2 \frac{T^2}{2} + O(\tau), \\
		\notag
		G_5(\chi) & = \frac{5}{4} c_0^2 T - \Lambda^2 c_0^4 \frac{T^3}{3} + O(\tau).
	\end{align}
	At $\tau = 0$, one checks that, for $\bar{c}_0 = - 2/(\Lambda T)$ and $\bar{c}_1 = -\Lambda \bar{c}_0 T$, $G_3 = G_4 = 0$ and $G_5 < 0$.
	By the implicit function theorem, there exists $(c_0^*,c_1^*,\tau^*)$ with $\tau^* > 0$, $c_0^* \approx \bar{c}_0$ and $c_1^* \approx \bar{c}_1$ such that $G_3(\chi^*) = G_4(\chi^*) = 0$ and $G_5(\chi^*) < 0$.
	Moreover, since $\partial_{c_1} (G_3,G_4)_{\rvert \bar c_0,\bar c_1,\tau} = (1,0) + O(\tau)$ and $\partial_{c_0} (G_3,G_4)_{\rvert \bar c_0,\bar c_1,\tau} = (0,-T)$, the $C^1$ map $(c_0,c_1) \mapsto (G_3,G_4)(\chi_{c_0,c_1,\tau^*})$ vanishes at $(c_0^*,c_1^*)$ and its differential is onto.
 
	\medskip
	
	By the implicit function theorem, there thus exists a $C^1$ map $a \mapsto \chi_a$ from $[0,1]$ to $C^5_c((0,T];\R)$ such that $G_3(\chi_a) + F_3(a,\chi_a) = G_4(\chi_a) + F_4(a,\chi_a) = 0$, with $G_5(\chi_0) < 0$.
	Thus $x(T) = a^8 (G_5(\chi_a) + F_5(a,\chi_a)) e_5 = a^8 G_5(\chi_0) e_5 + O(a^9)$, so one can move in the direction $-e_5$.
	
	\step{Conclusion} 
    Standard arguments using either tangent vectors or power series expansions (see e.g.\ \cite[Appendix]{zbMATH04154295} or \cite[Section~8.1]{zbMATH05150528}) entail that \eqref{Ex:W3vsQb10} is $L^\infty$-STLC.
\end{proof}

\paragraph{Limiting example for $Q^\flat_{1,1}$.}
Consider the system
\begin{equation} \label{Ex:W3vsQb11}
    \begin{cases}
        \dot{x}_1=u \\
        \dot{x}_2=x_1 + x_1^2 \\
        \dot{x}_3=x_2 \\
        \dot{x}_4=x_3 \\
        \dot{x}_5=x_4 \\
        \dot{x}_6=x_3^2 - 2x_1^2 x_5.
    \end{cases}
\end{equation}
Written in the form \eqref{syst}, this system satisfies
\begin{equation}
    \begin{split}
        & f_{M_{i-1}}(0)=e_i \text{ for } i \in \intset{1,5}, \quad 
        f_{W_{1,\nu}}(0) = 2 e_{2+\nu} \text{ for } \nu \in \intset{0,3}, \\
        & f_{Q^\flat_{1,1}}(0) = - 8e_6, \quad
        f_{W_3}(0)=2 e_6
    \end{split}
\end{equation}
and $f_b(0) = 0$ for any other $b \in \Bs$.

\begin{lemma}
    System \eqref{Ex:W3vsQb11} is $L^\infty$-STLC.
\end{lemma}

\begin{proof}
    Let us perform the (nonlinear) static-state-feedback transformation $v(t) := u(t)(1+2x_1(t))$.
    Then system \eqref{Ex:W3vsQb11} is mapped to a system which satisfies
    \begin{equation}
        \begin{split}
            & f_{M_{i-1}}(0)=e_i \text{ for } i \in \intset{1,5}, \quad 
            f_{W_3}(0)=2 e_6, \quad 
            f_{P_{1,5}}=-4e_6, \\
            & f_{(M_4,\ad^k_{X_1}(X_0))}(0) = c_k e_6 \text{ for } k \geq 3
        \end{split}
    \end{equation}
    for some constant $c_k \in \R$ and $f_b(0) = 0$ for every other $b \in \Bs$.
    In particular, as \eqref{Ex:W3vsP1l>4} with $l = 5$, it mainly features a competition between $W_3$ and $P_{1,5,0}$, which was proved to be $L^\infty$-STLC in \cref{Subsec:Ex_W3} using \cref{thm:AG}.
    Let us check that we can indeed ignore the added brackets.
    For $k \geq 3$, using the set $\Pi^1$ of \eqref{DefPi1_W3}, one has
    \begin{equation}
        (M_4,\ad^k_{X_1}(X_0)) = ((\ad_{M_2}(X_0), X_0), \ad_{X_1}^k(X_0)) \in \Pi^3.
    \end{equation}
    Thus $\omega((M_4,\ad^k_{X_1}(X_0))) = |(M_4,\ad^k_{X_1}(X_0))| - 3 = 3 + k \geq 6 > 5 = \omega(P_{1,5})$.
    Hence, the system after feedback is $L^\infty$-STLC, so \eqref{Ex:W3vsQb11} is $L^\infty$-STLC.
\end{proof}

\paragraph{Limiting example for $Q^\flat_{1,2}$.}
Consider the system
\begin{equation} \label{Ex:W3vsQb12}
    \begin{cases}
        \dot{x}_1=u + x_1^2 \\
        \dot{x}_2=x_1  \\
        \dot{x}_3=x_2 \\
        \dot{x}_4=x_3 \\
        \dot{x}_5=x_4 \\
        \dot{x}_6= x_5 \\
        \dot{x}_7= x_3^2 + 2x_1^2 x_6.
    \end{cases}
\end{equation}
Unlike the previous examples, this system is not nilpotent.
Nevertheless, written in the form \eqref{syst}, it satisfies
\begin{equation}
    \begin{split}
    &f_{W_3}(0)=2 e_7, \quad f_{Q^\flat_{1,2}}(0) = - 8e_7, \\
    &\forall b \in \mathcal{N}_3 \setminus \{ Q^\flat_{1,2} \},  \quad 
    f_b(0) \in \R e_1 + \dotsb + \R e_6.
    \end{split}
\end{equation}
Hence, along the line $e_7$, there is only a competition between $W_3$, $Q^\flat_{1,2}$, and brackets outside of $\mathcal{N}_3$, so that yield negligible contributions with respect to $\int u_3^2$.
It particular, it satisfies
\begin{equation}
    f_{W_3}(0) \notin (\mathcal{N}_3 \setminus \{ Q^\flat_{1,2} \})(f)(0).
\end{equation}

\begin{lemma}
    System \eqref{Ex:W3vsQb12} is $L^\infty$-STLC.
\end{lemma}

\begin{proof}
    Let us perform the (nonlinear) static-state-feedback transformation $v(t) := u(t) + x_1^2(t)$.
    Then system \eqref{Ex:W3vsQb12} is mapped to example \eqref{Ex:W3vsP1l>4} with $l = 6$, featuring a competition between $W_3$ and $P_{1,6,0}$, which was proved to be $L^\infty$-STLC in \cref{Subsec:Ex_W3}.
    Hence \eqref{Ex:W3vsQb12} is $L^\infty$-STLC.
\end{proof}

\subsection{A Riemann--Lebesgue-type lemma}
\label{s:RL}

\begin{lemma} \label{p:RL}
	Let $T > 0$ and $\mathcal{V} := \{ h \in C^3([0,T];\R) ; h(0)=h'(0)=h''(0)=0 \}$ endowed with its usual topology.
	Let $\theta \in C^0(\R;\R)$ be a $T$-periodic function.
	The map
	\begin{equation} \label{eq:Feta-def}
		F : \begin{cases}
			(0,1] \times \mathcal{V} &\to \R, \\
			(\tau,h) &\mapsto \int_0^T h(t) \theta((T-t)/\tau) \dd t
		\end{cases}
	\end{equation}
	admits a $C^1$ extension to $[0,1] \times \mathcal{V}$ which satisfies $F(0,h) = (1/T) \int_0^T \theta \int_0^T h$ for every $h \in \mathcal{V}$.
\end{lemma}

\begin{proof}
	Let $\theta_0 := \theta$.
	For $k \in \N$, we set $\theta_{k+1}(t) := (t/T) \int_0^T \theta_k - \int_0^t \theta_k$.
	In particular, for all $k \in \N$, $\theta_{k+1}(0) = 0$, $\theta_{k+1}$ is $T$-periodic and $\theta_{k+1}' = \langle \theta_k \rangle - \theta_k$.
	Integrating by parts in \eqref{eq:Feta-def} yields
	\begin{equation}
		\begin{split}
			F(\tau,h) 
			& = \langle \theta \rangle \int_0^T h + \tau [ h \theta_1((T-t)/\tau) ]_0^T- \tau \int_0^T h'(t) \theta_1((T-t)/\tau) \dd t \\
			& = \langle \theta \rangle \int_0^T h - \tau \int_0^T h'(t) \theta_1((T-t)/\tau) \dd t.
		\end{split}
	\end{equation}
	Hence $F$ admits a continuous extension setting $F(0,h) = \langle \theta \rangle \int_0^T h$.
	Repeating the process twice,
	\begin{equation}
		F(\tau,h) 
		= \langle \theta \rangle \int_0^T h 
		- \tau \langle \theta_1 \rangle h(T)
		+ \tau^2 \langle \theta_2 \rangle h'(T)
		- \tau^3 \int_0^T h'''(t) \theta_3((T-t)/\tau) \dd t.
	\end{equation}
	This development with respect to $\tau$ can be used to prove that $F$ enjoys $C^1$ regularity up to $\tau = 0$, since the $\tau^3$ factor allows to absorb the derivative with respect to $\tau$ of the last integral as $\tau \to 0$.
\end{proof}


\section*{Acknowledgments}

The authors wish to express their gratitude to Matthias Kawski for writing the enlightening survey~\cite{zbMATH04154295}, which inspired many of the results of this work, as recalled in the various examples given for each obstruction.
We also thank three referees who provided many suggestions to improve the overall presentation of the results and method.

The authors acknowledge support from grants ANR-20-CE40-0009 and ANR-11-LABX-0020, as well as from the Fondation Simone et Cino Del Duca -- Institut de France.

Data sharing not applicable to this article as no datasets were generated or analysed during the current study.

\bibliographystyle{plain}
\bibliography{biblio}

\end{document}